\documentclass[a4paper, 11pt]{amsart}    
\usepackage{latexsym, amsmath, amsthm, amssymb, setspace, verbatim}
\usepackage[all]{xy}
\usepackage[utf8]{inputenc} \usepackage[T1]{fontenc} \usepackage{lmodern}
\usepackage{hyperref}
\usepackage{enumitem}

\title[Classification of Rokhlin flows]{The classification of Rokhlin flows on \cstar-algebras}
\author{Gábor Szabó}
\address{Department of Mathematics, KU Leuven, Celestijnenlaan 200b, box 2400\linebreak
\phantom{-}\hspace{2mm} B-3001 Leuven, Belgium}
\email{gabor.szabo@kuleuven.be}

\numberwithin{equation}{section}

\begin{document}

\newcommand\set[1]{\left\{#1\right\}}  

\newcommand{\IA}[0]{\mathbb{A}} \newcommand{\IB}[0]{\mathbb{B}}
\newcommand{\IC}[0]{\mathbb{C}} \newcommand{\ID}[0]{\mathbb{D}}
\newcommand{\IE}[0]{\mathbb{E}} \newcommand{\IF}[0]{\mathbb{F}}
\newcommand{\IG}[0]{\mathbb{G}} \newcommand{\IH}[0]{\mathbb{H}}
\newcommand{\II}[0]{\mathbb{I}} \renewcommand{\IJ}[0]{\mathbb{J}}
\newcommand{\IK}[0]{\mathbb{K}} \newcommand{\IL}[0]{\mathbb{L}}
\newcommand{\IM}[0]{\mathbb{M}} \newcommand{\IN}[0]{\mathbb{N}}
\newcommand{\IO}[0]{\mathbb{O}} \newcommand{\IP}[0]{\mathbb{P}}
\newcommand{\IQ}[0]{\mathbb{Q}} \newcommand{\IR}[0]{\mathbb{R}}
\newcommand{\IS}[0]{\mathbb{S}} \newcommand{\IT}[0]{\mathbb{T}}
\newcommand{\IU}[0]{\mathbb{U}} \newcommand{\IV}[0]{\mathbb{V}}
\newcommand{\IW}[0]{\mathbb{W}} \newcommand{\IX}[0]{\mathbb{X}}
\newcommand{\IY}[0]{\mathbb{Y}} \newcommand{\IZ}[0]{\mathbb{Z}}

\newcommand{\CA}[0]{\mathcal{A}} \newcommand{\CB}[0]{\mathcal{B}}
\newcommand{\CC}[0]{\mathcal{C}} \newcommand{\CD}[0]{\mathcal{D}}
\newcommand{\CE}[0]{\mathcal{E}} \newcommand{\CF}[0]{\mathcal{F}}
\newcommand{\CG}[0]{\mathcal{G}} \newcommand{\CH}[0]{\mathcal{H}}
\newcommand{\CI}[0]{\mathcal{I}} \newcommand{\CJ}[0]{\mathcal{J}}
\newcommand{\CK}[0]{\mathcal{K}} \newcommand{\CL}[0]{\mathcal{L}}
\newcommand{\CM}[0]{\mathcal{M}} \newcommand{\CN}[0]{\mathcal{N}}
\newcommand{\CO}[0]{\mathcal{O}} \newcommand{\CP}[0]{\mathcal{P}}
\newcommand{\CQ}[0]{\mathcal{Q}} \newcommand{\CR}[0]{\mathcal{R}}
\newcommand{\CS}[0]{\mathcal{S}} \newcommand{\CT}[0]{\mathcal{T}}
\newcommand{\CU}[0]{\mathcal{U}} \newcommand{\CV}[0]{\mathcal{V}}
\newcommand{\CW}[0]{\mathcal{W}} \newcommand{\CX}[0]{\mathcal{X}}
\newcommand{\CY}[0]{\mathcal{Y}} \newcommand{\CZ}[0]{\mathcal{Z}}

\newcommand{\FA}[0]{\mathfrak{A}} \newcommand{\FB}[0]{\mathfrak{B}}
\newcommand{\FC}[0]{\mathfrak{C}} \newcommand{\FD}[0]{\mathfrak{D}}
\newcommand{\FE}[0]{\mathfrak{E}} \newcommand{\FF}[0]{\mathfrak{F}}
\newcommand{\FG}[0]{\mathfrak{G}} \newcommand{\FH}[0]{\mathfrak{H}}
\newcommand{\FI}[0]{\mathfrak{I}} \newcommand{\FJ}[0]{\mathfrak{J}}
\newcommand{\FK}[0]{\mathfrak{K}} \newcommand{\FL}[0]{\mathfrak{L}}
\newcommand{\FM}[0]{\mathfrak{M}} \newcommand{\FN}[0]{\mathfrak{N}}
\newcommand{\FO}[0]{\mathfrak{O}} \newcommand{\FP}[0]{\mathfrak{P}}
\newcommand{\FQ}[0]{\mathfrak{Q}} \newcommand{\FR}[0]{\mathfrak{R}}
\newcommand{\FS}[0]{\mathfrak{S}} \newcommand{\FT}[0]{\mathfrak{T}}
\newcommand{\FU}[0]{\mathfrak{U}} \newcommand{\FV}[0]{\mathfrak{V}}
\newcommand{\FW}[0]{\mathfrak{W}} \newcommand{\FX}[0]{\mathfrak{X}}
\newcommand{\FY}[0]{\mathfrak{Y}} \newcommand{\FZ}[0]{\mathfrak{Z}}

\renewcommand{\phi}[0]{\varphi}
\newcommand{\eps}[0]{\varepsilon}

\newcommand{\id}[0]{\operatorname{id}}		
\renewcommand{\sp}[0]{\operatorname{Sp}}		
\newcommand{\eins}[0]{\mathbf{1}}			
\newcommand{\diag}[0]{\operatorname{diag}}
\newcommand{\ad}[0]{\operatorname{Ad}}
\newcommand{\ev}[0]{\operatorname{ev}}
\newcommand{\fin}[0]{{\subset\!\!\!\subset}}
\newcommand{\Aut}[0]{\operatorname{Aut}}
\newcommand{\dimrok}[0]{\dim_{\mathrm{Rok}}}
\newcommand{\dst}[0]{\displaystyle}
\newcommand{\cstar}[0]{\ensuremath{\mathrm{C}^*}}
\newcommand{\dist}[0]{\operatorname{dist}}
\newcommand{\ann}[0]{\operatorname{Ann}}
\newcommand{\cc}[0]{\simeq_{\mathrm{cc}}}
\newcommand{\scc}[0]{\simeq_{\mathrm{scc}}}
\newcommand{\vscc}[0]{\simeq_{\mathrm{vscc}}}
\newcommand{\scd}[0]{\preceq_{\mathrm{scd}}}
\newcommand{\ue}[0]{{~\approx_{\mathrm{u}}}~}
\newcommand{\cel}[0]{\ensuremath{\mathrm{cel}}}
\newcommand{\acel}[0]{\ensuremath{\mathrm{acel}}}
\newcommand{\sacel}[0]{\ensuremath{\mathrm{sacel}}}
\newcommand{\prim}[0]{\ensuremath{\mathrm{Prim}}}
\newcommand{\co}[0]{\ensuremath{\mathrm{co}}}
\newcommand{\GL}[0]{\operatorname{GL}}
\newcommand{\Bott}[0]{\ensuremath{\mathrm{Bott}}}
\newcommand{\tK}[0]{\ensuremath{\underline{K}}}
\newcommand{\Hom}[0]{\operatorname{Hom}}
\newcommand{\wil}[0]{\ensuremath{\mathrm{wil}}}

\newtheorem{theorem}{Theorem}[section]	
\newtheorem{cor}[theorem]{Corollary}
\newtheorem{lemma}[theorem]{Lemma}
\newtheorem{prop}[theorem]{Proposition}

\newcounter{theoremintro}
\newtheorem{theoremi}[theoremintro]{Theorem}
\renewcommand*{\thetheoremintro}{\Alph{theoremintro}}

\theoremstyle{definition}
\newtheorem{defi}[theorem]{Definition}
\newtheorem*{defie}{Definition}
\newtheorem{nota}[theorem]{Notation}
\newtheorem*{notae}{Notation}
\newtheorem{rem}[theorem]{Remark}
\newtheorem*{reme}{Remark}
\newtheorem{example}[theorem]{Example}
\newtheorem*{examplee}{Example}
\newtheorem{conjecture}[theorem]{Conjecture}
\newtheorem*{conjecturee}{Conjecture}
\newtheorem{question}[theorem]{Question}

\begin{abstract} 
We study flows on \cstar-algebras with the Rokhlin property.
We show that every Kirchberg algebra carries a unique Rokhlin flow up to cocycle conjugacy, which confirms a long-standing conjecture of Kishimoto.
We moreover present a classification theory for Rokhlin flows on \cstar-algebras satisfying certain technical properties, which hold for many \cstar-algebras covered by the Elliott program.
As a consequence, we obtain the following further classification theorems for Rokhlin flows.
Firstly, we extend the statement of Kishimoto's conjecture to the non-simple case:\ Up to cocycle conjugacy, a Rokhlin flow on a separable, nuclear, $\CO_\infty$-absorbing \cstar-algebra is uniquely determined by its induced action on the prime ideal space. 
Secondly, we give a complete classification of Rokhlin flows on simple classifiable $KK$-contractible \cstar-algebras:\ Two Rokhlin flows on such a \cstar-algebra are cocycle conjugate if and only if their induced actions on the cone of lower-semicontinuous traces are affinely conjugate. 
\end{abstract}

\maketitle

\setcounter{tocdepth}{1}

\tableofcontents


\section*{Introduction}

The study of dynamical systems encompasses many diverse parts of mathematics.
It is generally agreed that the fundamental original motivation is to gain an understanding of mechanics of nature, in particular how physical systems change in time.
By exploring the algebraic properties of dynamics in the classical sense, one arrives at a notion of a dynamical system --- an action of a locally compact group on a locally compact space --- that provides a common framework for studying time evolutions and symmetry groups in classical mechanics.
As soon as quantum mechanics emerged in the early 20th century as a fundamental field of physics, important mathematical ideas came with it via its interpretation of observables as certain operators over a Hilbert space whose elements represent physical states; the mathematical field of operator algebras was born shortly afterwards.
By seeking the parallel framework for dynamics in the Hilbert space model, one arrives at the notion of a \cstar-dynamical system, which brings the classical and non-classical ideas together in one neat package. This has led to the general acceptance of \cstar-dynamical systems as an important field of research.
The connection between dynamical systems and operator algebras has remained close and intimate throughout their respective history. Dynamical ideas represent pivotal ingredients within major branches of operator algebras, such as in Connes' noncommutative geometry \cite{Connes94} or in Popa's deformation/rigidity theory \cite{Popa07}. 

While there is much to say about the relevance of discrete group actions on operator algebras --- see \cite{Izumi10} for a survey --- the most exciting applications in geometry or physics are typically related to time evolutions, or briefly \emph{flows}, which are continuous actions of the real numbers $\IR$. 
Compared to actions of other groups, flows on operator algebras are historically the most relevant for gaining insight into the structure theory of general operator algebras through the crossed product construction.
This is especially prominent in the theory of von Neumann algebras, which are also often called $\mathrm{W}^*$-algebras.
The Kubo--Martin--Schwinger condition for flows --- originally conceived in the context of mathematical physics \cite{Kubo57, MartinSchwinger59, HaagHugenholtzWinnink67} --- has influenced the invention of Tomita--Takesaki theory \cite{Takesaki70}, which in the form of the \emph{flow of weights} construction allows one to carry over structural properties back and forth between von Neumann algebras of type II and type III. 

On the \cstar-algebraic side, the subject of flows could be considered a hot topic in the 1970s and 80s, with highlights in the theory of derivations \cite{Kadison66, Sakai66, Elliott70, Elliott77} perhaps most notably due to Sakai \cite{Sakai}, and Bratteli--Robinson's approach \cite{BratteliRobinson_I, BratteliRobinson_II} to quantum statistical mechanics.
A natural but very ambitious problem is to determine when two flows on a \cstar-algebra are cocycle conjugate. 
It is fair to say that over the years, the progress on this general problem has remained rather modest, with only a couple results applicable in some very special cases.
A lot of this is due to the key distinguishing feature of \cstar-algebras, which is based on the idea of being close everywhere in some parameter space instead of just locally close as for $\mathrm{W}^*$-algebras.
This fundamental theme weighs particularly heavy in the context of dynamical systems and requires challengingly tight maneuvering when making approximations.

This is where Kishimoto's invaluable contributions on flows enters the picture, which can arguably be regarded as a branch of the subject initiated by Sakai.
Apart from his many papers dedicated to the fine structure of general flows on \cstar-algebras, one of Kishimoto's finest inventions in this context is the Rokhlin property; see \cite{Kishimoto96_R}.

\begin{defie}
Let $\alpha: \IR\curvearrowright A$ be a flow on a separable, unital \cstar-algebra.\footnote{For convenience we restrict to the unital case for the moment; see Definition \ref{def:Rokhlin-property} for the general formal definition.}
We say that $\alpha$ has the Rokhlin property, if for every $p>0$, there exists an approximately central sequence of unitaries $u_n$ in $A$ satisfying \[
\lim_{n\to\infty}\ \max_{|t|\leq 1}\ \| \alpha_t(u_n)-e^{ipt}u_n \| = 0.
\]
\end{defie}

The basic idea behind all kinds of Rokhlin-type properties for group actions on \cstar-algebras is to find an exhaustive set of finite (or compact) approximate representations of the acting group $G$ in the \cstar-algebra $A$, which are at the same time approximately central.
Using functional calculus, one can indeed see that the unitaries $u_n$ in the above definition correspond to a sequence of approximately central and equivariant $*$-homomorphisms of $\CC(\IR/\frac{2\pi}{p}\IZ)$ equipped with the left $\IR$-shift into $A$.
Thus for small constants $p>0$, the chosen unitaries represent large circles on which a large part of the reals is represented via a cyclic shift.

Although the Rokhlin property is somewhat restrictive --- such flows cannot exist on AF algebras, and they do not admit any KMS states --- it does appear in important examples \cite{Kishimoto96_R, Kishimoto98II, Kishimoto02, Kishimoto05, BratteliKishimotoRobinson07}.
A particular consequence of \cite{BratteliKishimotoRobinson07} is that every Kirchberg algebra carries at least one Rokhlin flow.
A higher-rank version of the Rokhlin property for flows, called finite Rokhlin dimension, was recently introduced in \cite{HirshbergSzaboWinterWu17}.
There it has been shown to be both relevant for the classifiability of the crossed product, and to be automatic for many classical dynamical systems, namely those arising from free flows on finite-dimensional locally compact metric spaces.

Kishimoto's early insight in \cite{Kishimoto96_R} was that, under a mild additional assumption on the cocycles, the Rokhlin property leads to an approximate vanishing of the first cohomology of the flow.
Given the importance of approximate cohomology vanishing arguments in all existing approaches to classify group actions up to cocycle conjugacy, this led him to suspect that flows with the Rokhlin property ought to be the ones accessible to classification; see the introduction of \cite{BratteliKishimotoRobinson07}.

Given that flows contain only little inherent $K$-theoretical information and that the behavior of Kirchberg algebras is governed by $K$-theory in the spirit of Kirchberg--Phillips classification \cite{KirchbergC, Phillips00}, Kishimoto conjectured that there exists precisely one Rokhlin flow on every Kirchberg algebra up to cocycle conjugacy; see the discussion after \cite[Theorem 1.3]{Kishimoto03_R} or the end of the introduction to \cite{Kishimoto02}. 
A suitable uniqueness theorem is also lacking for Rokhlin flows on simple A$\IT$ algebras of real rank zero, in particular on irrational rotation algebras, where Kishimoto has given various ways of constructing Rokhlin flows; cf.\ \cite[Section 7]{Kishimoto03}.
The closest thing to a uniqueness theorem so far has been accomplished by Kishimoto and Bratteli--Kishimoto--Robinson, who verified that any two Rokhlin flows on Cuntz algebras \cite{Cuntz77} are cocycle conjugate as long as they are quasi-free; see \cite[Corollary 5.11]{BratteliKishimotoRobinson07}. To my knowledge, no new insight has come after this result so far.

Before delving deeper on the \cstar-algebraic side, it is worthwhile to take a look back into the theory of von Neumann algebras.
Based on Kishimoto's ideas for \cstar-algebras and building on some preliminary work of Kawamuro \cite{Kawamuro00}, Masuda--Tomatsu \cite{MasudaTomatsu16} have recently introduced the Rokhlin property for flows on von Neumann algebras.
With it, they gave a complete classification of Rokhlin flows on a von Neumann algebra with separable predual. In short, two Rokhlin flows turn out to be cocycle conjugate as soon as they are pointwise approximately unitarily equivalent; see \cite[Theorem 1]{MasudaTomatsu16}.
Among other applications, they showcased how this yields an independent approach to the Connes--Haagerup--Krieger classification \cite{Connes73, Connes76, Haagerup87, Krieger76} of injective type III factors.
One of the reasons why Masuda--Tomatsu's achievement is impressive is that their classification of Rokhlin flows is in a sense complementary to other typical uniqueness results for flows, such as those arising in Tomita--Takesaki theory \cite{Connes73}.

The goal of this paper is to jump-start the classification theory for flows on \cstar-algebras by providing a satisfactory analog of Masuda--Tomatsu's classification theorem.
Of course one would not expect to cover all separable \cstar-algebras in the naive way, simply because these might have too complicated structure in general, but at least one would hope to cover enough ground to make substantial progress on some of the problems left open in Kishimoto's earlier work.

Masuda--Tomatsu's approach is technically very involved, but follows a by now fairly well-known strategy for classifying group actions:\ establishing an approximate second-cohomology vanishing, an approximate first-cohomology vanishing, and then combining these with an Evans--Kishimoto \cite{EvansKishimoto97} intertwining argument.
On the \cstar-algebraic side, the approximate first-cohomology vanishing is readily available as one of Kishimoto's first observations about Rokhlin flows, as well as Evans--Kishimoto intertwining being a \cstar-algebraic invention to begin with.
However, an approximate second-cohomology vanishing is so far completely missing from the theory.
It is due to the more flexibel nature of the weak operator topology that this technical obstacle can be overcome for von Neumann algebras, and even a close inspection of Masuda--Tomatsu's proof does not reveal how to carry their ideas over to the \cstar-algebraic context.
It is useful to note, however, that the approximate second-cohomology vanishing is just used as a technical prerequisite to show that two given flows are approximate cocycle perturbations of each other in the sense of \cite[Section 4]{Kishimoto02}.

In what may be considered the main novelty of this paper, a direct method is given to compare two flows as approximate cocycle perturbations of one another in the second section, which foregoes second-cohomology considerations altogether. 
The key difference to comparable techniques in the literature is that the cocycles are directly constructed as coboundaries, whose existence would often be an \emph{a posteriori} outcome together with the approximate vanishing of the first cohomology.
This approach works under the assumption that the underlying \cstar-algebra satisfies a property we call \emph{finite weak inner length} --- see Definition \ref{def:weak-inner-length} --- which is a property much weaker than having finite exponential length \cite{Ringrose92}. 
In particular, this method applies to many examples of classifiable \cstar-algebras, in fact in some important cases even when the exponential length is infinite. 
The basic idea of the proof draws inspiration from a similar idea from \cite{BarlakSzaboVoigt17} for classifying Rokhlin actions of compact (quantum) groups.
The crucial addition is a continuous implementation of Berg's technique \cite{Berg75} that compensates for the non-compactness of the reals $\IR$.
This is made possible by the aforementioned technical assumption on the \cstar-algebra, and bears some resemblence to Kishimoto's key argument in the proof of \cite[Theorem 2.1]{Kishimoto96_R}.

Building on Kishimoto's previous work mentioned so far, we then obtain a positive solution to his conjecture as our first main result:

\begin{theoremi} \label{Theorem-A}
Let $A$ be a unital Kirchberg algebra. Then up to cocycle conjugacy, there exists a unique flow on $A$ with the Rokhlin property.
\end{theoremi}

In the third section, we proceed to develop a general classification theorem à la Masuda--Tomatsu from scratch.
For this purpose, a general approximate first-cohomology vanishing is needed that respects approximate centrality. 
This works for \cstar-algebras satisfying a technical property that we call \emph{finite (weak) approximately central exponential length}; see Definition \ref{def:acel}. 
Although our proof is based on Kishimoto's earlier ingenious proof from \cite{Kishimoto96_R}, our treatment is more general and in fact contains some new aspects for the non-unital case.
For \cstar-algebras satisfying both of the two aforementioned technical properties, we then obtain a \cstar-algebraic analog of the Masuda--Tomatsu classification of Rokhlin flows; see Theorem \ref{thm:main-result}.

The notion of approximately central exponential length has implicitly appeared before in various places without a name, most recently in Enders' solution \cite{Enders15_2} of Blackadar's conjecture on the semiprojectivity of Kirchberg algebras.
It is interesting to note that, although this does not play an obvious or explicit role in Masuda--Tomatsu's proof of their classification theorem, it is observed in their appendix that all von Neumann algebras satisfy a stronger variant of this property, essentially due to Borel functional calculus; see \cite[Lemma 9.4]{MasudaTomatsu16}.
From the point of view of the present work, this yields a good explanation in hindsight why Masuda--Tomatsu's classification is possible in general, as well as a good reason to single out the approximately central exponential length of a \cstar-algebra as a natural notion. 
 
In short, the classification result (Theorem \ref{thm:main-result}) says that two Rokhlin flows $\alpha,\beta: \IR\curvearrowright A$ are cocycle conjugate when the embeddings $A \to \CC([0,1],A)$ coming from their orbit maps are approximately unitarily equivalent. 
To be more precise, we obtain an actual \emph{if and only if} condition when we ask for the cocycle conjugacy to be implemented by an approximately inner automorphism on $A$, which is the case in Masuda--Tomatsu's approach as well.
Compared to just characterizing cocycle conjugacy, this has the advantage that it allows one to obtain \emph{strong} classification theorems in applications, meaning that a conjugacy between two flows on a nice invariant actually lifts to a cocycle conjugacy on the \cstar-algebra.

Akin to Masuda--Tomatsu's classification, we in particular require that the automorphisms of the flows are approximately unitarily equivalent to each other, but moreover witnessed in a uniform way.
The relevance of this kind of comparison between flows was hinted at in \cite[Theorem 9.5]{MasudaTomatsu16} and credited to Izumi, albeit not fully exploited until now. 
The uniformity assumption is the minimum to be expected in the \cstar-algebra situation compared to Masuda--Tomatsu's pointwise assumption for $\mathrm{W}^*$-algebras. 
In fact there are examples showcasing that the pointwise assumption is not good enough in general; see Remark \ref{ex:uniform-assumption}.
One might say that this conceptual difference is analogous to uniform convergence on continuous functions over a space versus pointwise convergence on Borel functions.

The rest of the paper is then concerned with applying this general classification result for Rokhlin flows on certain classifiable classes of \cstar-algebras.
In the fourth section, we treat the purely infinite and possibly non-simple case. 
We obtain the following generalization of Theorem \ref{Theorem-A}, which is our second main result:

\begin{samepage}
\begin{theoremi} \label{Theorem-B}
Let $A$ be a separable, nuclear, $\CO_\infty$-absorbing \cstar-algebra.
Let $\alpha$ and $\beta$ be two flows on $A$ with the Rokhlin property.
Suppose that $\alpha$ and $\beta$ induce the same action on the prime ideal space of $A$.
Then $\alpha$ and $\beta$ are cocycle conjugate via an approximately inner automorphism.
\end{theoremi}
\end{samepage}

The proof of this result involves two steps.
Firstly, we combine arguments of Haagerup--R{\o}rdam \cite{HaagerupRordam95}, Nakamura \cite{Nakamura00} and Phillips \cite{Phillips97, Phillips02} to prove that any tensor product with $\CO_\infty$ has finite approximately central exponential length; see Theorem \ref{thm:fin-acel-Oinf}.
This may be viewed as an extension of Phillips' theorem from \cite{Phillips02} that such \cstar-algebras have finite exponential length, and is arguably of independent interest.
Secondly, we establish a uniqueness theorem for $*$-homomorphisms into $\CO_\infty$-absorbing \cstar-algebras with respect to ideal-preserving homotopy, which in particular applies to restricted coactions of flows.
We note that at least for stable \cstar-algebras, this would arise as a direct corollary of suitable uniqueness theorems from \cite{KirchbergC, Kirchberg00, Gabe19} with respect to ideal-related $KK$-theory.
However, it is unclear how to remove the stability assumption with any simple reduction argument.
Instead we give an independent elementary treatment, which hinges on the (comparably easier) $\CO_2$-stable classification theory; see \cite{Gabe18}.
Combining these two steps, Theorem \ref{Theorem-B} arises as a consequence of our general classification theorem from the third section.

In the fifth section, we treat the case of simple classifiable $KK$-contractible \cstar-algebras; see \cite{ElliottGongLinNiu17} for their recent classification theory due to Elliott--Gong--Lin--Niu.
As our third main result, we obtain the following classification of Rokhlin flows by the tracial invariant:

\begin{samepage}
\begin{theoremi} \label{Theorem-C}
Let $A$ be a separable, simple, stably projectionless \cstar-algebra with finite nuclear dimension and $KK(A,A)=0$. Then two Rokhlin flows $\alpha,\beta: \IR\curvearrowright A$ are cocycle conjugate if and only if $\alpha$ and $\beta$ induce affinely conjugate $\IR$-actions on the extended traces of $A$.\footnote{This means that there exists an affine homeomorphism $\gamma$ on the topological cone $T(A)$ of lower-semicontinuous extended traces respecting the scale $\Sigma_A$ such that $\gamma(\tau\circ\alpha_t)=\gamma(\tau)\circ\beta_t$ for all $\tau\in T(A)$ and $t\in\IR$.} Moreover, any conjugacy on the level of traces lifts to an automorphism on $A$ inducing cocycle conjugacy between $\alpha$ and $\beta$. 
\end{theoremi}
\end{samepage}

The proof of this result is underpinned by certain observations of Gong--Lin about approximately central unitaries in such \cstar-algebras.
Namely, it turns out that an approximately central unitary can always be connected to the unit via an approximately central unitary path.
Although one does not have any kind of control over the length of such a path in general, our main observation is that the path in Gong--Lin's construction approximately behaves like a path with length at most $2\pi$, at least when viewed in the strict topology.
In particular, the \cstar-algebras in the above class will often have infinite exponential length as well as infinite approximately central exponential length, but finite \emph{weak} approximately central exponential length, which is good enough for our purposes.
Theorem \ref{Theorem-C} will thus arise as another consequence of Theorem \ref{thm:main-result}.

We note that a similar observation as above was made by Nawata \cite{Nawata19} for \cstar-algebras isomorphic to either $\CW$ or $\CW\otimes\CK$.
To be more precise, he showed that the central sequence algebra of $\CW$ has a connected unitary group and exponential length at most $2\pi$.
As a byproduct of our efforts in the fifth section, the same follows for all \cstar-algebras in the classifiable $KK$-contractible class; see Theorem \ref{thm:nawata-property}.
By the arguments detailed in \cite[Section 7]{Nawata19}, we thus obtain a classification of single automorphisms with the Rokhlin property analogous to Theorem \ref{Theorem-C}; see Theorem \ref{thm:Rokhlin-automorphisms}.

Three main questions for future research, which we shall briefly outline, are left open. 

Firstly, one may observe (using Lin's work on basic homotopy lemmas \cite{Lin10}) that classifiable TAF \cstar-algebras have finite approximately central exponential length, and thus fall within the scope of the main classication theorem (\ref{thm:main-result}) of this paper. 
However, unlike in the situation covered by our main results, the classification of $*$-homomorphisms $A\to\CC([0,1],A)$ up to approximate unitary equivalence is necessarily then more complicated, as the Elliott invariant alone is not fine enough for this purpose. 
For example, it has to involve the rotation map; see \cite[Subsection 5.7]{Kishimoto03} and Remark \ref{ex:uniform-assumption}.
Thus we are tempted to ask:\ are Rokhlin flows on TAF \cstar-algebras classified by their rotation maps?

Secondly, how far can our approach to classifying Rokhlin flows be pushed into the realm of other classifiable \cstar-algebras that are either finite unital \cite{GongLinNiu15, ElliottGongLinNiu15, TikuisisWhiteWinter17} or stably projectionless \cite{GongLin17} with non-trivial $K$-theory? 

Thirdly, under what reasonable assumptions can one expect the Rokhlin property to hold automatically? 
Is it enough to assume that the flow acts sufficiently non-trivially on the classification invariant in the style of \cite{Shimada16}?
Does it hold for trace-scaling flows \cite{KishimotoKumjian96, KishimotoKumjian97}?

These questions shall be pursued in subsequent work.
\bigskip

\textbf{Acknowledgements.} I have been partially supported by the following sources, either while carrying out the research, or while writing or revising this manuscript:\
 EPSRC grant EP/N00874X/1; the Danish National Research Foundation through the \emph{Centre for Symmetry and Deformation} (DNRF92); the European Union's Horizon 2020 research and innovation programme under the grants MSCA-IF-2016-746272-SCCD and MSCA-RISE-2015-691246-QUANTUM DYNAMICS; a start-up grant of KU Leuven and an internal grant by the Research Council of KU Leuven.

I would like to thank the following colleagues for valuable discussions, remarks or other interactions that have benefited this paper in one way or another:\ Akitaka Kishimoto, Marius Dadarlat, Norio Nawata, James Gabe, Huaxin Lin, Masaki Izumi, Sel{\c c}uk Barlak, Guihua Gong, George Elliott, David Kerr, Aaron Tikuisis, and Hannes Thiel.


\section{Preliminaries}

\begin{notae}
Given a \cstar-algebra $A$, we denote its multiplier algebra by $\CM(A)$ and its proper unitization by $A^+$, i.e., we add a unit even if $A$ is unital.
If $A$ is unital, we denote by $\CU(A)$ the unitary group of $A$. 
In general, let us write $\eins+A$ for the obvious subset of $A^+$ and $\CU(\eins+A)=(\eins+A)\cap\CU(A^+)$ for the obvious unitary subgroup. 
Note that, if $A$ is unital, one can identify $\CU(\eins+A)\cong\CU(A)$, although these sets are not equal.
We will write $\CU_0(A)$ (in the unital case) or $\CU_0(\eins+A)$ for the path connected component of the unit. 
If we are given a continuous path of unitaries $w: [0,1]\to \CU(\eins+A)$ and $s\in [0,1]$, we use $w(s)$ or $w_s$ interchangably, whichever is more convenient. 

We sometimes denote by $A_{\leq 1}$ the closed unit ball of $A$.
For two elements $x$ and $y$ in some \cstar-algebra, we frequently write $x=_\eps y$ as short-hand for $\|x-y\|\leq\eps$.
We often write $\CF\fin M$ to mean that $\CF$ is a finite subset of some set $M$.
For two $*$-homomorphisms $\phi_1, \phi_2: A\to B$ between \cstar-algebras, we write $\phi_1\ue\phi_2$ to say that they are approximately unitarily equivalent, i.e., there exists a net of unitaries $u_\lambda\in\CU(\eins+B)$ with $\lim_{\lambda\to\infty} u_\lambda\phi_1(x)u_\lambda^* = \phi_2(x)$ for all $x\in A$.
For a given \cstar-algebra $A$, we will denote by $T(A)$ the topological cone of lower-semicontinuous, extended traces on $A$; see \cite{ElliottRobertSantiago11}.
\end{notae}

\begin{defi}
A flow on a \cstar-algebra $A$ is a group homomorphism $\alpha: \IR\to\Aut(A)$ with the property that $[t\mapsto\alpha_t(a)]$ is a norm-continuous map for every $a\in A$. One writes $\alpha: \IR\curvearrowright A$.
\end{defi}

\begin{defi}
\label{def:scc}
Let $A$ be a \cstar-algebra and $\alpha: \IR\curvearrowright A$ a flow. 
\begin{enumerate}[label=(\roman*),leftmargin=*] 
\item A strictly continuous map $w: \IR\to\CU(\CM(A))$ is called an $\alpha$-cocycle, if one has $w_s\alpha_s(w_t)=w_{s+t}$ for all $s,t\in\IR$.
In this case, the map $\alpha^w: \IR\to\Aut(A)$ given by $\alpha_t^w=\ad(w_t)\circ\alpha_t$ is again a flow, and is called a cocycle perturbation of $\alpha$. Two flows on $A$ are called exterior equivalent if one of them is a cocycle perturbation of the other.
\item Let $w$ be an $\alpha$-cocycle. It is called an approximate coboundary, if there exists a sequence of unitaries $x_n\in\CU(\CM(A))$ such that 
\[
\max_{|t|\leq 1}~\|(x_n\alpha_t(x_n^*) - w_t)a\|\stackrel{n\to\infty}{\longrightarrow} 0
\]
for all $a\in A$. 
If one even has 
\[
\max_{|t|\leq 1}~\|x_n\alpha_t(x_n^*) - w_t\|\stackrel{n\to\infty}{\longrightarrow} 0
\]
for some sequence $x_n\in\CU(\eins+A)$, then $w$ is called a norm-approximate coboundary. (In this case $w$ must be norm-continuous and have values in $\CU(\eins+A)$.) 
\item Let $\beta: \IR\curvearrowright B$ be another flow. The flows $\alpha$ and $\beta$ are called cocycle conjugate, 
if there exists an isomorphism $\psi: A\to B$ such that the flows $\psi^{-1}\circ\beta\circ\psi$ and $\alpha$ are exterior equivalent. 
If $\psi$ can be chosen such that $\psi^{-1}\circ\beta\circ\psi$ and $\alpha$ are exterior equivalent via a (norm-)approximate coboundary, then $\alpha$ and $\beta$ are called (norm-)strongly cocycle conjugate.
\end{enumerate}
\end{defi}

\begin{reme}
The notion of norm-strong cocycle conjugacy only differs from the definition of strong cocycle conjugacy in the non-unital case.
The original notion of strong cocycle conjugacy originates in Izumi--Matui's work on discrete group actions, but was arguably a mix of these two notions; cf.\ \cite[Definition 2.1]{IzumiMatui10}. 
\end{reme}

\begin{defi}
Let $(X,d)$ be a metric space. Let $v: [0,1]\to X$ be a continuous path.
One defines its length as
\[
\ell(v) = \sup_{\CP}~ \sum_{j=0}^{n-1} d\big( v(t_{j+1}), v(t_j) \big),
\]
where $\CP$ denotes the family of all sets $\set{t_0,t_1,\dots,t_n}\subset [0,1]$ with
\[
0=t_0 < t_1 < \dots < t_{n-1} < t_n=1.
\]
If $\ell(v)<\infty$, then $v$ is called rectifiable.
\end{defi}

\begin{defi}[see {\cite[Definition 2.1]{Phillips02}}] \label{def:exp-length}
Let $A$ be a \cstar-algebra. For a given unitary $u\in\CU_0(\eins+A)$, we define
\[
\cel(u) = \inf\set{\ell(v) \mid v: [0,1]\to\CU(\eins+A) \text{ continuous}, ~ v(0)=\eins,~ v(1)=u}.
\]
The exponential length of $A$ is defined as
\[
\cel(A) = \sup\set{ \cel(u) \mid u\in\CU_0(\eins+A)} \ \in \ [0,\infty].
\]
\end{defi}

\begin{rem} \label{rem:ringrose}
Originally, the notion of exponential length for a unital \cstar-algebra originates in Ringrose's work \cite{Ringrose92}. 
In any (unital) \cstar-algebra $A$, unitaries in the connected component of the unit are precisely the finite products of exponentials $\exp(ih_1)\cdots\exp(ih_m)$ for some self-adjoint elements $h_1,\dots, h_m\in A$. 
For any $j$, the unitary path of the form $[t\mapsto \exp(ith_j)]$ is analytic and $\|h_j\|$-Lipschitz after applying the composition formula for its derivative. 
Therefore we see that the length of any unitary $u\in\CU_0(A)$ is finite and bounded above by the sum $\|h_1\|+\dots+\|h_m\|$, for an arbitrary product decomposition of $u$ into exponentials. 
In fact, Ringrose proved that one has the equality
\[
\cel(u)=\inf\Big\{ \sum_{j=1}^m \|h_j\|  \mid  h_j=h_j^*\in A,\ u=\exp(ih_1)\cdots\exp(ih_m) \Big\}.
\]
\end{rem}

\begin{rem}
Let $(X, d)$ be a metric space and $v: [0,1]\to X$ a continuous path with $\ell(v)<\infty$. Then we may obtain its arc-length parametrization $v^0: [0,1]\to X$ by the following well-known method:

The function $\sigma: [0,1]\to [0,1]$ given by $\sigma(t) = \frac{\ell(v|_{[0,t]})}{\ell(v)}$ is endpoint-preserving, monotone, and continuous. 
If $\sigma(s)=\sigma(t)$ for $s<t$, then $\ell(v|_{[s,t]})=\ell(v|_{[0,t]})-\ell(v|_{[0,s]})=0$, and therefore $v$ is constant on $[s,t]$. 
Although $\sigma^{-1}$ is not a map, we see that $v^0=v\circ\sigma^{-1}$ yields a well-defined continuous map with $v^0(0)=v(0)$ and $v^0(1)=v(1)$.
By the definition of $\sigma$, the arc-length parametrization $v^0$ then satisfies the Lipschitz condition
\[
\|v^0(s)-v^0(t)\| \leq \ell(v)\cdot|s-t| \quad\text{for all } s,t\in [0,1].
\]
This is because, if $s<t$ and $\sigma(s_0)=s$ and $\sigma(t_0)=t$ for some $s_0<t_0$, then
\[
\begin{array}{ccl}
\|v^0(s)-v^0(t)\| &=& \|v(s_0)-v(t_0)\| \\
&\leq& \ell(v|_{[s_0,t_0]}) \\
&=& \ell(v|_{[0,t_0]})-\ell(v|_{[0,s_0]}) \\
&=& \ell(v)\cdot (t-s).
\end{array}
\]
Thus, this parametrization ensures that the path $v^0$ satisfies the Lipschitz condition with respect to its length.

In this paper, many arguments will involve choosing some unitary path with length bounded by some given constant. 
In all these instances, we will implicitly assume that the Lipschitz condition holds with respect to its length, by passing to its arc-length parametrization if necessary.
\end{rem}

The following notions are all well-known; cf.\ \cite[1.1]{Kirchberg04}, \cite[Section 1]{HirshbergSzaboWinterWu17} and \cite[Section 4]{Szabo18ssa2}.

\begin{defi} \label{def:central-sequence}
Let $A$ be a \cstar-algebra and $\alpha: \IR\curvearrowright A$ a flow. 
\begin{enumerate}[label={\textup{(\roman*)}},leftmargin=*]
\item The sequence algebra of $A$ is given as 
\[
A_\infty = \ell^\infty(\IN,A)/\set{ (x_n)_n \mid \lim_{n\to\infty}\| x_n\|=0}.
\]
There is a standard embedding of $A$ into $A_\infty$ by sending an element to its constant sequence. We shall always identify $A\subset A_\infty$ this way, unless specified otherwise.
\item Pointwise application of $\alpha$ on representing sequences defines a (not necessarily continuous) $\IR$-action on $\ell^\infty(\IN,A)$. Let $\ell^\infty_\alpha(\IN,A)$ be the \cstar-algebra consisting of those sequences $(x_n)_n$ where $[t\mapsto (\alpha_t(x_n))_n]$ is norm-continuous. Then
\[
A_{\infty,\alpha} = \ell^\infty_\alpha(\IN,A)/\set{ (x_n)_n \mid \lim_{n\to\infty}\| x_n\|=0}
\]
is the continuous part of $A_\infty$ with respect to $\alpha$.\footnote{Indeed $A_{\infty,\alpha}$ coincides with the \cstar-subalgebra of elements in $A_\infty$ on which the induced algebraic $\IR$-action is continuous; see \cite[Theorem 2]{Brown00}.}
We denote by $\alpha_\infty$ the natural flow induced by $\alpha$ on $A_{\infty,\alpha}$.
\item The central sequence algebra of $A$ is defined as the quotient
\[
F_\infty(A) = (A_\infty\cap A')/\ann(A,A_\infty),
\]
where
\[
A_\infty\cap A' = \set{x\in A_\infty \mid [x,A]=0}
\]
and 
\[
\ann(A,A_\infty)=\set{x\in A_\infty \mid xA=Ax=0}.
\]
\item The continuous central sequence algebra of $A$ with respect to $\alpha$ is similarly defined as
\[
F_{\infty,\alpha}(A) = (A_{\infty,\alpha}\cap A')/\ann(A,A_{\infty,\alpha}).
\]
We denote by $\tilde{\alpha}_\infty$ the natural flow induced by $\alpha$ on $F_{\infty,\alpha}(A)$.
\end{enumerate}
\end{defi}

The following definition is due to Kishimoto \cite{Kishimoto96_R} for unital \cstar-algebras; see also \cite[Section 2]{HirshbergSzaboWinterWu17}.

\begin{defi} \label{def:Rokhlin-property}
Let $A$ be a separable \cstar-algebra and $\alpha: \IR\curvearrowright A$ a flow. We say that $\alpha$ has the Rokhlin property, if for every $p>0$, there exists a unitary $u\in F_{\infty,\alpha}(A)$ satisfying $\tilde{\alpha}_{\infty,t}(u)=e^{ipt}u$ for all $t\in\IR$.
\end{defi}

We will often follow the convention to call a flow with the Rokhlin property simply a Rokhlin flow.

\begin{nota}
Given $T>0$, we denote by $\sigma^T: \IR\curvearrowright\CC(\IR/T\IZ)$ the $T$-periodic $\IR$-shift given by $\sigma^T_t(f)(s+T\IZ)=f(s-t+T\IZ)$ for all $s,t\in\IR$. 
If $A$ is a \cstar-algebra with some flow $\alpha: \IR\curvearrowright A$, then we have $\CC(\IR/T\IZ)\otimes A\cong\CC(\IR/T\IZ, A)$ naturally, and under this identification the tensor product flow is given by $(\sigma^T\otimes\alpha)_t(f)(s+T\IZ) = \alpha_t(f(s-t+T\IZ))$ for $s,t\in\IR$. 
\end{nota}

\begin{rem} \label{rem:Rp-perspective}
Identifying $\IR/T\IZ$ with the unit circle, and thus identifying $\CC(\IR/T\IZ)\cong\CC(\IT)$, the canonical unitary generator $z=\id_\IT\in\CC(\IT)$ satisfies the equation $\sigma^T_t(z)=e^{2\pi i t/T}\cdot z$ for all $t\in\IR$. 

Conversely, whenever $\beta: \IR\curvearrowright B$ is a flow on a unital \cstar-algebra and a unitary $u\in\CU(B)$ satisfies the condition $\alpha_t(u)=e^{2\pi i t/T} u$ for all $t\in\IR$, then the assignment $z\mapsto u$ yields an equivariant and unital $*$-homomorphism from $\big( \CC(\IR/T\IZ), \sigma^T \big)$ to $(B,\beta)$. 
This has been implicitely observed early on by Kishimoto in \cite[Proof of Theorem 2.1]{Kishimoto96_R}; see also \cite[Remarks 2.2 and 2.10]{HirshbergSzaboWinterWu17}.
\end{rem}

For many arguments related to the Rokhlin property, it is useful to exploit the perspective given by so-called sequentially split $*$-homomorphisms; see \cite{BarlakSzabo16} for details, although we will not need to refer to results about them beyond the next lemma. 
This leads to the following well-known characterization of the Rokhlin property:

\begin{lemma} \label{lem:seq-split-picture}
Let $A$ be a separable \cstar-algebra and $\alpha: \IR\curvearrowright A$ a flow. The following are equivalent:
\begin{enumerate}[label=\textup{(\roman*)}, leftmargin=*]
\item The flow $\alpha$ has the Rokhlin property; \label{lem:seq-split-picture:1}
\item For every $T>0$, there exists a unital and equivariant $*$-homomorphism from $( \CC(\IR/T\IZ), \sigma^T )$ to $\big( F_{\infty,\alpha}(A),\tilde{\alpha}_\infty \big)$; \label{lem:seq-split-picture:2}
\item For every $T>0$, there exists an equivariant $*$-homomorphism $\psi: \big( \CC(\IR/T\IZ)\otimes A, \sigma^T\otimes\alpha \big) \to ( A_{\infty,\alpha}, \alpha_\infty )$ such that $\psi(\eins\otimes a)=a$ for all $a\in A$. \label{lem:seq-split-picture:3}
\end{enumerate}
\end{lemma}
\begin{proof}
The equivalence \ref{lem:seq-split-picture:1}$\Leftrightarrow$\ref{lem:seq-split-picture:2} follows directly from Remark \ref{rem:Rp-perspective}. 
The equivalence \ref{lem:seq-split-picture:2}$\Leftrightarrow$\ref{lem:seq-split-picture:3} is a special case of \cite[Lemma 4.2]{BarlakSzabo16}.
\end{proof}

\begin{example}[see \cite{Kishimoto02, Kishimoto05, BratteliKishimotoRobinson07}] \label{ex:quasi-free}
Let $n\in\IN\cup\set{\infty}$ with $n\geq 2$. Let $\set{ p_j \mid 1\leq j\leq n}\subset\IR$ be a set of parameters.
Consider the Cuntz algebra of $n$ generators
\[
\CO_n = \begin{cases} \cstar\Big( s_1,\dots,s_n \mid s_j^*s_j=\eins=\sum_{j=1}^n s_js_j^* \Big) &,\quad n\neq\infty \\
\cstar\Big( s_1,\dots,s_n \mid s_j^*s_k=\delta_{j,k}\cdot\eins \Big) &,\quad n=\infty
\end{cases}
\]
and the so-called quasi-free flow
\[
\gamma: \IR\curvearrowright\CO_n \quad\text{via}\quad \gamma_t(s_j)=e^{2\pi i p_j t}s_j.
\]
If a finite subset of $\set{p_j}_{j\leq n}$ generates $\IR$ as a closed subsemigroup, then $\gamma$ has the Rokhlin property. 
Moreover, all Rokhlin flows arising in this fashion on $\CO_n$ are mutually cocycle conjugate.

In the case $n=\infty$, one may in particular choose $p_1=1$, $p_2=-\sqrt{2}$ and $p_j=0$ for $j\geq 3$.
Thus there exist Rokhlin flows on $\CO_\infty$, and due to the Kirchberg--Phillips absorption theorem \cite{KirchbergPhillips00}, there exist Rokhlin flows on every Kirchberg algebra $A$ because $A\cong A\otimes\CO_\infty$.
\end{example}

\begin{nota} \label{nota:restricted-coaction}
Let $\alpha: \IR\curvearrowright A$ be a flow on a \cstar-algebra. One may view $\alpha$ in the coaction picture as a $*$-homomorphism from $A$ to $\CC_b(\IR,A)$ via $a\mapsto [t\mapsto\alpha_t(a)]$. For $T>0$, we shall denote the resulting restriction on $[0,T]$ by $\alpha_\co^T: A\to\CC\big( [0,T], A \big)$, i.e., $\alpha_\co^T(a)(t)=\alpha_t(a)$ for all $t\in [0,T]$ and $a\in A$. For brevity, we will denote $\alpha_\co=\alpha_\co^1$ and refer to this $*$-homomorphism as the restricted coaction of $\alpha$.
\end{nota}

Since we will be interested in cases when the restricted coactions of two different flows are approximately unitarily equivalent, the following elementary lemmas will be useful, the proofs of which only involve elementary calculations.

\begin{lemma} \label{lem:res-co-ue}
Let $A$ be a separable \cstar-algebra and $\alpha, \beta:\IR\curvearrowright A$ two flows. The following are equivalent:
\begin{enumerate}[label=\textup{(\roman*)}]
\item $\alpha_\co\ue\beta_\co$; \label{lem:res-co-ue:1}
\item $\alpha_\co^T\ue\beta_\co^T$ for all $T>0$; \label{lem:res-co-ue:2}
\item $\alpha_\co^T\ue\beta_\co^T$ implemented by unitaries in $\eins+\CC_0\big( (0,T], A \big)$, for all $T>0$. \label{lem:res-co-ue:3}
\end{enumerate}
\end{lemma}
\begin{proof}
As the implications \ref{lem:res-co-ue:3}$\Rightarrow$\ref{lem:res-co-ue:2}$\Rightarrow$\ref{lem:res-co-ue:1} are trivial, we show \ref{lem:res-co-ue:1}$\Rightarrow$\ref{lem:res-co-ue:3}.

Assume that $\alpha_\co\ue\beta_\co$ holds. 
Since the claim is only relevant for arbitrarily large $T>0$, let us assume that $T\geq 1$ is a natural number. Let $\eps>0$ and $\CF\fin A$ be given. Consider the compact sets
\begin{equation} \label{eq:res-co-ue:F'F''}
\CF' = \bigcup_{0\leq s\leq T-1} \beta_s(\CF),\quad \CF'' = \bigcup_{0\leq s\leq 1} \alpha_s(\CF') \ \subset \ A.
\end{equation}
Let us choose a unitary $w'\in\CU(\eins+\CC([0,1], A))$ such that 
\[
\max_{a\in\CF''}~ \|w'\alpha_\co(a)w'^*-\beta_\co(a)\|\leq\eps/2T.
\] 
We may view this as a unitary path $w': [0,1]\to\CU(\eins+A)$, which then satisfies
\[
\max_{a\in\CF''}~ \max_{0\leq s\leq 1}~ \|w_s'\alpha_s(a)w_s'^*-\beta_s(a)\|\leq\eps/2T.
\]
By the definition of $\CF''$, this in particular implies
\[
\max_{a\in\CF'}~ \max_{0\leq s\leq 1}~ \|w_0'\alpha_s(a)w_0'^*-\alpha_s(a)\|\leq\eps/2T.
\]
Thus $w=w'w_0'^*: [0,1]\to\CU(\eins+A)$ yields a unitary in $\CU(\eins+\CC([0,1],A))$ with $w_0=\eins$, and with
\begin{equation} \label{eq:res-co-ue:w}
\max_{a\in\CF'}\ \|w\alpha_\co(a)w^*-\beta_\co(a)\| = \max_{a\in\CF'}\ \max_{0\leq s\leq 1}\ \|w_s\alpha_s(a)w_s^*-\beta_s(a)\|\leq\eps/T.
\end{equation}
Let us define $w^{(T)}: [0,T]\to\CU(\eins+A)$ inductively via $w^{(T)}_0=\eins$ and
\begin{equation} \label{eq:res-co-ue:wT}
w^{(T)}_{j+s}=w^{(T)}_j\alpha_{j}(w_s) ,\quad s\in [0,1],\quad j=0,\dots,T-1.
\end{equation}
As $w_0=\eins$, this is a well-defined unitary path and yields an element $w^{(T)}\in\eins+\CC_0\big( (0,T], A \big)$. 
Set
\begin{equation} \label{eq:res-co-ue:Fk}
\CF'_k = \bigcup_{0\leq s\leq T-1-k} \beta_s(\CF),\quad k=0,\dots,T-1. 
\end{equation}
Then one has
\begin{equation} \label{eq:res-co-ue:Fk-inclusion}
\CF' = \CF'_0 \supseteq \CF'_1 \supseteq \dots \supseteq \CF'_{T-1} = \CF.
\end{equation}
We shall show, inductively in $k\in\set{0,\dots, T-1}$, that
\begin{equation} \label{eq:res-co-ue:induction}
\max_{a\in\CF'_k}~ \max_{0\leq s\leq 1}~ \|w^{(T)}_{k+s}\alpha_{k+s}(a)w^{(T)*}_{k+s}-\beta_{k+s}(a)\| \leq k\eps/T.
\end{equation}
The case $k=0$ holds by assumption.
Suppose that this holds for some $k<T-1$. Then inserting $s=1$ in \eqref{eq:res-co-ue:induction} yields
\begin{equation} \label{eq:res-co-ue:special-induction}
\max_{a\in\CF'_k}~  \|w^{(T)}_{k+1}\alpha_{k+1}(a)w^{(T)*}_{k+1}-\beta_{k+1}(a)\| \leq k\eps/T.
\end{equation}
It follows for all $a\in\CF'_{k+1}$ and $s\in [0,1]$ that
\[
\renewcommand\arraystretch{1.5}
\begin{array}{cl}
\multicolumn{2}{l}{ \|w^{(T)}_{k+1+s}\alpha_{k+1+s}(a)w^{(T)*}_{k+1+s}-\beta_{k+1+s}(a)\| } \\
\stackrel{\eqref{eq:res-co-ue:wT}}{=}& \|w^{(T)}_{k+1} \alpha_{k+1}(w_s) \alpha_{k+1}(\alpha_s(a)) \alpha_{k+1}(w_s)^* w^{(T)*}_{k+1} - \beta_{k+1+s}(a)\| \\
=& \|w^{(T)}_{k+1} \alpha_{k+1}\big( w_s\alpha_s(a)w_s^* \big) w^{(T)*}_{k+1} - \beta_{k+1+s}(a)\| \\
\stackrel{\eqref{eq:res-co-ue:Fk-inclusion}, \eqref{eq:res-co-ue:w}}{\leq}& \eps/T + \|w^{(T)}_{k+1} \alpha_{k+1}\big( \beta_s(a) \big) w^{(T)*}_{k+1} - \beta_{k+1+s}(a)\| \\
\stackrel{\eqref{eq:res-co-ue:Fk}, \eqref{eq:res-co-ue:special-induction}}{\leq}& \eps/T+k\eps/T \ = \ (k+1)\eps/T.
\end{array}
\]
This finishes the induction.

We now compute that
\[
\renewcommand\arraystretch{1.5}
\begin{array}{cl}
\multicolumn{2}{l}{ \dst\max_{a\in\CF}~ \|w^{(T)}\alpha_\co^T(a)w^{(T)*}-\beta_\co^T(a)\| } \\
=& \dst\max_{a\in\CF}~ \max_{0\leq s\leq T}~ \|w^{(T)}_s\alpha_s(a)w^{(T)*}_s-\beta_s(a)\| \\
=& \dst\max_{a\in\CF}~ \max_{k\in\set{0,\dots,T-1}}~ \max_{0\leq s\leq 1}~ \|w^{(T)}_{k+s}\alpha_{k+s}(a)w^{(T)*}_{k+s}-\beta_{k+s}(a)\| \\
\stackrel{\eqref{eq:res-co-ue:Fk-inclusion}, \eqref{eq:res-co-ue:induction}}{\leq}& \dst\max_{a\in\CF}~ \max_{k\in\set{0,\dots,T-1}}~ k\eps/T \ \leq \ \eps.
\end{array}
\]
As $\CF\fin A$ and $\eps>0$ were arbitrary, this finishes the proof.
\end{proof}

\begin{lemma} \label{lem:cc-ue}
Let $A$ be a separable \cstar-algebra and $\alpha, \beta:\IR\curvearrowright A$ two flows. Suppose that $\alpha$ and $\beta$ are cocycle conjugate via an approximately inner automorphism of $A$. Then the restricted coactions $\alpha_\co$ and $\beta_\co$ are approximately unitarily equivalent.
\end{lemma}
\begin{proof}
By definition, it suffices to consider the two cases where $\alpha$ and $\beta$ are either conjugate by an approximately inner automorphism, or that they are exterior equivalent.

Let us first consider the case where $\alpha$ and $\beta$ are conjugate by an approximately inner automorphism $\phi$, i.e., $\beta_t=\phi^{-1}\circ\alpha_t\circ\phi$ for all $t\in\IR$. Let $\CF\fin A$ and $\eps>0$ be given. As $\phi$ is an approximately inner automorphism, we find a unitary $x\in\CU(\eins+A)$ such that
\[
\max_{a\in\CF}~ \|\phi(a)-\ad(x)(a)\|\leq\eps/2
\]
and
\[
\max_{a\in\CF}~\max_{0\leq t\leq 1}~ \|\phi^{-1}(\alpha_t(\phi(a)))-\ad(x^*)(\alpha_t(\phi(a)))\|\leq\eps/2.
\]
If we define $w\in\CU(\eins+\CC( [0,1], A ) )$ via $w_t=x^*\alpha_t(x)$, then this yields
\[
\renewcommand\arraystretch{1.5}
\begin{array}{cl}
\multicolumn{2}{l}{ \dst\max_{a\in\CF}~ \|w\alpha_\co(a)w^*-\beta_\co(a)\| }\\
=& \dst\max_{a\in\CF}~\max_{0\leq t\leq 1}~ \|x^*\alpha_t(x)\alpha_t(a)\alpha_t(x^*)x-\beta_t(a)\| \\
=& \dst\max_{a\in\CF}~\max_{0\leq t\leq 1}~ \|\big( \ad(x^*)\circ\alpha_t\circ\ad(x) \big)(a)-\beta_t(a)\| \\
\leq & \eps/2+\dst\max_{a\in\CF}~\max_{0\leq t\leq 1}~ \|\big( \ad(x^*)\circ\alpha_t\circ \phi \big)(a)-\beta_t(a)\| \\
\leq & \eps.
\end{array}
\]
This finishes the first case.

Now let us consider the case where $\alpha$ and $\beta$ are exterior equivalent. Let $w: \IR\to\CU(\CM(A))$ be an $\alpha$-cocycle with $\beta_t=\ad(w_t)\circ\alpha_t$ for all $t\in\IR$. Let $\CF\fin A$ and $\eps>0$ be given. By \cite[Theorem 1.1]{Kishimoto06}, $w$  can be strictly approximated, uniformly on $t\in [0,1]$, by norm-continuous cocycles in $\CU(\eins+A)$. In particular, we find some norm-continuous map $z: [0,1]\to\CU(\eins+A)$ such that
\[
\max_{a\in\CF}~\max_{0\leq t\leq 1}~ \|\ad(z_t)(\alpha_t(a))-\ad(w_t)(\alpha_t(a))\| \leq \eps.
\]
As $\ad(w_t)\circ\alpha_t=\beta_t$, the induced unitary $z\in\CU(\eins+\CC([0,1],A))$ satisfies
\[
\max_{a\in\CF}~ \|z\alpha_\co(a)z^*-\beta_\co(a)\|\leq\eps.
\]
As $\CF\fin A$ and $\eps>0$ were arbitrary, this finishes the proof.
\end{proof}


\section{Approximate cocycle perturbations}

In this section, we solve an intermediate technical problem that constitutes a crucial step in the classification of flows via the Evans--Kishimoto intertwining \cite{EvansKishimoto97} method.
In general, if we want to construct a cocycle conjugacy between two group actions this way, it is necessary that one of these actions is an approximation of cocycle perturbations of the other one, and vice versa.
In the particular case of flows, this observation is due to Kishimoto, and has been pointed out by him --- see \cite[Theorem 4.7]{Kishimoto02} --- as the important missing piece towards Theorem \ref{Theorem-A}.

\begin{defi}[see {\cite[Section 4]{Kishimoto02}}] \label{def:acp}
Let $A$ be a separable \cstar-algebra with two flows $\alpha, \beta: \IR\curvearrowright A$.
We say that $\alpha$ is an approximate cocycle perturbation of $\beta$, if there exists a sequence of $\beta$-cocycles $\{w_t^{(n)}\}_{t\in\IR}\subset\CU(\eins+A)$ such that
\[
\max_{|t|\leq 1}~ \|\alpha_t(a)-\ad(w^{(n)}_t)\circ\beta_t(a)\| \stackrel{n\to\infty}{\longrightarrow} 0
\]
for all $a\in A$.
\end{defi}

\begin{rem}
Let $A$ be a separable \cstar-algebra and $\alpha: \IR\curvearrowright A$ a flow. 
By a result \cite[Theorem 1.1]{Kishimoto06} of Kishimoto, a strictly continuous multiplier cocycle $\set{w_t}_t\subset\CU(\CM(A))$ can be strictly approximated, uniformly on $t\in [-1,1]$, by cocycles $\set{w_t'}_t\subset\CU(\eins+A)$. 
Thus Definition \ref{def:acp} is equivalent to its obvious generalization using cocycles in the multiplier algebra.
\end{rem}

\begin{defi} \label{def:weak-inner-length}
Let $A$ be a \cstar-algebra. 
We define the weak inner length $\wil(u)$ of a unitary $u\in\CU_0(\eins+A)$ as the smallest number $L>0$ for which the following holds:

For every finite set of contractions $\CF\fin A_{\leq 1}$ and $\eps>0$, there exists a unitary path $v: [0,1]\to\CU(\eins+A)$ with $v_0=\eins$, $v_1=u$, and
\[
\|v_{t_1}^*av_{t_1}-v_{t_2}^*av_{t_2}\| = \|[v_{t_2}v_{t_1}^*,a]\|\leq\eps+L|t_1-t_2|
\]
for all $t_1,t_2\in [0,1]$ and $a\in\CF$. If no such constant exists, then $\wil(u):=\infty$.

The weak inner length of $A$ is defined as
\[
\wil(A)=\sup \set{ \wil(u) \mid u\in\CU_0(\eins+A) } \ \in \ [0,\infty].
\]
\end{defi}

\begin{rem} \label{rem:celw-cel}
In some cases it might be useful to let $\CF$ be a compact set instead of a finite set, which one can do by an elementary compactness argument.

As one has the estimate $\|v_1^*av_1-v_2^*av_2\|\leq 2\|a\|\|v_1-v_2\|$ for all elements $a,v_1,v_2$ in a \cstar-algebra with $v_1,v_2$ being unitaries, it follows that 
\[
\wil(u)\leq 2\cel(u) \quad\text{for all } u\in\CU_0(\eins+A).
\]
In particular, the weak inner length of any \cstar-algebra is at most twice its exponential length in the sense of Definition \ref{def:exp-length}.

A look at the abelian \cstar-algebra $A=\CC(\IT)$ makes it clear that a \cstar-algebra can have infinite exponential length, but finite weak inner length. 
As we will see in the fifth section, this phenomenon can be observed for some simple \cstar-algebras as well, e.g., the classifiable $KK$-contractible \cstar-algebras treated in the last section tend to have infinite exponential length but all of them have finite weak inner length.
\end{rem}

The technique at the heart of the following lemmas is partly inspired by \cite[Lemma 5.5]{BarlakSzaboVoigt17}, which was a technical tool used for a similar purpose to classify Rokhlin actions of compact quantum groups.
One might also argue that its true origin is somewhere in the intersection between Izumi's work on finite group actions with the Rokhlin property --- see \cite[Lemma 3.3]{Izumi04} --- and Kishimoto's original work on Rokhlin flows \cite{Kishimoto96_R}.

\begin{lemma} \label{lem:acp-pre}
Let $A$ be a \cstar-algebra with $\wil(A)<\infty$. Let $\alpha, \beta: \IR\curvearrowright A$ be two flows with $\alpha_\co\ue\beta_\co$. Then for every $\eps>0$ and $\CF\fin A$, there exists $T>0$ and a unitary $z\in\CU(\eins+\CC(\IR/T\IZ)\otimes A)$ such that
\[
\max_{a\in\CF}~ \max_{|t|\leq 1}~ \| \eins\otimes\alpha_t(a)-\big( \ad(z)\circ(\sigma^T\otimes\beta)_t\circ\ad(z^*) \big)(\eins\otimes a) \|\leq\eps.
\]
\end{lemma}
\begin{proof}
Set $C:=\wil(A)$. Let $\eps>0$ and $\CF\fin A$ be given. Without loss of generality, let us assume that $\CF$ consists of contractions. 
We then choose $T>0$ big enough so that $\frac{5C}{T}<\eps$. 
Set
\begin{equation} \label{eq:acp-pre-F-prime}
\CF'=\bigcup_{|s|\leq T+1} \alpha_{-s}(\CF).
\end{equation}
As $\alpha_\co\ue\beta_\co$, we have $\alpha_\co^T\ue\beta_\co^T$, so by Lemma \ref{lem:res-co-ue} we may choose a unitary path $w: [0,T]\to\CU(\eins+A)$ with $w_0=\eins$ and
\begin{equation} \label{eq:acp-pre-w}
\max_{a\in\CF'}~ \max_{0\leq s\leq T}~ \|w_s\alpha_s(a)w_s^*-\beta_s(a)\| \leq \eps/5.
\end{equation}
Then the unitary $\beta_{-T}(w_T)$ clearly is in $\CU_0(\eins+A)$. By our choice of $C$ and $T$ (recall Definition \ref{def:weak-inner-length}), we may find a unitary path $\kappa: [0,T]\to\CU(\eins+A)$ with 
\begin{equation} \label{eq:acp-pre-kappa-endpoint}
\kappa_0=\eins,\quad \kappa_T=\beta_{-T}(w_T)
\end{equation}
and
\begin{equation} \label{eq:acp-pre-kappa}
\|\kappa_{s_1}^*a\kappa_{s_1}-\kappa_{s_2}^*a\kappa_{s_2}\| \leq \Big(\frac{\eps}{5}-\frac{C}{T}\Big)+\frac{C}{T}|s_1-s_2| \leq \frac{\eps}{5} 
\end{equation}
for all $s_1,s_2\in [0,T]$ with $|s_1-s_2|\leq 1$ and all $a\in\CF'\cup\beta_{-T}(\CF')$.

We then define the unitary path
\begin{equation} \label{eq:acp-pre-z}
z: [0,T]\to\CU(\eins+A) \quad\text{via}\quad z_s = w_s^*\beta_s(\kappa_s).
\end{equation}
By construction, we have $z_0=\eins=z_T$, and thus we may view $z$ as a unitary in $\CU(\eins+\CC(\IR/T\IZ)\otimes A)$ via $z(s+T\IZ)=z_s$ for $s\in [0,T]$.
We claim that this unitary satisfies the desired property.

Let us from now on fix a number $t\in [-1,1]$ and $a\in\CF$. Let $s\in [0,T]$ be given. If $s-t\in [0,T]$, then we compute
\[
\renewcommand\arraystretch{1.25}
\begin{array}{cl}
\multicolumn{2}{l}{ \Big( (\sigma^T\otimes\beta)_t(z^*) (\eins\otimes \beta_t(a)) (\sigma^T\otimes\beta)_t(z) \Big)(s) } \\
\stackrel{\eqref{eq:acp-pre-z}}{=}& \beta_t\big( \beta_{s-t}(\kappa_{s-t}^*) w_{s-t} \big) \beta_t(a) \beta_t\big( w_{s-t}^* \beta_{s-t}(\kappa_{s-t}) \big) \\
=& \beta_s(\kappa_{s-t}^*)\cdot \beta_t\big( w_{s-t} a w_{s-t}^* \big) \cdot \beta_s(\kappa_{s-t}) \\
\stackrel{\eqref{eq:acp-pre-F-prime}, \eqref{eq:acp-pre-w}}{=}_{\makebox[0pt]{\footnotesize \hspace{-8mm}$\eps/5$}}& \beta_s(\kappa_{s-t}^*) \cdot (\beta_t\circ\beta_{s-t}\circ\alpha_{t-s})(a) \cdot \beta_s(\kappa_{s-t}) \\
=& \beta_s\Big( \kappa_{s-t}^* \alpha_{t-s}(a) \kappa_{s-t} \Big) \\
\stackrel{\eqref{eq:acp-pre-kappa}}{=}_{\makebox[0pt]{\footnotesize $\eps/5$}}& \beta_s\Big( \kappa_{s}^* \alpha_{t-s}(a) \kappa_{s} \Big) .
\end{array}
\]
Hence
\[
\renewcommand\arraystretch{1.25}
\begin{array}{cl}
\multicolumn{2}{l}{ \Big( z(\sigma^T\otimes\beta)_t(z^*)  (\eins\otimes \beta_t(a))  (\sigma^T\otimes\beta)_t(z)z^* \Big)(s) } \\
\stackrel{\eqref{eq:acp-pre-z}}{=}_{\makebox[0pt]{\footnotesize \hspace{2mm}$2\eps/5$}} ~ & w_s^*\beta_s(\kappa_s)\beta_s\Big( \kappa_{s}^* \alpha_{t-s}(a) \kappa_{s} \Big) \beta_s(\kappa_s^*)w_s \\
=& w_s^*\beta_s\big( \alpha_{t-s}(a) \big)w_s \\
\stackrel{\eqref{eq:acp-pre-F-prime}, \eqref{eq:acp-pre-w}}{=}_{\makebox[0pt]{\footnotesize\hspace{-8mm}$\eps/5$}}& \alpha_s\big( \alpha_{t-s}(a) \big) \ = \ \alpha_t(a).
\end{array}
\]
On the other hand, if $s-t\notin [0,T]$, then either $s-t<0$ or $s-t>T$.

If $s<t$, then in particular $s\leq 1$ and $T-1 \leq T+s-t$, and we compute
\[
\renewcommand\arraystretch{1.25}
\begin{array}{cl}
\multicolumn{2}{l}{ \Big( (\sigma^T\otimes\beta)_t(z^*)  (\eins\otimes \beta_t(a))  (\sigma^T\otimes\beta)_t(z) \Big)(s) } \\
\stackrel{\eqref{eq:acp-pre-z}}{=}& \beta_t\big( \beta_{T+s-t}(\kappa_{T+s-t}^*) w_{T+s-t} \big) \beta_t(a) \beta_t\big( w_{T+s-t}^* \beta_{T+s-t}(\kappa_{T+s-t}) \big) \\
\stackrel{\eqref{eq:acp-pre-F-prime}, \eqref{eq:acp-pre-w}}{=}_{\makebox[0pt]{\footnotesize \hspace{-8mm}$\eps/5$}}& \beta_t\Big( \beta_{T+s-t}(\kappa_{T+s-t}^*) \beta_{T+s-t}(\alpha_{-(T+s-t)}(a)) \beta_{T+s-t}(\kappa_{T+s-t}) \Big) \\
=& \beta_{T+s}\Big( \kappa_{T+s-t}^* \cdot \alpha_{-(T+s-t)}(a) \cdot \kappa_{T+s-t} \Big) \\
\stackrel{\eqref{eq:acp-pre-kappa-endpoint}, \eqref{eq:acp-pre-kappa}}{=}_{\makebox[0pt]{\footnotesize\hspace{-8mm}$\eps/5$}}& \beta_{T+s}\Big( \beta_{-T}(w_T^*) \cdot \alpha_{-(T+s-t)}(a) \cdot \beta_{-T}(w_T) \Big) \\
=& \beta_s\Big( w_T^* \beta_T\big( \alpha_{-(T+s-t)}(a) \big) w_T \Big) \\
\stackrel{\eqref{eq:acp-pre-F-prime}, \eqref{eq:acp-pre-w}}{=}_{\makebox[0pt]{\footnotesize \hspace{-8mm}$\eps/5$}}& (\beta_s\circ\alpha_{t-s})(a).
\end{array}
\]
Hence
\[
\renewcommand\arraystretch{1.25}
\begin{array}{cl}
\multicolumn{2}{l}{ \Big( z (\sigma^T\otimes\beta)_t(z^*)  (\eins\otimes \beta_t(a))  (\sigma^T\otimes\beta)_t(z)z^* \Big)(s) } \\
\stackrel{\eqref{eq:acp-pre-z}}{=}_{\makebox[0pt]{\footnotesize \hspace{2mm}$3\eps/5$}}& w_s^*\beta_s(\kappa_s)\cdot (\beta_s\circ\alpha_{t-s})(a) \cdot \beta_s(\kappa_s^*) w_s \\
=& w_s^*\beta_s\Big( \kappa_s \alpha_{t-s}(a) \kappa_s^* \Big) w_s \\
\stackrel{\eqref{eq:acp-pre-kappa-endpoint}, \eqref{eq:acp-pre-kappa}}{=}_{\makebox[0pt]{\footnotesize \hspace{-8mm}$\eps/5$}}& w_s^* \beta_s(\alpha_{t-s}(a))w_s \\
\stackrel{\eqref{eq:acp-pre-F-prime}, \eqref{eq:acp-pre-w}}{=}_{\makebox[0pt]{\footnotesize \hspace{-8mm}$\eps/5$}}& \alpha_s(\alpha_{t-s}(a)) \ = \ \alpha_t(a).
\end{array}
\]
Lastly, if $s-t>T$, then in particular $0< s-t-T\leq 1$ and $s\geq T-1$, and we compute
\[
\renewcommand\arraystretch{1.25}
\begin{array}{cl}
\multicolumn{2}{l}{ \Big( (\sigma^T\otimes\beta)_t(z^*)  (\eins\otimes \beta_t(a))  (\sigma^T\otimes\beta)_t(z) \Big)(s) } \\
\stackrel{\eqref{eq:acp-pre-z}}{=}& \beta_t\big( \beta_{s-t-T}(\kappa_{s-t-T}^*) w_{s-t-T} \big) \beta_t(a) \beta_t\big( w_{s-t-T}^* \beta_{s-t-T}(\kappa_{s-t-T}) \big) \\
=& \beta_{s-T}(\kappa_{s-t-T}^*)\cdot \beta_t\big(  w_{s-t-T} a w_{s-t-T}^* \big) \beta_{s-T}(\kappa_{s-t-T}) \\
\stackrel{\eqref{eq:acp-pre-F-prime}, \eqref{eq:acp-pre-w}}{=}_{\makebox[0pt]{\footnotesize \hspace{-8mm}$\eps/5$}}& \beta_{s-T}(\kappa_{s-t-T}^*)\cdot (\beta_t\circ\beta_{s-t-T}\circ\alpha_{-(s-t-T)})(a) \cdot \beta_{s-T}(\kappa_{s-t-T}) \\
=& \beta_{s-T}\Big( \kappa_{s-t-T}^* \cdot \alpha_{-(s-t-T)}(a) \cdot \kappa_{s-t-T} \Big) \\
\stackrel{\eqref{eq:acp-pre-kappa-endpoint}, \eqref{eq:acp-pre-kappa}}{=}_{\makebox[0pt]{\footnotesize \hspace{-8mm}$\eps/5$}}& (\beta_{s-T}\circ\alpha_{-(s-t-T)})(a). 
\end{array}
\]
Hence
\[
\renewcommand\arraystretch{1.25}
\begin{array}{cl}
\multicolumn{2}{l}{ \Big( z (\sigma^T\otimes\beta)_t(z^*)  (\eins\otimes \beta_t(a))  (\sigma^T\otimes\beta)_t(z)z^* \Big)(s) } \\
\stackrel{\eqref{eq:acp-pre-z}}{=}_{\makebox[0pt]{\footnotesize \hspace{2mm}$2\eps/5$}}& w_s^*\beta_s(\kappa_s) \cdot (\beta_{s-T}\circ\alpha_{-(s-t-T)})(a) \cdot \beta_s(\kappa_s^*) w_s \\
=& w_s^*\beta_s\Big( \kappa_s \cdot (\beta_{-T}\circ\alpha_{-(s-t-T)})(a) \cdot \kappa_s^* \Big) w_s \\
\stackrel{\eqref{eq:acp-pre-kappa-endpoint}, \eqref{eq:acp-pre-kappa}}{=}_{\makebox[0pt]{\footnotesize \hspace{-8mm}$\eps/5$}}& w_s^*\beta_s\Big( \beta_{-T}(w_T) \cdot (\beta_{-T}\circ\alpha_{-(s-t-T)})(a) \cdot \beta_{-T}(w_T^*) \Big) w_s \\
=& w_s^*\beta_{s-T}\Big( w_T \cdot \alpha_{-(s-t-T)}(a) \cdot w_T^* \Big) w_s \\
\stackrel{\eqref{eq:acp-pre-F-prime}, \eqref{eq:acp-pre-w}}{=}_{\makebox[0pt]{\footnotesize \hspace{-8mm}$\eps/5$}}& w_s^*\beta_{s-T}\Big( \beta_T\circ\alpha_{t-s}(a) \Big) w_s \\
=& w_s^*\beta_s(\alpha_{t-s}(a))w_s \\
\stackrel{\eqref{eq:acp-pre-F-prime}, \eqref{eq:acp-pre-w}}{=}_{\makebox[0pt]{\footnotesize \hspace{-8mm}$\eps/5$}}& \alpha_s(\alpha_{t-s}(a)) \ = \ \alpha_t(a).
\end{array}
\]
Since $s\in [0,T]$ was arbitrary, this finishes the proof.
\end{proof}

\begin{lemma} \label{lem:acp}
Let $A$ be a separable \cstar-algebra with $\wil(A)<\infty$. Let $\alpha, \beta: \IR\curvearrowright A$ be two flows with $\alpha_\co\ue\beta_\co$. Suppose that $\beta$ has the Rokhlin property. Then for every $\eps>0$ and $\CF\fin A$, there exists a unitary $z\in\CU(\eins+A)$ such that
\[
\max_{a\in\CF}~ \max_{|t|\leq 1}~ \| \alpha_t(a)-\big( \ad(z)\circ\beta_t\circ\ad(z^*) \big)(a) \|\leq\eps.
\]
\end{lemma}
\begin{proof}
We apply Lemma \ref{lem:acp-pre} and find $T>0$ and a unitary $z_0\in\CU(\eins+\CC(\IR/T\IZ)\otimes A)$ such that
\begin{equation} \label{eq:acp-pre-2-1}
\max_{|t|\leq 1}~ \| \eins\otimes\alpha_t(a)-\big( \ad(z_0)\circ(\sigma^T\otimes\beta)_t\circ\ad(z_0^*) \big)(\eins\otimes a) \|\leq\eps/2
\end{equation}
for all $a\in\CF$. As $\beta$ has the Rokhlin property, we may apply Lemma \ref{lem:seq-split-picture} and find an equivariant $*$-homomorphism $\psi: \big( \CC(\IR/T\IZ)\otimes A, \sigma^T\otimes\beta \big) \to ( A_{\infty,\beta}, \beta_\infty )$ such that $\psi(\eins\otimes a)=a$ for all $a\in A$.
We may naturally extend $\psi$ to a unital $*$-homomorphism $\psi^+$ between the proper unitizations. Then \eqref{eq:acp-pre-2-1} implies
\begin{equation} \label{eq:acp-pre-2-2}
\max_{|t|\leq 1}~ \| \alpha_t(a)-\big( \ad(\psi^+(z_0))\circ\beta_{\infty,t}\circ\ad(\psi^+(z_0^*)) \big)(a) \|\leq\eps/2
\end{equation}
for all $a\in\CF$. We may find a sequence of unitaries $z_n\in\CU(\eins+A)$ representing $\psi^+(z_0)\in\CU(\eins+A_{\infty,\beta})$. Then for each $a\in\CF$, \eqref{eq:acp-pre-2-2} is equivalent to
\[
\limsup_{n\to\infty}~ \max_{|t|\leq 1}~ \| \alpha_t(a)-\big( \ad(z_n)\circ\beta_{t}\circ\ad(z_n^*) \big)(a) \|\leq\eps/2.
\]
In particular, we may choose $z=z_n$ for some sufficiently large $n$ to have the desired property.
\end{proof}

\begin{cor} \label{cor:acp}
Let $A$ be a separable \cstar-algebra with $\wil(A)<\infty$. Let $\alpha, \beta: \IR\curvearrowright A$ be two flows with $\alpha_\co\ue\beta_\co$. Suppose that $\beta$ has the Rokhlin property. Then $\alpha$ is an approximate cocycle perturbation of $\beta$.
\end{cor}

\begin{rem}
Notice that, in the situation of Corollary \ref{cor:acp}, the converse always holds, even without the Rokhlin property. That is, if a flow $\alpha$ is an approximate cocycle purturbation of another flow $\beta$, then it is easy to show that $\alpha_\co\ue\beta_\co$. This can be proved along the same lines as Lemma \ref{lem:cc-ue}.  
\end{rem}

Based on Kishimoto's earlier work, we are now already in the position to obtain Theorem \ref{Theorem-A} with the help of Corollary \ref{cor:acp}.

\begin{proof}[Proof of Theorem {\em\ref{Theorem-A}}]
Due to a result of Bratteli--Kishi\-moto--Robinson \cite{BratteliKishimotoRobinson07}, there exists a Rokhlin flow on every Kirchberg algebra; see Example \ref{ex:quasi-free}. So we have to show the uniqueness part.

Let $A$ be a unital Kirchberg algebra and $\alpha,\beta: \IR\curvearrowright A$ two flows with the Rokhlin property. By a result of Kishimoto \cite[Theorem 4.7]{Kishimoto02}, in order to show that $\alpha$ and $\beta$ are cocycle conjugate, it is enough to show that they are approximate cocycle perturbations of each other. It is well-known that Kirchberg algebras have finite exponential length, in fact one has $\cel(A)\leq 2\pi$; see \cite{Phillips02}. In particular, we have $\wil(A)\leq 4\pi<\infty$ by Remark \ref{rem:celw-cel}. 

As the two $*$-homomorphisms
\[
\alpha_\co,\ \beta_\co: A\to\CC\big( [0,1], A\big)
\]
are homotopic to the constant embedding, it follows from the uniqueness theorem in Kirchberg--Phillips classification \cite[Theorem 4.1.1]{Phillips00} that $\alpha_\co$ and $\beta_\co$ are approximately unitarily equivalent after composing them with the inclusion 
\[
\id\otimes\eins_{\CO_\infty}\otimes e: \CC\big( [0,1], A\big) \to \CC\big( [0,1], A\big)\otimes\CO_\infty\otimes\CK,
\] 
where $e\in\CK$ is some minimal projection. 
Since $A$ is unital, it follows that the sequence of unitaries implementing this equivalence approximately commute with $e$, so we can slightly perturb them to genuinely commute with $e$.
This allows us to conlude that $\alpha_\co\otimes\eins_{\CO_\infty}$ is approximately unitarily equivalent to $\beta_\co\otimes\eins_{\CO_\infty}$.
At the same time, $A$ is $\CO_\infty$-absorbing by \cite[Theorem 3.14]{KirchbergPhillips00}, so it is a straightforward consequence of \cite[Remark 2.7]{TomsWinter07} that in fact $\alpha_\co$ and $\beta_\co$ are approximately unitarily equivalent.

By Corollary \ref{cor:acp}, the flows $\alpha$ and $\beta$ are thus approximate cocycle perturbations of each other. This shows our claim.
\end{proof}


\section{The classification theorem}

In this section, we prove a general classification theorem for Rokhlin flows on a class of \cstar-algebras with some abstract properties detailed below. 
The main purpose of the later sections will be to verify these properties for some natural classes of \cstar-algebras featuring in Theorems \ref{Theorem-B} and \ref{Theorem-C}.

\begin{defi} \label{def:acel}
Let $A$ be a \cstar-algebra. 
\begin{enumerate}[label=\textup{(\roman*)}, leftmargin=*]
\item We define the approximately central exponential length of $A$, written $\acel(A)\in [0,\infty]$, as the smallest constant $C>0$ for which the following is true.

For every $\eps>0$ and $\CF\fin A_{\leq 1}$, there exist $\delta>0$ and $\CG\fin A_{\leq 1}$ with the following property. For every unitary path $u: [0,1]\to\CU(\eins+A)$ with $u(0)=\eins$ and
\[
\max_{a\in\CG}~ \max_{0\leq t\leq 1}~ \|[u(t),a]\|\leq\delta,
\]
there exists a unitary path $v: [0,1]\to\CU(\eins+A)$ with
\[
v(0)=\eins,\quad v(1)=u(1),\quad \ell(v)\leq C+\eps,
\]
and
\[
\max_{a\in\CF}~ \max_{0\leq t\leq 1}~ \|[v(t),a]\|\leq\eps.
\]
\label{def:acel-1}
\item We define the weak approximately central exponential length of $A$, written $\acel_w(A)\in [0,\infty]$, as the smallest constant $C>0$ for which the following is true.

For every $\eps>0$ and $\CF\fin A_{\leq 1}$, there exists $\delta>0$ and $\CG\fin A_{\leq 1}$ with the following property. For every unitary path $u: [0,1]\to\CU(\eins+A)$ with $u(0)=\eins$ and
\[
\max_{a\in\CG}~ \max_{0\leq t\leq 1}~ \|[u(t),a]\|\leq\delta,
\]
there exists a unitary path $v: [0,1]\to\CU(\eins+A)$ with
\[
v(0)=\eins,\quad v(1)=u(1),\quad \max_{a\in\CF}~ \max_{0\leq t\leq 1}~ \|[v(t),a]\|\leq\eps,
\]
and
\[
\max_{a\in\CF}~ \|a(v(t_1)-v(t_2))\|\leq\eps+C|t_1-t_2| 
\]
for all $t_1,t_2\in [0,1]$.
\label{def:acel-2}
\end{enumerate}
\end{defi}

\begin{rem}
If one compares the two concepts \ref{def:acel-1} and \ref{def:acel-2} above, one can observe that they always coincide for unital \cstar-algebras, simply by inserting $a=\eins$ in \ref{def:acel-2}.
In general, the difference can be summarized by saying that \ref{def:acel-1} cares about the length of unitary paths in the norm topology, whereas \ref{def:acel-2} only cares about the length of unitary paths in the natural semi-norms inducing the strict topology.
In some important classes of \cstar-algebras, such as the ones treated in the last section, one has a uniform bound on the length as in \ref{def:acel-2}, despite the fact that these \cstar-algebras have infinite exponential length in general.

An elementary compactness argument shows that, in Definition \ref{def:acel} and in other similar statements, one can replace the finite sets $\CF,\CG\fin A$ with compact subsets of $A$ instead.
We will henceforth make use of this fact without further mention. 
\end{rem}

\begin{example}[see {\cite{Phillips02}, \cite[Theorem 2.5]{Enders15_2} and \cite[Theorem 7]{Nakamura00}}]
The approximately central exponential length of any Kirchberg algebra is at most $6\pi$. If it satisfies the UCT, then it is at most $2\pi$.
As we will see in the next section, it is in fact always at most $2\pi$ for any tensor product with the Cuntz algebra $\CO_\infty$.
\end{example}

The following proof is closely related to an argument in Kishimoto's original paper \cite{Kishimoto96_R} introducing the Rokhlin property, where an approximate first-cohomology vanishing result was obtained. Several variants of this have appeared in his work since, the most relevant one being \cite[Proposition 3.4]{Kishimoto02}.

\begin{lemma} \label{lem:cv-pre}
Let $A$ be a \cstar-algebra with $\acel_w(A)<\infty$. Let $\alpha: \IR\curvearrowright A$ be a flow. Then for every $\eps>0$ and $\CF\fin A$, there exist $T>0$, $\delta>0$ and $\CG\fin A$ with the following property:

Suppose that $w: \IR\to\CU(\eins+A)$ is an $\alpha$-cocycle with
\[
\max_{a\in\CG}~ \max_{0\leq t\leq 1}~ \|[w_t,a]\|\leq\delta.
\]
Then there exists a unitary $v\in\CU(\eins+\CC(\IR/T\IZ)\otimes A))$ such that
\[
\max_{a\in\CF}~ \|[v,\eins\otimes a]\|\leq\eps
\]
and
\[
\max_{a\in\CF}~ \max_{|t|\leq 1}~ \|\eins\otimes a\cdot \big( \eins\otimes w_t-v(\sigma^T\otimes\alpha)_t(v^*) \big)\|\leq\eps.
\]
If moreover $\acel(A)<\infty$, then we can improve the last part of the conclusion to
\[
\max_{|t|\leq 1}~ \|\eins\otimes w_t-v(\sigma^T\otimes\alpha)_t(v^*)\|\leq\eps.
\] 
\end{lemma}

\begin{proof}
Let $\eps>0$ and $\CF\fin A$ be given. Without loss of generality, let us assume that $\CF$ consists of self-adjoint contractions.
Set $C:=\acel_w(A)$. 

We then choose $T\in\IN$ big enough so that $\frac{6C}{T}<\eps$.
Set
\begin{equation} \label{eq:cv-pre-delta-F'}
\CF' = \bigcup_{|s|\leq T+1} \alpha_{s}(\CF).
\end{equation}
Choose a pair $(\delta',\CG')$ for the pair $(\eps',\CF')$ according to Definition \ref{def:acel}\ref{def:acel-2}, where $\eps'=\frac{\eps}{6}-\frac{C}{T}$. Without loss of generality we may assume $\delta'\leq\eps/6$ and $\CG'\supseteq\CF$. 
Set
\begin{equation} \label{eq:cv-pre-delta-G}
\delta=\delta'/T \quad\text{and}\quad \CG = \bigcup_{|s|\leq T+1} \alpha_{s}(\CG').
\end{equation}
Now let $w: \IR\to\CU(\eins+A)$ be an $\alpha$-cocycle. Suppose that
\begin{equation} \label{eq:cv-pre-w}
\max_{a\in\CG}~ \max_{0\leq t\leq 1}~ \|[w_t,a]\|\leq\delta.
\end{equation}
Now let $a\in\CG'$ and $s\in (0,T]$. Choose a natural number $j<T$ with $j< s\leq j+1$. Set $t_s=s-j\in (0,1]$. By applying the cocycle identity, we have for every $r\in [0,T]$ that
\[
\renewcommand\arraystretch{1.5}
\begin{array}{ccl}
\|[w_s,\alpha_r(a)]\| &=&  \| \big[ w_1\alpha_1(w_1)\cdots\alpha_{j-1}(w_1)\cdot\alpha_{j}(w_{t_s}), \alpha_r(a) \big]\| \\
&\leq & \dst \|[\alpha_j(w_{t_s}), \alpha_r(a)]\| + j\cdot\max_{0\leq j_0<j}~\|[\alpha_{j_0}(w_1),\alpha_r(a)]\| \\
&\stackrel{\eqref{eq:cv-pre-w}}{\leq}& (j+1)\cdot\delta \ \leq \ T\delta \ \stackrel{\eqref{eq:cv-pre-delta-G}}{\leq} \ \delta'.
\end{array}
\] 
In particular, let us record
\begin{equation} \label{eq:cv-pre-w-2}
\max_{a\in\CG'}~ \max_{0\leq s\leq T}~ \|[w_s,a]\|\leq\delta'.
\end{equation}
As $a\in\CG'$, $s\in (0,T]$ and $r\in [0,T]$ were arbitrary in the above calculation, the unitary path $[0,T]\to\CU(\eins+A)$ given by $s\mapsto\alpha_{-s}(w_s)=w_{-s}^*$ starts at the unit and commutes componentwise with elements in $\CG'$ up to $\delta'$.
By the choice of the pair $(\delta',\CG')$ and our choice of $T$, we may thus find a unitary path $\kappa: [0,T]\to\CU(\eins+A)$ with 
\begin{equation} \label{eq:cv-pre-kappa-endpoint}
\kappa_0=\eins,\quad \kappa_T=\alpha_{-T}(w_T)=w_{-T}^*;
\end{equation}
\begin{equation} \label{eq:cv-pre-kappa-1}
\max_{a\in\CF'}~ \max_{0\leq s\leq T}~ \|[\kappa_s,a]\| \leq \frac{\eps}{6};
\end{equation}
\begin{equation} \label{eq:cv-pre-kappa-2}
\|a(\kappa_{s_1}-\kappa_{s_2})\| \leq \big( \frac{\eps}{6}-\frac{C}{T} \big)+\frac{C}{T}|s_1-s_2| \leq \frac{\eps}{6} 
\end{equation}
for all $s_1,s_2\in [0,T]$ with $|s_1-s_2|\leq 1$ and all $a\in\CF'$.

We then define the unitary path
\begin{equation} \label{eq:cv-pre-v}
v: [0,T]\to\CU(\eins+A) \quad\text{via}\quad v_s = w_s\alpha_s(\kappa_s^*).
\end{equation}
By construction, we have $v_0=\eins=v_T$, and thus we may view $v$ in a natural way as a unitary in $\CU(\eins+\CC(\IR/T\IZ)\otimes A)$ via $v(s+T\IZ)=v_s$ for $s\in [0,T]$.

The first part of the claim now follows as we compute
\[
\begin{array}{ccl}
\renewcommand\arraystretch{1.5}
\dst \max_{a\in\CF}~ \|[v,\eins\otimes a]\| &\stackrel{\eqref{eq:cv-pre-v}}{=}& \dst \max_{a\in\CF}~\max_{0\leq s\leq T}~ \|[w_s\alpha_s(\kappa_s^*),a]\| \\
&\leq& \dst \max_{a\in\CF}~\max_{0\leq s\leq T}~ \|[w_s,a]\|+\|[\kappa_s,\alpha_{-s}(a)]\| \\
&\stackrel{\eqref{eq:cv-pre-delta-F'}, \eqref{eq:cv-pre-w-2}, \eqref{eq:cv-pre-kappa-1}}{\leq}& \delta' + \eps/6 \ \leq \ \eps.
\end{array}
\]
For the second part of the claim, let us from now on fix a given $t\in [-1,1]$ and $a\in\CF$. Let also $s\in [0,T]$ be given. If $s-t\in [0,T]$, then one has
\[
\begin{array}{ccl}
\renewcommand\arraystretch{1.25}
a\cdot \big( v\cdot(\sigma^T\otimes\alpha)_t(v^*) \big)(s)  &\stackrel{\eqref{eq:cv-pre-v}}{=}& a \cdot w_s\alpha_s(\kappa_s^*)\cdot \alpha_t\big( \alpha_{s-t}(\kappa_{s-t}) w_{s-t}^* \big) \\
&\stackrel{\eqref{eq:cv-pre-w-2}, \eqref{eq:cv-pre-kappa-1}}{=}_{\makebox[0pt]{\footnotesize \hspace{-6mm}$2\eps/6$}}& w_s \alpha_s(\kappa_s^*)\cdot a \cdot \alpha_s(\kappa_{s-t}) \alpha_t(w_{s-t}^*) \\
&\stackrel{\eqref{eq:cv-pre-kappa-2}}{=}_{\makebox[0pt]{\footnotesize $\eps/6$}}& w_s\cdot \alpha_s(\kappa_s^*) a \alpha_s(\kappa_s) \cdot \alpha_t(w_{s-t}^*) \\
&\stackrel{\eqref{eq:cv-pre-w-2}, \eqref{eq:cv-pre-kappa-1}}{=}_{\makebox[0pt]{\footnotesize \hspace{-6mm}$2\eps/6$}}&  a \cdot w_s\alpha_t(w_{s-t}^*) \\
&=& a\cdot w_s\alpha_t(w_{s-t}^*)w_t^*\cdot w_t \ = \ a w_t .
\end{array}
\]
On the other hand, if $s-t\notin [0,T]$, then either $s-t<0$ or $s-t>T$.

If $s<t$, then in particular $s\leq 1$ and $T-1 \leq T+s-t$, and we compute
\[
\renewcommand\arraystretch{1.25}
\begin{array}{cl}
\multicolumn{2}{l}{ a \cdot \big( v\cdot(\sigma^T\otimes\alpha)_t(v^*) \big)(s) } \\
\stackrel{\eqref{eq:cv-pre-v}}{=}& a \cdot w_s\alpha_s(\kappa_s^*)\cdot \alpha_t\big( \alpha_{T+s-t}(\kappa_{T+s-t}) w_{T+s-t}^* \big) \\
\stackrel{\eqref{eq:cv-pre-w-2}, \eqref{eq:cv-pre-kappa-1}}{=}_{\makebox[0pt]{\footnotesize \hspace{-6mm}$2\eps/6$}}& w_s\alpha_s(\kappa_s^*)\cdot a\cdot \alpha_{T+s}(\kappa_{T+s-t}) \alpha_t(w_{T+s-t}^*) \\
\stackrel{\eqref{eq:cv-pre-kappa-endpoint}, \eqref{eq:cv-pre-kappa-2}}{=}_{\makebox[0pt]{\footnotesize \hspace{-6mm}$2\eps/6$}}& w_s \cdot a \cdot \alpha_{s}(w_T)  \alpha_t(w_{T+s-t}^*) \\
\stackrel{\eqref{eq:cv-pre-w-2}}{=}_{\makebox[0pt]{\footnotesize $\eps/6$}}& a\cdot w_s\alpha_s(w_T)\alpha_t(w_{T+s-t}^*) \\
=& a \cdot w_{T+s}\alpha_t(w_{T+s-t}^*)w_t^*\cdot w_t \\
=& aw_t .
\end{array}
\]
Lastly, if $s-t>T$, then in particular $0< s-t-T\leq 1$ and $s\geq T-1$, and we compute
\[
\renewcommand\arraystretch{1.25}
\begin{array}{cl}
\multicolumn{2}{l}{ a \cdot \big( v\cdot(\sigma^T\otimes\alpha)_t(v^*) \big)(s) } \\
\stackrel{\eqref{eq:cv-pre-v}}{=}& a \cdot w_s\alpha_s(\kappa_s^*)\cdot \alpha_t\big( \alpha_{s-t-T}(\kappa_{s-t-T}) w_{s-t-T}^* \big) \\
\stackrel{\eqref{eq:cv-pre-w-2}, \eqref{eq:cv-pre-kappa-1}}{=}_{\makebox[0pt]{\footnotesize \hspace{-6mm}$2\eps/6$}}&
w_s\alpha_s(\kappa_s^*)\cdot a \cdot \alpha_{s-T}(\kappa_{s-t-T})\alpha_t(w_{s-t-T}^*) \\
\stackrel{\eqref{eq:cv-pre-kappa-endpoint}, \eqref{eq:cv-pre-kappa-2}}{=}_{\makebox[0pt]{\footnotesize \hspace{-6mm}$2\eps/6$}}&
w_s\alpha_s(w_{-T})\cdot  a\cdot\alpha_t(w_{s-t-T}^*) \\
\stackrel{\eqref{eq:cv-pre-w-2}}{=}_{\makebox[0pt]{\footnotesize \hspace{2mm}$2\eps/6$}}&
a \cdot w_s\alpha_s(w_{-T})\alpha_t(w_{s-t-T}^*) \ = \ a w_t.
\end{array}
\]
All in all, this shows
\[
\max_{a\in\CF}~ \max_{|t|\leq 1}~ \|\eins\otimes a \cdot \big( \eins\otimes w_t-v(\sigma^T\otimes\alpha)_t(v^*) \big)\|\leq\eps
\]
and finishes the proof of the first claim.

If we additionally assume $\acel(A)<\infty$ and pick $C>\acel(A)$ and $T\in\IN$ as above, then the analogous calculations\footnote{In fact, just insert $a=\eins$ in the previous calculations, although some steps then become redundant.} yield
\[
\max_{|t|\leq 1}~ \|\eins\otimes w_t-v(\sigma^T\otimes\alpha)_t(v^*)\|\leq\eps
\] 
instead.
\end{proof}

In addition to the key Lemma \ref{lem:acp} from the previous section, we need the following approximate first-cohomology vanishing result, which follows from Lemma \ref{lem:cv-pre} and is a modified version of Kishimoto's cohomology vanishing result from \cite[Theorem 2.1]{Kishimoto96_R}.

\begin{lemma} \label{lem:cv}
Let $A$ be a separable \cstar-algebra with $\acel_w(A)<\infty$. Let $\alpha: \IR\curvearrowright A$ be a Rokhlin flow. Then for every $\eps>0$ and $\CF\fin A$, there exist $\delta>0$ and $\CG\fin A$ with the following property:

Suppose that $w: \IR\to\CU(\eins+A)$ is an $\alpha$-cocycle with
\[
\max_{a\in\CG}~ \max_{0\leq t\leq 1}~ \|[w_t,a]\|\leq\delta.
\]
Then there exists a unitary $v\in\CU(\eins+A)$ such that
\[
\max_{a\in\CF}~ \|[v,a]\|\leq\eps
\]
and
\[
\max_{a\in\CF}~ \max_{|t|\leq 1}~ \|a\cdot \big(w_t-v\alpha_t(v^*)\big)\|\leq\eps.
\]
If moreover $\acel(A)<\infty$, then the last part of the conclusion can be improved to
\[
\max_{|t|\leq 1}~ \|w_t-v\alpha_t(v^*)\|\leq\eps.
\]
\end{lemma}
\begin{proof}
Let $\eps>0$ and $\CF\fin A$ be given. We apply Lemma \ref{lem:cv-pre} and find $T>0$, $\delta>0$ and $\CG\fin A$ for the pair $(\eps/2,\CF)$. We claim that the resulting pair $(\delta,\CG)$ has the desired property.

As $\alpha$ has the Rokhlin property, we may apply Lemma \ref{lem:seq-split-picture} and find an equivariant $*$-homomorphism $\psi: \big( \CC(\IR/T\IZ)\otimes A, \sigma^T\otimes\alpha \big) \to ( A_{\infty,\alpha}, \alpha_\infty )$ such that $\psi(\eins\otimes a)=a$ for all $a\in A$. Let $\psi^+$ be the unital extension between the proper unitizations on both sides.

Now let $w: \IR\to\CU(\eins+A)$ be an $\alpha$-cocycle with
\[
\max_{a\in\CG}~ \max_{0\leq t\leq 1}~ \|[w_t,a]\|\leq\delta.
\]
According to our choice of $(T,\delta,G)$, we find a unitary 
\[
v_0\in\CU(\eins+\CC(\IR/T\IZ)\otimes A))
\]
such that
\[
\max_{a\in\CF}~ \|[v_0,\eins\otimes a]\|\leq\eps/2
\]
and
\[
\max_{a\in\CF}~ \max_{|t|\leq 1}~ \|\eins\otimes a\cdot \big( \eins\otimes w_t-v_0(\sigma^T\otimes\alpha)_t(v_0^*) \big)\|\leq\eps/2.
\]
Let us find a sequence of unitaries $v_n\in\CU(\eins+A)$ representing $\psi^+(v_0)\in\CU(\eins+A_{\infty,\alpha})$. Then by the choice of $\psi$ we have
\[
\limsup_{n\to\infty}~ \max_{a\in\CF}~ \|[v_n,a]\|\leq\eps/2
\]
and
\[
\limsup_{n\to\infty}~ \max_{a\in\CF}~ \max_{|t|\leq 1}~ \|a \cdot \big(w_t-v_n\alpha_t(v_n^*)\big)\|\leq\eps/2.
\]
Hence for sufficiently large $n$, the unitary $v=v_n$ has the desired property as in the claim.

If moreover $\acel(A)<\infty$, then we can follow the same line of argument with the improved statement in Lemma \ref{lem:cv-pre}, whereby the unitary $v_0$ will satisfy
\[
\max_{|t|\leq 1}~ \|\eins\otimes w_t-v_0(\sigma^T\otimes\alpha)_t(v_0^*)\|\leq\eps/2.
\]
The resulting sequence of unitaries $v_n$ will then satisfy
\[
\limsup_{n\to\infty}~ \max_{|t|\leq 1}~ \|w_t-v_n\alpha_t(v_n^*)\|\leq\eps/2
\]
and so we get the desired property for sufficiently large $n$.
\end{proof}

The following abstract classification theorem for Rokhlin flows is arguably the main abstract result of this paper, and is a natural \cstar-algebraic analog of Masuda--Tomatsu's classification \cite[Theorem 1]{MasudaTomatsu16} of Rokhlin flows on von Neumann algebras.

The reader should now recall the Definitions \ref{def:weak-inner-length} and \ref{def:acel} for what the assumptions $\wil(A)<\infty$, $\acel_w(A)<\infty$ or $\acel(A)<\infty$ mean.
See also Notation \ref{nota:restricted-coaction}, where we introduced the restricted coaction of a flow.

\begin{theorem} \label{thm:main-result}
Let $A$ be a separable \cstar-algebra with $\wil(A)<\infty$ and $\acel_w(A)<\infty$. Let $\alpha, \beta: \IR\curvearrowright A$ be two Rokhlin flows. Then the following are equivalent:
\begin{enumerate}[label=\textup{(\roman*)}]
\item $\alpha$ and $\beta$ are strongly cocycle conjugate via an approximately inner automorphism; \label{main-result:1}
\item $\alpha$ and $\beta$ are cocycle conjugate via an approximately inner automorphism; \label{main-result:2}
\item The maps $\alpha_\co$ and $\beta_\co$ are approximately unitarily equivalent. \label{main-result:3}
\end{enumerate}
If moreover $\acel(A)<\infty$, then these statements are further equivalent to
\begin{enumerate}[label=\textup{(\roman*)}, resume]
\item $\alpha$ and $\beta$ are norm-strongly cocycle conjugate via an approximately inner automorphism. \label{main-result:4}
\end{enumerate}
\end{theorem}
\begin{proof}
The implications \ref{main-result:4}$\Rightarrow$\ref{main-result:1}$\Rightarrow$\ref{main-result:2} are trivial.
The implication \ref{main-result:2}$\Rightarrow$\ref{main-result:3} holds due to Lemma \ref{lem:cc-ue}.
The essence of the claim is in the implication \ref{main-result:3}$\Rightarrow$\ref{main-result:1}, as well as \ref{main-result:3}$\Rightarrow$\ref{main-result:4} in case of $\acel(A)<\infty$.
Let us first proceed with \ref{main-result:3}$\Rightarrow$\ref{main-result:1}.

Before going into the details, it should be mentioned that the proof consists of implementing a slightly modified version of the Evans--Kishimoto intertwining argument invented in \cite{EvansKishimoto97}, with the help of Lemmas \ref{lem:acp} and \ref{lem:cv}. 
In particular, most of the arguments below may be familiar to some experts. 
That being said, the actual implementation of the intertwining technique is by no means trivial, and therefore we give the argument in full detail.

Assume that $\alpha_\co\ue\beta_\co$.
Note that if we perturb either $\alpha$ or $\beta$ with some cocycle, then both will still have the Rokhlin property and the restricted coactions will still be approximately unitarily equivalent; cf.\ Lemma \ref{lem:cc-ue}. This will be used repeatedly without mention.
For the rest of the proof, let $\CF_n\fin A_{\leq 1}$ be an increasing sequence of finite sets of contractions with dense union in the unit ball.

Put $\alpha^0=\alpha$ and $\beta^1=\beta$. 
By applying Lemmas \ref{lem:acp} and \ref{lem:cv} in a certain zig-zag way, we will inductively construct flows $\alpha^{2k}$ and $\beta^{2k+1}$ on $A$\footnote{The superscripts attached to flows in this context are not to be confused with taking powers of the involved automorphisms, which we will never consider throughout this proof.}
and unitaries $z_k,v_k\in\CU(\eins+A)$ satisfying the following list of properties for all $k\geq 0$ (or $k\geq 1$ where appropriate):
\begin{equation} \label{eq:alpha-2k}
\alpha^{2(k+1)}=\ad(z_{2k})\circ\alpha^{2k}\circ\ad(z_{2k}^*);
\end{equation}
\begin{equation} \label{eq:beta-2k+1}
\beta^{2k+1} = \ad(z_{2k-1})\circ\beta^{2k-1}\circ\ad(z_{2k-1}^*);
\end{equation}
\begin{equation} \label{eq:alpha-beta}
\max_{a\in\CF_{2k+1}}~ \max_{|t|\leq 1}~ \|\beta^{2k+1}_t(a)-\alpha^{2k}_t(a)\| \leq 2^{-(2k+1)};
\end{equation}
\begin{equation} \label{eq:F-prime-even}
\CF_{2k}' = \CF_{2k}\cup\ad(v_{2(k-1)}\cdots v_0)(F_{2k})\cup\CF_{2k}v_0^*\cdots v_{2(k-1)}^*;
\end{equation}
\begin{equation} \label{eq:F-prime-odd}
\CF_{2k+1}' = \CF_{2k+1}\cup\ad(v_{2k-1}\cdots v_1)(F_{2k+1})\cup\CF_{2k+1}v_1^*\cdots v_{2k-1}^*;
\end{equation}
\begin{equation} \label{eq:zv-even}
\max_{a\in\CF_{2k}'}~ \max_{|t|\leq 1}~ \| a\big( z_{2k}\alpha^{2k}_t(z_{2k}^*)-v_{2k}\alpha^{2k}_t(v_{2k}^*) \big)\|\leq 2^{-2k};
\end{equation}
\begin{equation} \label{eq:zv-odd}
\max_{a\in\CF_{2k+1}'}~ \max_{|t|\leq 1}~ \| a\big( z_{2k+1}\beta^{2k+1}_t(z_{2k+1}^*)-v_{2k+1}\beta^{2k+1}_t(v_{2k+1}^*) \big)\|\leq 2^{-(2k+1)};
\end{equation}
\begin{equation} \label{eq:cauchy-even}
\max_{a\in\CF_{2k}'}~ \|[v_{2k},a]\|\leq 2^{-2k};
\end{equation}
\begin{equation} \label{eq:cauchy-odd}
\max_{a\in\CF_{2k+1}'}~\|[v_{2k+1},a]\|\leq 2^{-(2k+1)}.
\end{equation}
Once this is done, these unitaries will allow us to construct approximately inner automorphisms $\phi_0,\phi_1$ on $A$, an approximate $\alpha$-coboundary $\set{ W_t^0}_{t\in\IR} \subset \CU(\CM(A))$ and an approximate $\beta$-coboundary $\set{W_t^1}_{t\in\IR} \subset \CU(\CM(A))$ such that
\[
\phi_0\circ\ad(W_t^0)\circ\alpha_t\circ\phi_0^{-1} = \phi_1\circ\ad(W_t^1)\circ\beta_t\circ\phi_1^{-1} ,\quad t\in\IR.
\]
Let us now implement this strategy and first construct the flows $\alpha^{2k}$, $\beta^{2k
1}$, and unitaries $z_k, v_k$ with the above properties. 

Apply Lemma \ref{lem:cv} to $\beta^1$ and the pair $(1/2, \CF_1)$ and choose a pair $(\delta_1,\CG_1)$. Define the compact set
\begin{equation} \label{eq:G-prime-1}
\CG_1' = \bigcup_{|t|\leq 1} \beta^1_{-t}(\CG_1)\cup\CF_1.
\end{equation}
Apply Lemma \ref{lem:acp} to $\alpha^0$ in place of $\beta$ and find a unitary $z_0\in\CU(\eins+A)$ such that
\begin{equation} \label{eq:step-0}
\max_{a\in\CG_1'}~ \max_{|t|\leq 1}~ \|\beta^1_t(a)-\big( \ad(z_0)\circ\alpha^0_t\circ\ad(z_0^*) \big)(a)\|<\delta_1/2.
\end{equation}
Set $\alpha^2=\ad(z_0)\circ\alpha^0\circ\ad(z_0^*)$. Set $v_0=z_0$ and
\[
\CF_2' = \CF_2\cup v_0\CF_2v_0^*\cup\CF_2 v_0^*.
\]
Apply Lemma \ref{lem:cv} to $\alpha^2$ and the pair $(1/4, \CF_2')$ and choose a pair $(\delta_2,\CG_2)$ with $\delta_2\leq\min(1/4, \delta_1)$. We let
\begin{equation} \label{eq:G-prime-2}
\CG_2' = \bigcup_{|t|\leq 1} \alpha^2_{-t}(\CG_2)\cup\beta^1_{-t}(\CG_1)\cup\CF_2.
\end{equation}
Apply Lemma \ref{lem:acp} to $\beta^1$ in place of $\beta$ and find a unitary $z_1\in\CU(\eins+A)$ with
\begin{equation} \label{eq:step-1}
\max_{a\in\CG_2'}~ \max_{|t|\leq 1}~ \|\alpha^2_t(a)-\big( \ad(z_1)\circ\beta^1_t\circ\ad(z_1^*) \big)(a)\|<\delta_2/2.
\end{equation}
Set $\beta^3=\ad(z_1)\circ\beta^1\circ\ad(z_1^*)$. Combining \eqref{eq:G-prime-1}, \eqref{eq:step-0} and \eqref{eq:step-1}, we see that
\[
\max_{a\in\CG_1}~ \max_{|t|\leq 1}~ \|[z_1\beta^1_t(z_1^*),a]\|\leq\delta_1.
\]
By our choice of $(\delta_1,\CG_1)$, there exists a unitary $v_1\in\CU(\eins+A)$ with
\[
\|[v_1,a]\|\leq 1/2 \quad\text{and}\quad \max_{|t|\leq 1}~ \|a\big( z_1\beta^1_t(z_1^*)-v_1\beta^1_t(v_1^*) \big)\|\leq 1/2 
\]
for all $a\in\CF_1$. Set
\[
\CF_3' = \CF_3\cup v_1\CF_3v_1^* \cup \CF_3v_1^*.
\]
Apply Lemma \ref{lem:cv} to $\beta^3$ and the pair $(1/8, \CF_3')$ and choose a pair $(\delta_3,\CG_3)$ with $\delta_3\leq\min(1/8,\delta_2)$. Set
\begin{equation} \label{eq:G-prime-3}
\CG_3' = \bigcup_{|t|\leq 1} \beta^3_{-t}(\CG_3)\cup\alpha^2_{-t}(\CG_2)\cup\CF_3.
\end{equation}
Apply Lemma \ref{lem:acp} to $\alpha^2$ in place of $\beta$ and find a unitary $z_2\in\CU(\eins+A)$ with
\begin{equation} \label{eq:step-2}
\max_{a\in\CG_3'}~ \max_{|t|\leq 1}~ \|\beta^3_t(a)-\big( \ad(z_2)\circ\alpha^2_t\circ\ad(z_2^*) \big)(a)\|<\delta_3/2.
\end{equation}
Set $\alpha^4= \ad(z_2)\circ\alpha^2\circ\ad(z_2^*)$. Combining \eqref{eq:G-prime-2}, \eqref{eq:step-1} and \eqref{eq:step-2}, we see
\[
\max_{a\in\CG_2}~ \max_{|t|\leq 1}~ \|[z_2\alpha^2_t(z_2^*),a]\|\leq\delta_2.
\]
By our choice of $(\delta_2,\CG_2)$, there exists a unitary $v_2\in\CU(\eins+A)$ with
\[
\|[v_2,a]\|\leq 1/4 \quad\text{and}\quad \max_{|t|\leq 1}~ \|a\big( z_2\alpha^2_t(z_2^*)-v_2\alpha^2_t(v_2^*) \big)\|\leq 1/4 
\]
for all $a\in\CF_2'$. Set
\[
\CF_4' = \CF_4\cup v_2v_0\CF_4v_0^*v_2^* \cup\CF_4v_0^*v_2^*.
\]
Apply Lemma \ref{lem:cv} to $\alpha^4$ and the pair $(1/16, \CF_4')$ and choose a pair $(\delta_4,\CG_4)$ with $\delta_4\leq\min(1/16,\delta_3)$. Set
\[
\CG_4' = \bigcup_{0\leq t\leq 1} \alpha^4_{-t}(\CG_4)\cup\beta^3_{-t}(\CG_3)\cup\CF_4.
\]
Apply Lemma \ref{lem:acp} to $\beta^3$ in place of $\beta$ and find a unitary $z_3\in\CU(\eins+A)$ with
\begin{equation} \label{eq:step-3}
\max_{a\in\CG_4'}~ \max_{|t|\leq 1}~ \|\alpha^4_t(a)-\big( \ad(z_3)\circ\beta^3_t\circ\ad(z_3^*) \big)(a)\|<\delta_4/2.
\end{equation}
Set $\beta^5= \ad(z_3)\circ\beta^3\circ\ad(z_3^*)$. Combining \eqref{eq:G-prime-3}, \eqref{eq:step-2} and \eqref{eq:step-3}, we see
\[
\max_{a\in\CG_3}~ \max_{|t|\leq 1}~ \|[z_3\beta^3_t(z_3^*),a]\|\leq\delta_3.
\]
By our choice of $(\delta_3,\CG_3)$, there exists a unitary $v_3\in\CU(\eins+A)$ with
\[
\|[v_3,a]\|\leq 1/8 \quad\text{and}\quad \max_{|t|\leq 1}~ \|a\big( z_3\beta^3_t(z_3^*)-v_3\beta^3_t(v_3^*) \big)\|\leq 1/8 
\]
for all $a\in\CF_3'$.

Carry on inductively as above.
This yields sequences of flows $(\alpha^{2k})_{k\geq 0}$ and $(\beta^{2k+1})_{k\geq 0}$ on $A$ and sequences of unitaries $z_k, v_k\in\CU(\eins+A)$ satisfying the properties as in \eqref{eq:alpha-2k} to \eqref{eq:cauchy-odd}.

We shall now construct the aforementioned automorphisms $\phi_0,\phi_1$ on $A$ and cocycles $\set{ W_t^0}_{t\in\IR}$, $\set{W_t^1}_{t\in\IR}$ that will implement the strong cocycle conjugacy between $\alpha$ and $\beta$.

Consider the sequence of unitaries $V_n\in\CU(\eins+A)$ given by
\begin{equation} \label{eq:Vn}
V_n = \begin{cases} v_{2k+1}v_{2k-1}\cdots v_3 v_1 &,\quad n=2k+1 \\
v_{2k}v_{2(k-1)}\cdots v_2 v_0 &,\quad n=2k.  \end{cases}
\end{equation}
In other words, depending whether $n$ is odd or even, $V_n$ is defined as the product of all unitaries $v_j$ for odd or even indices $j\leq n$.
Define inner automorphisms $\sigma_n\in\Aut(A)$ by $\sigma_n=\ad(V_n)$.
Since the union of the sets $\CF_n$ is dense in the unit ball of $A$, it follows from \eqref{eq:F-prime-even} and \eqref{eq:cauchy-even} that the sequences $\big( \sigma_{2k}(a) \big)_{2k}$ and $\big( \sigma_{2k}^{-1}(a) \big)_{2k}$ are Cauchy sequences for every $a\in A$. 
Hence the point-norm limit
\begin{equation} \label{eq:phi-0}
\phi_0 = \lim_{k\to\infty} \sigma_{2k}
\end{equation}
exists and yields an automorphism on $A$. Similarly, it follows from \eqref{eq:F-prime-odd} and \eqref{eq:cauchy-odd} that
\begin{equation} \label{eq:phi-1}
\phi_1 = \lim_{k\to\infty} \sigma_{2k+1}
\end{equation}
exists and yields an automorphism on $A$.

Let us now turn to constructing the cocycles. Consider the sequences of unitaries $Z_n, X_n\in\CU(\eins+A)$ given by
\begin{equation} \label{eq:Zn}
Z_n = \begin{cases} z_{2k+1} z_{2k-1}\cdots z_3 z_1 &,\quad n=2k+1 \\
z_{2k} z_{2(k-1)}\cdots z_2 z_0 &,\quad n=2k  \end{cases}
\end{equation}
and
\begin{equation} \label{eq:Xn}
X_n = V_n^*Z_n.
\end{equation}
We shall now show that for any $a\in A$, the sequences of functions of the form 
\[
[t\mapsto aX_{2k}\alpha_t(X_{2k}^*)] \quad\text{and}\quad [t\mapsto aX_{2k+1}\beta_t(X_{2k+1}^*)]
\]
satisfy the Cauchy criterion uniformly on $t\in [-1,1]$.
For this purpose, it is enough to show this for $a$ being in one of the sets $\CF_n$.
Let such $n$ be fixed for the moment.

Let us treat the even sequence.
Note that \eqref{eq:alpha-2k} translates to 
\begin{equation} \label{eq:alpha-2k-2}
\alpha^{2(k+1)}=\ad(Z_{2k})\circ\alpha\circ\ad(Z_{2k}^*) ,\quad k\geq 0,
\end{equation}
and \eqref{eq:beta-2k+1} translates to 
\begin{equation} \label{eq:beta-2k+1-2}
\beta^{2k+1}=\ad(Z_{2k-1})\circ\beta\circ\ad(Z_{2k-1}^*) ,\quad k\geq 1.
\end{equation}
Using \eqref{eq:alpha-2k}, we see for all $t\in [-1,1]$ that
\[
z_{2k}\alpha^{2k}_t(z_{2k}^*) = \alpha^{2(k+1)}_t(z_{2k}^*)z_{2k}
\]
and
\[
v_{2k}\alpha^{2k}_t(v_{2k}^*) = v_{2k} z_{2k}^* \alpha_t^{2(k+1)}\big( z_{2k} v_{2k}^* z_{2k}^* \big) z_{2k}.
\]
For $b\in\CF_{2k}'$ we use this with \eqref{eq:zv-even} and compute
\[
\renewcommand\arraystretch{1.5}
\begin{array}{cl}
\multicolumn{2}{l}{
\| b \big( v_{2k}^* \alpha_t^{2(k+1)}(v_{2k}) - z_{2k}^* \alpha_t^{2(k+1)}(z_{2k}) \big) \| } \\
=& \| v_{2k} b \big( v_{2k}^* \alpha_t^{2(k+1)}(v_{2k}) - z_{2k}^* \alpha_t^{2(k+1)}(z_{2k}) \big) \| \\
\stackrel{\eqref{eq:cauchy-even}}{\leq}& 2^{-2k}+\| b\big(\alpha^{2(k+1)}_t(v_{2k})-v_{2k}z_{2k}^*\alpha^{2(k+1)}_t(z_{2k}) \big)\| \\
=& 2^{-2k} + \| b\big( \alpha^{2(k+1)}_t(z_{2k}^*)z_{2k} - v_{2k} z_{2k}^* \alpha_t^{2(k+1)}\big( z_{2k} v_{2k}^* z_{2k}^* \big) z_{2k} \big)\| \\
=& 2^{-2k} + \| b\big( z_{2k}\alpha^{2k}_t(z_{2k}^*) - v_{2k}\alpha^{2k}_t(v_{2k}^*) \big)\| \\
\stackrel{\eqref{eq:zv-even}}{\leq}& 2^{1-{2k}}.
\end{array}
\]
To summarize, we have
\begin{equation} \label{eq:zv-even-2}
\max_{b\in\CF_{2k}'}~ \max_{|t|\leq 1}~ \| b\big( v_{2k}^* \alpha_t^{2(k+1)}(v_{2k}) - z_{2k}^* \alpha_t^{2(k+1)}(z_{2k}) \big)\| \leq 2^{1-2k}.
\end{equation}
If $a\in\CF_n$, then $aV_{2(k-1)}^*\in\CF_{2k}'$ for all $k\geq n/2$ by \eqref{eq:F-prime-even}. So for every $t\in [-1,1]$ it follows that
\[
\renewcommand\arraystretch{1.5}
\begin{array}{cl}
\multicolumn{2}{l}{ a X_{2k}\alpha_t(X_{2k}^*) } \\
\stackrel{\eqref{eq:Xn}}{=}& a V_{2k}^* Z_{2k} \alpha_t(Z_{2k}^* V_{2k}) \\
=& a V_{2k}^* Z_{2k} \alpha_t(Z_{2k}^* V_{2k} Z_{2k}^* Z_{2k}) Z_{2k}^* Z_{2k} \\
\stackrel{\eqref{eq:alpha-2k-2}}{=}& a V_{2k}^* \alpha_t^{2(k+1)}(V_{2k} Z_{2k}^*) Z_{2k} \\
=& a V_{2(k-1)}^* v_{2k}^* \alpha^{2(k+1)}_t\big( v_{2k} V_{2(k-1)} Z_{2(k-1)}^* z_{2k}^* \big) z_{2k} Z_{2(k-1)} \\
\stackrel{\eqref{eq:zv-even-2}}{=}_{\makebox[0pt]{\footnotesize\hspace{2mm}$2^{1-2k}$}} ~~ &  a V_{2(k-1)}^* z_{2k}^* \alpha_t^{2(k+1)}\big( z_{2k} V_{2(k-1)} Z_{2(k-1)}^* z_{2k}^* \big) z_{2k} Z_{2(k-1)} \\
\stackrel{\eqref{eq:alpha-2k}}{=}& a V_{2(k-1)}^*\alpha_t^{2k}(V_{2(k-1)} Z_{2(k-1)}^*) Z_{2(k-1)} \\
=& a V_{2(k-1)}^* Z_{2(k-1)} Z_{2(k-1)}^* \alpha_t^{2k}(Z_{2(k-1)} Z_{2(k-1)}^* V_{2(k-1)} Z_{2(k-1)}^*) Z_{2(k-1)} \\
\stackrel{\eqref{eq:alpha-2k-2}}{=}& a V_{2(k-1)}^* Z_{2(k-1)} \alpha_t( Z_{2(k-1)}^* V_{2(k-1)} ) \ \stackrel{\eqref{eq:Xn}}{=} \ a X_{2(k-1)}\alpha_t(X_{2(k-1)}^*).
\end{array}
\]
As the union $\bigcup_{n\in\IN} \CF_n$ is dense in the unit ball of $A$, we see that the sequence of functions 
\[
\big[ t\mapsto aX_{2k}\alpha_t(X_{2k}^*) \big]
\]
is uniformly Cauchy on $t\in [-1,1]$ for every $a\in A$. As one has
\[
X_{2k}\alpha_t(X_{2k}^*) \cdot a = \alpha_{t}\Big( \alpha_{-t}(a)^*\cdot X_{2k}\alpha_{-t}(X_{2k}^*) \Big)^*,
\]
the same follows for the functions
\[
\big[ t\mapsto X_{2k}\alpha_t(X_{2k}^*) a \big].
\]
In particular, we see that the strict limits
\begin{equation} \label{eq:final-cocycle-alpha}
W^0_t = \lim_{k\to\infty} X_{2k}\alpha_t(X_{2k}^*) \ \in \ \CU(\CM(A)),\quad t\in\IR
\end{equation}
exist and yield a strictly continuous map $W^0: \IR\to\CU(\CM(A))$, which by definition is an approximate $\alpha$-coboundary. Following the same line of argument for the odd sequence, one sees that the strict limits  
\begin{equation} \label{eq:final-cocycle-beta}
W^1_t = \lim_{k\to\infty} X_{2k+1}\beta_t(X_{2k+1}^*) \ \in \ \CU(\CM(A)),\quad t\in\IR
\end{equation}
exist and yield a strictly continuous, approximate $\beta$-coboundary $W^1: \IR\to\CU(\CM(A))$.

By the definition of $\sigma_{n}$, the construction of the flows $\alpha^{2k}$ \eqref{eq:alpha-2k-2} and $\beta^{2k+1}$ \eqref{eq:beta-2k+1-2}, and the definition of the elements $X_{n}$ \eqref{eq:Xn}, we have
\begin{equation} \label{eq:cocycle-eq-even}
\sigma_{2k}\circ\ad(X_{2k})\circ\alpha\circ\ad(X_{2k}^*)\circ\sigma_{2k}^{-1} = \alpha^{2(k+1)}
\end{equation}
and
\begin{equation} \label{eq:cocycle-eq-odd}
\sigma_{2k+1}\circ\ad(X_{2k+1})\circ\beta\circ\ad(X_{2k+1}^*)\circ\sigma_{2k+1}^{-1} = \beta^{2k+3}
\end{equation}
for all $k\geq 0$.
Finally, we compute for every $t\in\IR$ that
\[
\renewcommand\arraystretch{1.75}
\begin{array}{cl}
\multicolumn{2}{l}{ \phi_0\circ\ad(W^0_t)\circ\alpha_t\circ\phi_0^{-1} } \\
\stackrel{\eqref{eq:final-cocycle-alpha}, \eqref{eq:phi-0}}{=}& 
\dst \lim_{k\to\infty} ~ \sigma_{2k}\circ\ad(X_{2k})\circ\alpha_t\circ\ad(X_{2k}^*)\circ\sigma_{2k}^{-1} \\
\stackrel{\eqref{eq:cocycle-eq-even}}{=}& \dst \lim_{k\to\infty} ~ \alpha_t^{2(k+1)} \\
\stackrel{\eqref{eq:alpha-beta}}{=}& \dst \lim_{k\to\infty} ~ \beta_t^{2k+3} \\
\stackrel{\eqref{eq:cocycle-eq-odd}}{=}& \dst \lim_{k\to\infty} ~ \sigma_{2k+1}\circ\ad(X_{2k+1})\circ\beta_t\circ\ad(X_{2k+1}^*)\circ\sigma_{2k+1}^{-1} \\
\stackrel{\eqref{eq:final-cocycle-beta}, \eqref{eq:phi-1}}{=}& \phi_1\circ\ad(W_t^1)\circ\beta_t\circ\phi_1^{-1}.
\end{array}
\]
This finishes the proof for the implication \ref{main-result:3}$\Rightarrow$\ref{main-result:1}.

Now let us proceed with \ref{main-result:3}$\Rightarrow$\ref{main-result:4}.
Assume $\acel(A)<\infty$.
In this case, follow the same intertwining argument as above, but using the improved statement in Lemma \ref{lem:cv}.
The only difference in the argument is then that \eqref{eq:zv-even} and \eqref{eq:zv-odd} are replaced by
\begin{equation} \label{eq:zv-even-mod}
\max_{|t|\leq 1}~ \|z_{2k}\alpha^{2k}_t(z_{2k}^*)-v_{2k}\alpha^{2k}_t(v_{2k}^*)\|\leq 2^{-2k}
\end{equation}
and
\begin{equation} \label{eq:zv-odd-mod}
\max_{|t|\leq 1}~ \|z_{2k+1}\beta^{2k+1}_t(z_{2k+1}^*)-v_{2k+1}\beta^{2k+1}_t(v_{2k+1}^*)\|\leq 2^{-(2k+1)}.
\end{equation}
Carrying out the calculation right after \eqref{eq:zv-even-2} with $a=\eins$ yields
\[
\max_{|t|\leq 1}~ \|X_{2k}\alpha_t(X_{2k}^*)-X_{2(k-1)}\alpha_t(X_{2(k-1)}^*)\|\leq 2^{-2k}, \quad k\geq 1.
\]
Thus it follows that the convergence
\begin{equation} \label{eq:final-cocycle-alpha-mod}
W^0_t = \lim_{k\to\infty} X_{2k}\alpha_t(X_{2k}^*) \ \in \ \CU(\eins+A)
\end{equation}
holds in norm uniformly in $t\in [-1,1]$, and so $\set{W^0_t}_{t\in\IR}$ forms a norm-approximate $\alpha$-coboundary.

Following the same line of argument for the odd sequence, one sees that  
\begin{equation} \label{eq:final-cocycle-beta-mod}
W^1_t = \lim_{k\to\infty} X_{2k+1}\beta_t(X_{2k+1}^*) \ \in \ \CU(\eins+A)
\end{equation}
yields a norm-approximate $\beta$-coboundary.
Thus we obtain norm-strong cocycle conjugacy between $\alpha$ and $\beta$.
This concludes the proof.
\end{proof}

\begin{rem} \label{ex:uniform-assumption}
As we have seen in Lemma \ref{lem:cc-ue}, the condition $\alpha_\co\ue\beta_\co$ is a necessary assumption to get any of the other statements in Theorem \ref{thm:main-result}.
For von Neumann algebras, Masuda--Tomatsu's classification theorem already yields the desired conclusion under the assumption that $\alpha_t\ue\beta_t$ for all $t\in\IR$.
This is essentially due to Borel functional calculus, and in fact the point-wise assumption always implies the uniform one; see \cite[Theorem 9.5]{MasudaTomatsu16}.
As was alluded to in the introduction, the pointwise assumption cannot be expected to be sufficient in general for \cstar-algebras.

To see this more concretely, we remark that if $\alpha_\co\ue\beta_\co$, then an elementary calculation (which we shall omit) reveals that for every $t\in\IR$, the automorphisms $\alpha_t$ and $\beta_t$ must in fact be strongly asymptotically unitarily equivalent.
In particular, if the underlying \cstar-algebra has tracial states, then the rotation maps associated to $\alpha$ and $\beta$ must coincide; see \cite[Subsection 5.7]{Kishimoto03} and \cite{Connes81, Kishimoto98II, KishimotoKumjian01, Lin09}.

If we fix parameters $\theta\in [0,1]\setminus\IQ$ and $\rho_1, \rho_2\in\IR$, we may consider the flow $\gamma: \IR\curvearrowright A_\theta=\cstar(u,v)$ on the irrational rotation algebra given by $\gamma_t(u)=e^{2\pi i \rho_1 t}u$ and $\gamma_t(v)=e^{2\pi i \rho_2 t} v$.
Although it follows from classification \cite{ElliottEvans93, Elliott93} that the individual automorphisms given by $\gamma$ are approximately inner, the rotation map of $\gamma$ can only be trivial when $\rho_1=\rho_2=0$.
In particular, this yields a natural class of counterexamples to the (false) implication discussed above.
\end{rem}


\section{$\CO_\infty$-absorbing \cstar-algebras}

In this section, we will prove Theorem \ref{Theorem-B} with the help of Theorem \ref{thm:main-result}.
This requires a two-step approach; an analysis of the approximately central exponential length of the \cstar-algebras under consideration, followed by a uniqueness theorem for the restricted coactions of flows up to approximate unitary equivalence.
We remark that the class of separable nuclear $\CO_\infty$-absorbing \cstar-algebras has been classified by Kirchberg in unpublished work \cite{KirchbergC}, which has been rigorously reproved by Gabe \cite{Gabe19}.

\subsection{Approximately central exponential length}

In this subsection, we prove that the approximately central exponential length of any tensor product with $\CO_\infty$ is finite; in fact its value is at most $2\pi$.
This is analogous to Phillips' optimal bound on the exponential length of such \cstar-algebras \cite{Phillips02} and the proof below imports some of his techniques. 
Other important ingredients, which take care of approximate centrality, are arguments due to Haagerup--R{\o}rdam \cite{HaagerupRordam95} and Nakamura \cite{Nakamura00}, where in turn Nakamura's argument originates in other work of Phillips \cite{Phillips97}.

The proof of the following lemma is essentially contained in the proof of \cite[Lemma 5.1]{HaagerupRordam95}. For the reader's convenience, we shall give the argument in its distilled form.

\begin{lemma} \label{lem:Haagerup-Rordam}
Let $A$ be a unital \cstar-algebra and $z\in\CU(A)$ a unitary. Then there exists a unitary path $u: [0,1]\to\CU(A\otimes\CO_2)$ satisfying
\[
u(0)=\eins,\quad u(1)=z\otimes\eins,\quad \ell(u)\leq\frac{8\pi}{3},
\]
and
\[
\max_{0\leq t\leq 1}~ \|[u(t),a\otimes\eins]\| \leq 4\|[z,a]\| \quad\text{for all } a\in A.
\]
\end{lemma}
\begin{proof}
Let $s_1, s_2\in\CO_2$ be the two generating isometries. Consider two unital $*$-homomorphisms $\theta_1, \theta_2: M_3\to\CO_2$ given by
\[
\theta_1(e_{11}) = s_1s_1^*,\ \theta_1(e_{21}) = s_2s_1s_1^*,\ \theta_1(e_{31})=s_2^2s_1^*
\]
and
\[
\theta_2(e_{11}) = s_2s_2^*,\ \theta_2(e_{21}) = s_1s_2s_2^*,\ \theta_2(e_{31})=s_1^2s_2^*.
\]
We consider the permutation matrix 
\[
w=e_{21}+e_{32}+e_{13} \ \in \ M_3,
\]
whose set of eigenvalues is $\set{1, e^{\pm 2\pi i/3}}$. Thus $w=\exp(ih)$ for a self-adjoint element $h\in M_3$ with $\|h\|= 2\pi/3$. 
For convenience below, we will slightly abuse notation and write $w$ and $h$ also for the elements $\eins\otimes w$ and $\eins\otimes h$ in $A\otimes M_3$.
Consider 
\[
v=z^*\otimes e_{11} + \eins\otimes e_{22} + z\otimes e_{33} \ \in \ A\otimes M_3.
\]
Define $u: [0,1]\to\CU(A\otimes\CO_2)$ via
\[
u(t) = (\id_A\otimes\theta_1)\big( \exp(ith)v\exp(-ith)v^* \big)\cdot (\id_A\otimes\theta_2)\big( \exp(ith)v\exp(-ith)v^* \big).
\]
Clearly $u(0)=\eins$. We compute
\[
\begin{array}{ccl}
u(1) &=& (\id_A\otimes\theta_1)\big( wvw^*v^* \big)\cdot (\id_A\otimes\theta_2)\big( wvw^*v^* \big) \\
&=& (\id_A\otimes\theta_1)\big( z^2\otimes e_{11} + z^*\otimes (e_{22}+e_{33}) \big)\cdot \\
&& \cdot (\id_A\otimes\theta_2)\big( z^2\otimes e_{11} + z^*\otimes (e_{22}+e_{33}) \big) \\
&=& (z^2\otimes s_1s_1^* + z^*\otimes s_2s_2^*)\cdot (z^2\otimes s_2s_2^* + z^*\otimes s_1s_1^*) \\
&=& z\otimes\eins.
\end{array}
\]
Since $\|h\|= 2\pi/3$, the path $[t\mapsto\exp(ith)]$ is $2\pi/3$-Lipschitz, so indeed $\ell(u)\leq 8\pi/3$. As both $\theta_1(\exp(ith))$ and $\theta_2(\exp(ith))$ live in $\eins_A\otimes\CO_2$, we immediately see by the definition of $v$ and $u(t)$ that 
\[
\max_{0\leq t\leq 1}~ \|[u(t),a\otimes\eins]\| \leq 4\|[z,a]\| \quad\text{for all } a\in A.
\]
This finishes the proof.
\end{proof}

The next lemma arises as an adaptation of the proof of \cite[Theorem 7]{Nakamura00}.

\begin{lemma} \label{lem:Nakamura}
Let $A$ be a unital \cstar-algebra. 
Let $\eps>0$ and $\CF\fin A_{\leq 1}$ be a finite subset of contractions. Suppose that $u: [0,1]\to\CU(A)$ is a unitary path satisfying $u(0)=\eins$
and
\[
\max_{a\in\CF}~ \max_{0\leq t\leq 1}~ \|[u(t),a]\|\leq\eps.
\]
Then given any $C>\frac{16\pi}{3}$, there exists another unitary path $w: [0,1]\to\CU(A\otimes\CO_\infty)$ satisfying
\[
w(0)=\eins,\quad w(1)=u(1)\otimes\eins,\quad \ell(w) < C
\]
and
\[
\max_{a\in\CF}~ \max_{0\leq t\leq 1}~ \|[w(t),a\otimes\eins]\|\leq 9\eps.
\]
\end{lemma}
\begin{proof}
Let $C,\eps,\CF, u$ be as in the statement.
Let $\eta>0$ be any number with $\eta < \frac13 \max\set{ 1, \eps, C-\frac{16\pi}{3} }$. 

As $u$ is uniformly continuous, we choose some $n\in\IN$ large enough so that
\begin{equation} \label{eq:uniform-cont-u}
\|u(s)-u(t)\|\leq\eta\quad\text{whenever } |s-t|\leq\frac1n .
\end{equation}
Using basic properties of simple and purely infinite \cstar-algebras from \cite{Cuntz81}, we choose a partition of unity of projections
\[
p_1,\dots,p_{2n+1}\in\CO_\infty \quad\text{with}\quad [p_j] = \begin{cases} 1 &,\quad j \text{ is odd} \\ -1 &,\quad j \text{ is even}. \end{cases}
\]
Here $[p_j]$ denotes the class of $p_j$ in $K_0(\CO_\infty)\cong\IZ$.
We note that, as a consequence, one has $[p_j+p_{j+1}]=0$ for all $j=1,\dots,2n$.
In particular, each corner of the form $(p_j+p_{j+1})\CO_\infty(p_j+p_{j+1})$ admits some unital embedding of $\CO_2$, one of which we shall choose for every $j$ and denote by $\iota_j$.

For every $k=1,\dots,n$, we apply Lemma \ref{lem:Haagerup-Rordam} and find a unitary path $u_k: [0,1]\to\CU(A\otimes\CO_2)$ with
\[
u_k(0)=\eins,\quad u_k(1)=u(k/n)\otimes\eins,\quad \ell(u_k)\leq\frac{8\pi}{3},
\]
and
\[
\max_{0\leq t\leq 1}~ \|[u_k(t),a\otimes\eins]\| \leq 4\|[u(k/n),a]\| \leq 4\eps \quad\text{for all } a\in\CF.
\]
Similarly, apply Lemma \ref{lem:Haagerup-Rordam} for each $k=0,\dots,n-1$ and find a unitary path $v_k': [0,1]\to\CU(A\otimes\CO_2)$ with
\[
v_k'(0)=\eins,\quad v_k'(1)=u(k/n)^*u(1)\otimes\eins,\quad \ell(v_k')\leq\frac{8\pi}{3},
\]
and
\[
\max_{0\leq t\leq 1}~ \|[v_k'(t),a\otimes\eins]\| \leq 4\|[u(k/n)^*u(1),a]\| \leq 8\eps \quad\text{for all } a\in\CF.
\]
Then set $v_k=(u(k/n)\otimes\eins)v_k'$, which is a path in $\CU(A\otimes\CO_2)$ satisfying
\[
v_k(0)=u(k/n)\otimes\eins ,\quad v_k(1)=u(1)\otimes\eins,\quad \ell(v_k)\leq\frac{8\pi}{3},
\]
and
\[
\max_{0\leq t\leq 1}~ \|[v_k(t),a\otimes\eins]\| \leq \max_{0\leq t\leq 1}~ \|[v_k'(t),a\otimes\eins]\| + \|[u(k/n),a]\| \leq 9\eps
\]
for all $a\in\CF$.
Let us now consider two paths of unitaries 
\[
w^{(1)}, w^{(2)}: [0,1]\to\CU(A\otimes\CO_\infty)
\]
via
\[
w^{(1)}(t) = \eins_A\otimes p_1 + \sum_{k=1}^n (\id_A\otimes\iota_{2k})(u_k(t))
\]
and
\[
w^{(2)}(t) = u(1)\otimes p_{2n+1} + \sum_{k=0}^{n-1} (\id_A\otimes\iota_{2k+1})(v_k(t)).
\]
As the $p_j$ form a partition of unity of $\CO_\infty$, it follows from the choice of the maps $\iota_j$ that these are indeed well-defined unitary paths. Our choices of $u_k$ and $v_k$ entail that
\[
w^{(1)}(0)=\eins,\quad w^{(2)}(1)=u(1)\otimes\eins,\quad \ell(w^{(1)})\leq\frac{8\pi}{3},\quad \ell(w^{(2)})\leq\frac{8\pi}{3},
\]
and
\[
\max_{i=1,2}~\max_{0\leq t\leq 1}~ \|[w^{(i)}(t),a\otimes\eins]\|\leq 9\eps \quad\text{for all } a\in\CF.
\]
Let us now compare the remaining endpoints for these paths. We have
\[
w^{(1)}(1) = \eins_A\otimes p_1 + \sum_{k=1}^n u(k/n)\otimes (p_{2k}+p_{2k+1})
\]
and
\[
w^{(2)}(0) = u(1)\otimes p_{2n+1} + \sum_{k=0}^{n-1} u(k/n)\otimes (p_{2k+1}+p_{2(k+1)}).
\]
By \eqref{eq:uniform-cont-u}, we thus have $\|w^{(1)}(1)-w^{(2)}(0)\|\leq\eta$. 
As $\eta<\frac13$, we may find a unitary path $w': [0,1]\to\CU(A\otimes\CO_\infty)$ with
\[
w'(0) = w^{(1)}(1),\quad w'(1)=w^{(2)}(0) \quad\text{and}\quad \ell(w')\leq 3\eta < C-\frac{16\pi}{3}.
\]
Note that
\[
\|[w^{(1)}(1),a\otimes\eins]\|\leq \max_{k=1,\dots,n}~ \|[u(k/n),a]\| \leq \eps \quad\text{for all } a\in\CF
\]
by our initial assumption. 
Thus we have
\[
\max_{0\leq t\leq 1}~ \|[w'(t),a\otimes\eins]\|\leq \eps+2\ell(w') \leq 3\eps \quad\text{for all } a\in\CF.
\]
We may thus obtain the desired unitary path $w: [0,1]\to\CU(A\otimes\CO_\infty)$ as the path composition of $w^{(1)}$, $w'$ and $w^{(2)}$. 
\end{proof}

It would be fairly simple to show finite approximately central exponential length as a straightforward consequence of Lemma \ref{lem:Nakamura}.
However, the resulting immediate upper bound would be at least $\frac{16\pi}{3}$, which is far from optimal.
For completeness, we will insert a few additional arguments that show the optimal bound of $2\pi$.
Readers who do not wish to bother with the precise value may consider to skip ahead and continue reading from Theorem \ref{thm:fin-acel-Oinf} onward.

Recall the following fact, which is shown as part of the proof of \cite[Theorem 3.1]{Phillips02}.

\begin{lemma} \label{lem:Phillips}
Let $A$ be a unital \cstar-algebra and let $u\in\CU_0(A)$ with $\sp(u)=\IT$. Let $u_f\in\CU(\CO_\infty)$ be a unitary with $\sp(u_f)=\IT$.
Then for every $\eps>0$, there exists a self-adjoint element $h\in A\otimes\CO_\infty$ with 
\[
\|h\|\leq\pi \quad\text{and}\quad \|\exp(ih)-u\otimes u_f\|\leq\eps.
\]
\end{lemma}

\begin{lemma} \label{lem:fin-acel-Oinf-pre}
For every $\eps>0$, there exists $\delta>0$ satisfying the following property:

Let $A$ be a unital \cstar-algebra.
Suppose that $\CF\fin A_{\leq 1}$ is any finite set of contractions and $u: [0,1]\to\CU(A)$ is a unitary path satisfying $u(0)=\eins$ and
\[
\max_{a\in\CF}~ \max_{0\leq t\leq 1}~ \|[u(t),a]\| \leq \delta.
\]
Suppose that $\sp(u(1))\subseteq\IT$ is $\delta$-dense in the Euclidean metric. Then there exists a unitary $u_f\in\CU(\CO_\infty)$ with full spectrum $\sp(u_f)=\IT$ and a self-adjoint element $h\in A\otimes\CO_\infty$ with
\[
\|h\|\leq\pi ,\quad \max_{a\in\CF}~ \|[h,a\otimes\eins]\|\leq\eps
\]
and
\[
\|u(1)\otimes u_f-\exp(ih)\|\leq\eps.
\]
\end{lemma}
\begin{proof}
Suppose that the claim were false.
Then there exists some $\eps>0$ and a decreasing null sequence $\delta_n\to 0$, unital \cstar-algebras $A^{(n)}$, finite sets $\CF_n\fin A^{(n)}$ of contractions and unitary paths $u^{(n)}: [0,1]\to\CU(A^{(n)})$ with $u^{(n)}(0)=\eins$ so that $\sp(u^{(n)}(1))\subseteq\IT$ is $\delta_n$-dense and
\[
\max_{a\in\CF_n}~ \max_{0\leq t\leq 1}~ \|[u^{(n)}(t),a]\| \leq \delta_n,
\]
and such that the conclusion of the Lemma fails. We are going to lead this to a contradiction.

First, we remark that since one has the identity
\[
\|[v,b]\|=\|[v^*,b]\|=\|[v,b^*]\|
\]
for arbitrary elements $b$ and unitaries $v$ in a \cstar-algebra, we may as well assume that $\CF_n=\CF_n^*$ for all $n$.

Let $C:=\frac{17\pi}{3}$. For every $n$, we apply Lemma \ref{lem:Nakamura} and find unitary paths $w^{(n)}: [0,1]\to\CU(A^{(n)}\otimes\CO_\infty)$ satisfying
\[
w^{(n)}(0)=\eins,\quad w^{(n)}(1)=u^{(n)}(1)\otimes\eins,\quad \ell(w^{(n)}) \leq C
\]
and
\[
\max_{a\in\CF_n}~ \max_{0\leq t\leq 1}~ \|[w^{(n)}(t),a\otimes\eins]\|\leq 9\delta_n.
\]
By passing to the arc-length parametrization, if necessary, we may assume that each path $w^{(n)}$ is $C$-Lipschitz.

We consider
\[
A^\sharp := \prod_{n=1}^\infty A^{(n)}\otimes\CO_\infty / \bigoplus_{n=1}^\infty A^{(n)}\otimes\CO_\infty
\]
and
\[
A^{\sharp\sharp} := \prod_{n=1}^\infty A^{(n)}\otimes\CO_\infty\otimes\CO_\infty / \bigoplus_{n=1}^\infty A^{(n)}\otimes\CO_\infty\otimes\CO_\infty
\]
Then we can identify $A^\sharp\subset A^{\sharp\sharp}$ in a natural way such that $A^{\sharp\sharp}\cap (A^\sharp)'$ contains a unital copy of $\CO_\infty$, i.e., we also have an inclusion $A^\sharp\otimes\CO_\infty\subset A^{\sharp\sharp}$.

Let $S\subset A^\sharp$ be the set of elements represented by all sequences that have the form $(s_1,s_2,s_3,\dots)$ with $s_n\in\CF_n\otimes\eins_{\CO_\infty}$. 
As we have assumed $\CF_n=\CF_n^*$ for all $n$ and that it consists of contractions, we see that $S$ is a well-defined self-adjoint set, and thus the relative commutant $A^\sharp\cap S'$ (as well as $A^{\sharp\sharp}\cap S'$) is a unital \cstar-algebra.

Consider
\[
W: [0,1]\to\CU(A^\sharp) \quad\text{via}\quad W(t)=[(w^{(1)}(t),w^{(2)}(t),w^{(3)}(t),\dots)].
\]
Then by the choice of $w^{(n)}$, we see that this is a well-defined continuous path of unitaries whose image is in fact in $A^\sharp\cap S'$. In particular, as $W(0)=\eins$ we have that
\[
U:= W(1) = [(u^{(1)}(1)\otimes\eins, u^{(2)}(1)\otimes\eins,\dots)] \in \CU_0(A^\sharp\cap S').
\]
At the same time, our assumptions on the spectra of $u^{(n)}(1)$ implies that $U$ has full spectrum $\sp(U)=\IT$. Let $u_f\in\CU(\CO_\infty)$ be a unitary with full spectrum. Then by applying Lemma \ref{lem:Phillips}, we find a self-adjoint element
\[
H\in (A^\sharp\cap S')\otimes\CO_\infty \subset A^{\sharp\sharp}\cap S'
\]
with $\|H\|\leq\pi$ and $\|\exp(iH)-U\otimes u_f\|\leq\eps/2$. Write
\[
H=[(h_1,h_2,h_3,\dots)] \quad\text{for}\quad h_n=h_n^*\in A^{(n)}\otimes\CO_\infty\otimes\CO_\infty \text{ with } \|h_n\|\leq\pi.
\]
Note that $\CO_\infty\cong\CO_\infty\otimes\CO_\infty$, and the unitary $\eins\otimes u_f$ can thus be identified with some unitary in $\CO_\infty$, and each self-adjoint element $h_n\in A\otimes\CO_\infty\otimes\CO_\infty$ can thus be identified with some self-adjoint element in $A\otimes\CO_\infty$.
By our choice of the sequence of tuples $(\delta_n, A^{(n)}, \CF_n, u^{(n)})$, it follows that for every $n$, there exists $a_n\in\CF_n$ such that either
\[
\|[h_n,a_n\otimes\eins\otimes\eins]\|>\eps
\]
or
\[
\|\exp(ih_n)-u^{(n)}(1)\otimes\eins\otimes u_f\|>\eps.
\]
But then $a:=[(a_1,a_2,\dots)]\in S$ and thus
\[
0=\|[a,H]\|=\limsup_{n\to\infty}~ \|[a_n\otimes\eins\otimes\eins,h_n]\|
\]
rules out the first condition for all but finitely many $n$. Moreover, the estimate
\[
\eps/2 \geq \|\exp(iH)-U\otimes u_f\| = \limsup_{n\to\infty}~ \|\exp(ih_n)-u^{(n)}(1)\otimes\eins\otimes u_f\|
\]
rules out the second condition for all but finitely many $n$. This leads to a contradiction. Hence our claim is true.
\end{proof}

As the following Lemma follows easily with functional calculus, we omit the proof.

\begin{lemma}[cf.~{\cite[Proposition 1.1]{TomsWinter13}}] \label{lem:commutators}
Let $K\subset\IC$ be a compact subset and $f: K\to\IC$ a continuous function. For every $\eps>0$, there exists $\eta>0$ with the following property:
Whenever $x\in A$ is a normal element in a \cstar-algebra $A$ with $\sp(x)\subseteq K$, and $y\in A$ is a contraction with $\|[x,y]\|\leq\eta$ and $\|[x^*,y]\|\leq\eta$, then $\|[f(x),y]\|\leq\eps$.
\end{lemma}

\begin{lemma} \label{lem:fin-acel-Oinf}
For every $\eps>0$, there exists $\delta>0$ satisfying the following property:

Let $A$ be a \cstar-algebra with $A\cong A\otimes\CO_\infty$.
Suppose that $\CF\fin A_{\leq 1}$ is any finite set of contractions and $u: [0,1]\to\CU(\eins+A)$ is a unitary path satisfying $u(0)=\eins$ and
\[
\max_{a\in\CF}~ \max_{0\leq t\leq 1}~ \|[u(t),a]\| \leq \delta.
\]
Then there exists a unitary path $v: [0,1]\to\CU(\eins+A)$ satisfying
\[
v(0)=\eins,\quad v(1)=u(1),\quad \ell(v)\leq 2\pi+\eps,
\]
and
\[
\max_{a\in\CF}~ \max_{0\leq t\leq 1}~ \|[v(t),a]\| \leq \eps.
\]
\end{lemma}
\begin{proof}
First we remark that it suffices for the claim (up to replacing $\eps$ by $\eps/3$) to prove that a path $v$ as above exists satisfying
\[
 \|v(1)-u(1)\|\leq\eps,\quad \ell(v)\leq 2\pi
\]
instead of
\[
 v(1)=u(1),\quad \ell(v)\leq 2\pi+\eps.
\]
This follows from a standard perturbation argument. So let us show this modified claim.

Let $\eps>0$ be given. Choose $\eta>0$ as in Lemma \ref{lem:commutators} for $K_1=[-\pi,\pi]$ and $f_2=\exp(i\cdot \_ )$ and $\eps/3$ in place of $\eps$. Assume $\eta\leq\eps/3$ without loss of generality.
Now choose $\delta_1>0$ as in Lemma \ref{lem:fin-acel-Oinf-pre} for $\eta$ in place of $\eps$. 
Consider $K_2=\set{ z\in\IT \mid |z+1|\geq\delta_1 }$ and let $f_2: K_2\to [-\pi,\pi]$ be the continuous branch of the logarithm given by $f_2(e^{it})=t$. Choose $\delta>0$ as in Lemma \ref{lem:commutators} for $K_2$, $f_2$ and $\eta$ in place of $\eps$. We may assume $\delta\leq\delta_1$.
We claim that this $\delta$ is as desired.

Let $A$ be a \cstar-algebra with $A\cong A\otimes\CO_\infty$.
Suppose that $\CF\fin A$ is any finite set of contractions and $u: [0,1]\to\CU(\eins+A)$ is a unitary path satisfying $u(0)=\eins$ and
\[
\max_{a\in\CF}~ \max_{0\leq t\leq 1}~ \|[u(t),a]\| \leq \delta.
\]
Then we may consider two cases. 

\textbf{Case 1:}
Suppose that the spectrum of $u(1)$ is not $\delta_1$-dense in $\IT$.\footnote{This case will only use functional calculus and not the assumption that $A\cong A\otimes\CO_\infty$.}
We find a scalar $\lambda\in\IT$ such that the spectrum of $u'=\lambda u(1)\in\CU(A^+)$ is a subset of $K_2$. 
Then $h=f_2(u')$ is a self-adjoint element with $\|h\|\leq\pi$ and $u'=\exp(ih)$. 
By our choice of $\delta$, we have $\|[h,a]\|\leq\eta$. Consider
\[
v': [0,1]\to\CU(A^+),\quad v'(t)=\exp(ith).
\]
By our choice of $\eta$, it follows that
\[
\max_{a\in\CF}~\max_{0\leq t\leq 1}~ \|[v'(t),a]\| \leq \eps/3\leq \eps.
\]
Moreover, it is clear that $v'(0)=\eins$ and $\ell(v')\leq\pi$. Let $\chi: A^+\to\IC$ denote the canonical character.
Then
\[
v: [0,1]\to\CU(\eins+A),\quad \chi(v'(t)^*)v'(t)
\]
defines a unitary path with
\[
v(0)=\eins,\quad v(1)=u(1),\quad \ell(v)\leq 2\pi
\]
and
\[
\max_{a\in\CF}~\max_{0\leq t\leq 1}~ \|[v(t),a]\| \leq \eps.
\]
This concludes the proof for the first case.

\textbf{Case 2:} Suppose that $\sp(u(1))\subseteq\IT$ is $\delta_1$-dense. 
As $\eins+A\subset A^+$ and $\delta\leq\delta_1$, we can thus apply the statement in Lemma \ref{lem:fin-acel-Oinf-pre} and find a unitary $u_f\in\CU(\CO_\infty)$ with full spectrum and a self-adjoint element $h_1\in A^+\otimes\CO_\infty$ such that
\[
\|h_1\|\leq\pi,\quad \|\exp(ih_1)-u(1)\otimes u_f\|\leq\eta,\quad \max_{a\in\CF}~ \|[h_1,a\otimes\eins]\|\leq\eta.
\]
Let $\chi: A^+\to\IC$ be the canonical character and set 
\[
h_2\:=-(\chi\otimes\id_{\CO_\infty})(h_1) \ \in \ \IC\otimes\CO_\infty \subset A^+\otimes\CO_\infty.
\]
Then we define
\[
v': [0,1]\to\CU(\eins+A\otimes\CO_\infty) \text{ via } v'(t)=\exp(ith_1)\exp(ith_2).
\]
By the definition of $h_2$, we see that this is a well-defined unitary path with
\[
v'(0)=\eins \quad\text{and}\quad \|v'(1)-u(1)\otimes\eins\|\leq 2\eta.
\]
As $\|h_1\|, \|h_2\|\leq\pi$, we see that $\ell(v')\leq 2\pi$.
Moreover, observe that as $[h_2,a\otimes\eins]=0$ for all $a\in A$, our choice of $\eta$ according to Lemma \ref{lem:commutators} implies
\[
\|[v'(t),a\otimes\eins]\| = \|[\exp(ith_1),a\otimes\eins]\|\leq\eps/3
\]
for all $t\in [0,1]$ and $a\in\CF$.

As $A\cong A\otimes\CO_\infty$, it follows from \cite[Remark 2.7]{TomsWinter07} that there is a $*$-homomorphism $\psi: A\otimes\CO_\infty\to A$ satisfying
\[
\|\psi(x\otimes\eins)-x\|\leq\eps/3 \quad\text{for all } x\in\CF\cup\set{u(1)-\eins}.
\]
We then define $v: [0,1]\to\CU(\eins+A)$ via $v(t)=\psi^+(v'(t))$. By the construction so far, we have $v(0)=\eins$, $\ell(v)\leq 2\pi$, and it follows that
\[
\|v(1)-u(1)\|\leq\|u(1)-\psi^+(u(1)\otimes\eins)\|+\|v'(1)-u(1)\otimes\eins\|\leq 2\eps/3+\eps/3\leq \eps
\]
and
\[
\max_{a\in\CF}~ \max_{0\leq t\leq 1}~ \|[v(t),a]\| \leq \max_{a\in\CF}~\Big( 2\|a-\psi(a\otimes\eins)\|+\max_{0\leq t\leq 1}~ \|[v'(t),a\otimes\eins]\| \Big) \leq \eps.
\]
This concludes the proof for the second case, and finishes the proof.
\end{proof}

\begin{theorem} \label{thm:fin-acel-Oinf}
For every \cstar-algebra $A$ with $A\cong A\otimes\CO_\infty$, the approximately central exponential length of $A$ is at most $2\pi$.
\end{theorem}
\begin{proof}
This follows directly from Lemma \ref{lem:fin-acel-Oinf}.
\end{proof}


\subsection{Uniqueness for restricted coactions}

In this subsection, we prove the following result, which completes the second step in our approach to obtain Theorem \ref{Theorem-B} indicated at the beginning of this section.

\begin{theorem} \label{thm:spi-flow-uniqueness}
Let $A$ be a separable, nuclear, $\CO_\infty$-absorbing \cstar-algebra.
Let $\alpha, \beta: \IR\curvearrowright A$ be two flows inducing the same actions on $\prim(A)$.
Then the restricted coactions $\alpha_\co$ and $\beta_\co$ are approximately unitarily equivalent.
\end{theorem}

For this, we need to consider a more general uniqueness theorem for nuclear maps into $\CO_\infty$-absorbing \cstar-algebras, taking into account the following notion of equivalence.
For a given $*$-homomorphism of the form $\Phi: A\to\CC( [0,1], B)$, we denote for each $s\in [0,1]$ the fibre map $\Phi_s=\ev_s\circ\Phi: A\to B$.

\begin{defi}[cf.\ \cite{KirchbergRordam05}]
Two non-degenerate $*$-homomorphisms $\phi_0, \phi_1: A\to B$ between \cstar-algebras are called ideal-preservingly homotopic through non-degenerate maps, if there exists a non-degenerate $*$-homomorphism $\Phi: A\to \CC([0,1], B)$ with $\Phi_0=\phi_0$, $\Phi_1=\phi_1$, and $\overline{B\phi_0(a)B}=\overline{B\Phi_t(a)B}$ for all $a\in A$ and all $t\in [0,1]$.
\end{defi}

Since the \cstar-algebras treated in Theorem \ref{Theorem-B} are $\CO_\infty$-absorbing, the following reduction argument for homotopic maps will be useful.
This is an observation of Phillips, which is essentially contained in the proof of \cite[Theorem 3.4]{Phillips97}, although I do not know whether it has ever been stated in this generality in the literature before.
We will give the short argument for the reader's convenience, which is similar to the one in Lemma \ref{lem:Nakamura}.

\begin{lemma} \label{lem:homotopy-reduction}
Let $A$ and $B$ be two \cstar-algebras.
Let $\Phi: A\to\CC\big( [0,1], B \big)$ be a $*$-homomorphism, and denote the endpoint maps by $\phi_i=\ev_i\circ\Phi$ for $i=0,1$.
Suppose that $\Phi_s\otimes\eins_{\CO_2}\ue\Phi_t\otimes\eins_{\CO_2}$ as $*$-homomorphisms from $A$ to $B\otimes\CO_2$, for all $s,t\in [0,1]$.
Then $\phi_0\otimes\eins_{\CO_\infty}\ue\phi_1\otimes\eins_{\CO_\infty}$ as $*$-homomorphisms from $A$ to $B\otimes\CO_\infty$.
\end{lemma}
\begin{proof}
Let $\CF\fin A_{\leq 1}$ and $\eps>0$ be given.
We choose $n\in\IN$ large enough so that
\begin{equation} \label{eq:ideal-homotopy}
\max_{a\in\CF}~ \|\Phi_t(a)-\Phi_s(a)\|\leq\eps/3 \quad\text{whenever } |s-t|\leq\frac1n.
\end{equation}
As in the proof of Lemma \ref{lem:Nakamura}, we choose a partition of unity of projections
\[
p_1,\dots,p_{2n+1}\in\CO_\infty \quad\text{with}\quad [p_j] = \begin{cases} 1 &,\quad j \text{ is odd} \\ -1 &,\quad j \text{ is even}. \end{cases}
\]
Here $[p_j]$ denotes the class of $p_j$ in $K_0(\CO_\infty)\cong\IZ$.
As a consequence, one has $[p_j+p_{j+1}]=0$ for all $j=1,\dots,2n$.
In particular, each corner of the form $(p_j+p_{j+1})\CO_\infty(p_j+p_{j+1})$ admits some unital embedding of $\CO_2$, one of which we shall choose for every $j$ and denote by $\iota_j$.

By assumption, we may choose unitaries $v_1,\dots,v_n\in\CU(\eins+B\otimes\CO_2)$ satisfying
\begin{equation} \label{eq:O2-vk}
\max_{k=1,\dots,n}~ \max_{a\in\CF}~ \|v_k(\phi_0(a)\otimes\eins)v_k^*-\Phi_{k/n}(a)\otimes\eins\|\leq\eps/3,
\end{equation}
as well as unitaries $u_0,\dots,u_{n-1}\in\CU(\eins+B\otimes\CO_2)$ satifying
\begin{equation} \label{eq:O2-uk}
\max_{k=0,\dots,n-1}~ \max_{a\in\CF}~ \|u_k(\Phi_{k/n}(a)\otimes\eins)u_k^*-\phi_1(a)\otimes\eins\|\leq\eps/3.
\end{equation}
We then set
\[
V = \eins\otimes p_1 + \sum_{k=1}^n (\id_B\otimes\iota_{2k})(v_k)
\]
and
\[
U = \eins\otimes p_{2n+1} + \sum_{k=0}^{n-1} (\id_B\otimes\iota_{2k+1})(u_k).
\]
Note that in these instances we extend the map $(\id_B\otimes\iota_{j})$ to a map from $(B\otimes\CO_2)^+$ to $B^+\otimes\CO_\infty$ by sending the unit to the projection $\eins\otimes (p_j+p_{j+1})$.
As the $p_j$ form a partition of unity of $\CO_\infty$, it follows from the choice of the maps $\iota_j$ that $U$ and $V$ are indeed well-defined unitaries in $\CU(\eins+B\otimes\CO_\infty)$.
For the unitary $W=UV$, we then compute for all $a\in\CF$ that
\[
\begin{array}{cl}
\multicolumn{2}{l}{ W(\phi_0(a)\otimes\eins)W^* }\\
=& \dst UV\Big( \phi_0(a)\otimes p_1 + \sum_{k=1}^n \phi_0(a)\otimes (p_{2k}+p_{2k+1}) \Big)V^*U^* \\
=& \dst U\Big( \phi_0(a)\otimes p_1 + \sum_{k=1}^n (\id_B\otimes\iota_{2k})\big( v_k(\phi_0(a)\otimes\eins)v_k^* \big) \Big)U^* \\
\stackrel{\eqref{eq:O2-vk}}{=}_{\makebox[0pt]{\footnotesize $\eps/3$}} & \dst U\Big( \Phi_0(a)\otimes p_1 + \sum_{k=1}^n \Phi_{k/n}(a)\otimes (p_{2k}+p_{2k+1}) \Big)U^* \\
 \stackrel{\eqref{eq:ideal-homotopy}}{=}_{\makebox[0pt]{\footnotesize $\eps/3$}} & \dst U\Big( \Phi_1(a)\otimes p_{2n+1} + \sum_{k=0}^{n-1} \Phi_{k/n}(a)\otimes (p_{2k+1}+p_{2(k+1)}) \Big)U^* \\
 = & \dst \Phi_1(a)\otimes p_{2n+1} + \sum_{k=0}^{n-1} (\id_B\otimes\iota_{2k+1})\big( u_k(\Phi_{k/n}(a)\otimes\eins)u_k^* \big) \\
 \stackrel{\eqref{eq:O2-uk}}{=}_{\makebox[0pt]{\footnotesize $\eps/3$}} & \dst \phi_1(a)\otimes p_{2n+1} + \sum_{k=0}^{n-1} \phi_1(a)\otimes (p_{2k+1}+p_{2(k+1)}) \ = \ \phi_1(a)\otimes\eins.
\end{array}
\]
In particular, we see
\[
\max_{a\in\CF}~ \|W(\phi_0(a)\otimes\eins)W^*-\phi_1(a)\otimes\eins\| \leq \eps.
\]
As $\CF\fin A_{\leq 1}$ and $\eps>0$ were arbitrary, this finishes the proof.
\end{proof}

In particular, we now need to consider uniqueness results for maps into $\CO_2$-absorbing \cstar-algebras.
The following is an important intermediate result due to Gabe, which is based on Kirchberg's work on strongly purely infinite \cstar-algebras.
Recall from \cite[Definition 3.4, Observation 3.7]{Gabe18} that two $*$-homomorphisms $\phi,\psi: A\to B$ from a separable \cstar-algebra to another \cstar-algebra are called approximately Murray--von Neumann equivalent, if there is a contraction $v\in B_\infty$ satisfying
\begin{equation} \label{eq:MvN:1}
v^*\phi(a)v=\psi(a) \text{ and } v\psi(a)v^*=\phi(a) \quad \text{for all } a\in A.
\end{equation}

\begin{theorem}[see {\cite[Theorem 3.23]{Gabe18}}] \label{thm:Gabe-uniqueness-1}
Let $A$ and $B$ be \cstar-algebras with $A$ being separable and exact.  
Let $\phi, \psi: A\to B$ be two nuclear $*$-homomorphisms with $\overline{B\phi(a)B}=\overline{B\psi(a)B}$ for all $a\in A$.
Then $\phi\otimes\eins_{\CO_2}$ and $\psi\otimes\eins_{\CO_2}$ are approximately Murray--von Neumann equivalent as maps from $A$ to $B\otimes\CO_2$.
\end{theorem}

We will combine this with the following fact:

\begin{lemma} \label{lem:MvN-vs-au}
Let $A$ and $B$ be two \cstar-algebras with $A$ being separable.
Let $\phi, \psi: A\to B$ be two non-degenerate $*$-homomorphisms.
Then the maps $\phi\otimes\eins_{\CO_2}$ and $\psi\otimes\eins_{\CO_2}$ are approximately Murray--von Neumann equivalent if and only if they are approximately unitarily equivalent.
\end{lemma}
\begin{proof}
Since the ``if'' part is trivial, we shall show the ``only if'' part.
Suppose that $\phi\otimes\eins_{\CO_2}$ and $\psi\otimes\eins_{\CO_2}$ are approximately Murray--von Neumann equivalent.
Using an isomorphism $\CO_2\cong\CO_2\otimes\CO_2$, we may replace $\phi$ by $\phi\otimes\eins_{\CO_2}$ and $\psi$ by $\psi\otimes\eins_{\CO_2}$, and thus assume without loss of generality that $\phi$ and $\psi$ are already approximately Murray--von Neumann equivalent.
We may hence choose a contraction $v\in B_\infty$ satisfying condition \eqref{eq:MvN:1}.
Choose some strictly positive contraction $e\in A$.
As $\phi$ and $\psi$ are non-degenerate, both elements $\phi(e)$ and $\psi(e)$ are strictly positive contractions of $B$.
Using \eqref{eq:MvN:1} with \cite[Lemma 3.8]{Gabe18}, we see that 
\begin{equation} \label{eq:MvN:2}
\phi(a)v=v\psi(a)
\end{equation}
as well as
\begin{equation} \label{eq:MvN:3}
vv^*\phi(a)=\phi(a),\quad v^*v\psi(a)=\psi(a)
\end{equation}
for all $a\in A$.

Let us consider the hereditary \cstar-algebra
\[
\CB = \overline{B\cdot B_\infty\cdot B} \ \subseteq \ B_\infty,
\]
which by assumption coincides with the hereditary subalgebra in $B_\infty$ generated by either $\phi(e)$ or $\psi(e)$.
We also consider its normalizer given by
\[
\CN = \set{ x\in B_\infty \mid x\CB + \CB x \subseteq \CB}.
\]
There is a canonical $*$-homomorphism
\[
m: \CN \to \CM(\CB),\quad x\mapsto m_x \quad\text{with}\quad m_xb=xb,\ bm_x=bx
\]
given by the universal property of the multiplier algebra.
Note that in this context, this map is surjective; see \cite[Proposition 1.5]{BarlakSzabo16}.

Now \eqref{eq:MvN:2} applied to $a=e$ implies $v\in\CN$.
The non-degeneracy of $\phi$ and $\psi$ and \eqref{eq:MvN:3} moreover imply that $m_v\in\CM(\CB)$ is a unitary.
By Lemma \ref{lem:Haagerup-Rordam}, the unitary $m_v\otimes\eins_{\CO_2}\in\CM(\CB)\otimes\CO_2$ is homotopic to the unit.
Note that we may consider the canonical extension to the smallest unitization
\[
m'=(m\otimes\id_{\CO_2})^\sim : \widetilde{\CN\otimes\CO_2} \to \CM(\CB)\otimes\CO_2.
\]
Then $m_v\otimes\eins_{\CO_2}$ lifts under $m'$ to a unitary 
\[
u\in\CU\big( \widetilde{\CN\otimes\CO_2} \big) \ \subseteq \ \CU\big( (\widetilde{B\otimes\CO_2})_\infty \big).
\]
In particular, $u$ may be represented by a sequence of unitaries $u_n\in\widetilde{B\otimes\CO_2}$.
Then $m'_{v\otimes\eins} = m'_u$ means $ub=(v\otimes\eins)b$ and $bu=b(v\otimes\eins)$ for all $b\in B\otimes\CO_2$, which in particular implies
\[
\phi(a)\otimes\eins_{\CO_2} \stackrel{\eqref{eq:MvN:1}}{=} v\psi(a)v^* \otimes \eins_{\CO_2} = u(\psi(a)\otimes\eins_{\CO_2})u^* = \lim_{n\to\infty} u_n(\psi(a)\otimes\eins_{\CO_2})u_n^*.
\]
This finishes the proof.
\end{proof}

\begin{cor} \label{cor:Gabe-uniqueness-2}
Let $A$ and $B$ be \cstar-algebras with $A$ being separable and exact. 
Let $\phi, \psi: A\to B$ be two non-degenerate, nuclear $*$-homomorphisms with $\overline{B\phi(a)B}=\overline{B\psi(a)B}$ for all $a\in A$.
Then $\phi\otimes\eins_{\CO_2}$ and $\psi\otimes\eins_{\CO_2}$ are approximately unitarily equivalent as maps from $A$ to $B\otimes\CO_2$.
\end{cor}
\begin{proof}
Combine Theorem \ref{thm:Gabe-uniqueness-1} and Lemma \ref{lem:MvN-vs-au}.
\end{proof}

\begin{cor} \label{cor:homotopy-uniqueness}
Let $A$ be a separable, exact \cstar-algebra and $B$ a separable, nuclear, $\CO_\infty$-absorbing \cstar-algebra.
Suppose that $\phi, \psi: A\to B$ are two non-degenerate $*$-homomorphisms that are ideal-preservingly homotopic through non-degenerate maps.
Then $\phi$ and $\psi$ are approximately unitarily equivalent.
\end{cor}
\begin{proof}
By assumption (and \cite[Remark 2.7]{TomsWinter07}), there exist $*$-homomorphisms from $B\otimes\CO_\infty$ to $B$ mapping finitely many elements of the form $b\otimes\eins$ approximately to $b$.
By a standard argument it thus suffices to show that $\phi\otimes\eins_{\CO_\infty}$ and $\psi\otimes\eins_{\CO_\infty}$ are approximately unitarily equivalent as $*$-homomorphisms from $A$ to $B\otimes\CO_\infty$.
The claim then follows directly from combining Lemma \ref{lem:homotopy-reduction} and Corollary \ref{cor:Gabe-uniqueness-2}.
\end{proof}

We can now prove the desired uniqueness result for restricted coactions of flows on $\CO_\infty$-absorbing \cstar-algebras:

\begin{proof}[Proof of Theorem \ref{thm:spi-flow-uniqueness}]
Consider the \cstar-algebra $B=\CC([0,1],A)$, which is also separable, nuclear, and $\CO_\infty$-absorbing.
We consider the point-norm continuous path of automorphisms $\kappa: [0,1]\to\Aut(B)$ given by
\[
\kappa_s(f)(t) = (\beta_{st}\circ\alpha_{-st})(f(t)).
\]
Notice that $\kappa_0=\id_B$ and $\kappa_1\circ\alpha_\co = \beta_\co$.

Since arbitrary ideals of $A$ correspond to open subsets of $\prim(A)$, we have $\alpha_t(J)=\beta_t(J)$ for all $t\in\IR$ and all closed, two-sided ideals $J$ in $A$.
This in turn implies that $\kappa_s(I)=I$ for all $s\in [0,1]$ and all closed, two-sided ideals $I$ in $B$. 
In other words, $\set{\kappa_s}_{s\in [0,1]}$ is a continuous family of automorphisms connecting $\id_B$ to $\kappa_1$ through ideal-preserving automorphisms, and in particular yields an ideal-preserving homotopy through non-degenerate maps.
From Corollary \ref{cor:homotopy-uniqueness} it thus follows that $\kappa_1$ is approximately inner. 
In particular, we see that $\alpha_\co$ and $\beta_\co$ are approximately unitarily equivalent.
\end{proof}

\begin{rem}
We note that in the classification theory of $\CO_\infty$-absorbing \cstar-algebras \cite{Gabe19}, there is a more general and stronger uniqueness theorem using ideal-related $KK$-theory, which by definition is an ideal-preserving homotopy invariant; see also \cite{KirchbergC} and \cite[Hauptsatz 4.2]{Kirchberg00}.
However, there is a priori a problem when trying to obtain Corollary \ref{cor:homotopy-uniqueness} as a direct consequence given the fact that the \cstar-algebra $B$ in this context is usually assumed to be stable in these references.
The stability assumption bypasses all considerations related to non-degeneracy, which are in general necessary.
For example, it is easy to write down examples of two maps between unital \cstar-algebras where one of them is unital, the other one is not, but they have the same invariant.
I do not know in general whether two non-degenerate $*$-homomorphisms giving the same ideal-related $KK$-element must always be approximately unitarily equivalent as in Corollary \ref{cor:homotopy-uniqueness}.
\end{rem}


\subsection{Proof of Theorem \ref{Theorem-B}}

Using the results we have collected so far in this section, we can now obtain our main result about flows on $\CO_\infty$-absorbing \cstar-algebras:

\begin{proof}[Proof of Theorem {\em\ref{Theorem-B}}]
Let $A$ be a separable, nuclear, $\CO_\infty$-absorbing \cstar-algebra. 
Let $\alpha, \beta: \IR\curvearrowright A$ be two flows satisfying the Rokhlin property. 
Suppose that the induced actions of $\alpha$ and $\beta$ agree on $\prim(A)$.
Then $\alpha_\co$ and $\beta_\co$ are approximately unitarily equivalent by Theorem \ref{thm:spi-flow-uniqueness}.

By Theorem \ref{thm:fin-acel-Oinf}, $A$ has finite approximately central exponential length. 
It also has finite exponential length due to \cite{Phillips02}, and thus finite weak inner length by Remark \ref{rem:celw-cel}.
Thus Theorem \ref{thm:main-result} implies that $\alpha$ and $\beta$ are (norm-strongly) cocycle conjugate via an approximately inner automorphism on $A$. 
This shows our claim.
\end{proof}

\begin{rem}
It appears to be an obvious question whether it is possible to improve Theorem \ref{Theorem-B} to a strong classification result. 
That is, let $A$ be a separable, nuclear, $\CO_\infty$-absorbing \cstar-algebra with two Rokhlin flows $\alpha, \beta: \IR\curvearrowright A$. 
Suppose that there exists a homeomorphism $\Theta$ on $\prim(A)$ that yields a conjugacy between the induced $\IR$-actions of $\alpha$ and $\beta$. 
Does $\Theta$ lift to an automorphism inducing cocycle conjugacy between $\alpha$ and $\beta$?
Theorem \ref{Theorem-B} at least reduces this question to the much weaker question whether $\Theta$ lifts to any automorphism on $A$.

In general, this is not possible, as already the simple example $A=\CO_2\oplus\CO_\infty$ shows, where we equip $A$ with a Rokhlin flow on each summand and $\Theta$ is the flip map on the two-point set, which represents the prime ideal space.
However, a homeomorphism $\Theta$ as above could in principle exist when it is assumed to preserve compact-open subsets of $\prim(A)$, but I do not know of any existence result of this kind.

We point out that in the special case where $A\cong A\otimes\CO_2\otimes\CK$, the prime ideal space constitutes the entire classification invariant.
So using Kirchberg's results, a lifting of $\Theta$ indeed exists; see Gabe's proof \cite{Gabe18}  of the $\CO_2$-stable classification.
In particular, we observe that the improved version of Theorem \ref{Theorem-B} holds for $\CO_2$-absorbing, stable \cstar-algebras.

In general, Kirchberg's classification allows one to turn Theorem \ref{Theorem-B} into a strong classification result for stable \cstar-algebra, but at the price of enlarging the invariant.
The underlying object of the invariant would still be $X:=\prim(A)$, but we would need to equip it with a Hom-set consisting of all the invertible, ideal-related $KK$-elements in $KK(X; A, A)$.
Since the ideal-related $KK$-theory is virtually impossible to compute, however, it is unclear how useful this variant of Theorem \ref{Theorem-B} could be in applications.
\end{rem}

\begin{rem}
As can be easily seen from the proof of Lemma \ref{lem:homotopy-reduction}, the given argument is in a sense universal and can be generalized in various ways.
One obvious direction is the equivariant setting.
That is, if we equip the \cstar-algebras $A$ and $B$ in the statement with actions $\alpha: G\curvearrowright A$ and $\beta: G\curvearrowright B$ of a locally compact group $G$, then the analogous claim is true with respect to equivariant $*$-homomorphisms and approximate $G$-unitary equivalence; see \cite[Definition 2.1(i)]{Szabo17ssa3}.
We will not pursue this here further, but note that this has already been used as an ingredient to obtain \cite[Theorem F]{Szabo18kp}.
\end{rem}


\section{Simple $KK$-contractible algebras}

In this section, we consider the class of simple $KK$-contractible \cstar-algebras recently classified by Elliott--Gong--Lin--Niu \cite{ElliottGongLinNiu17}.
Our main technical observation, relying on the proof of Gong--Lin's basic homotopy lemma for such \cstar-algebras, is that these \cstar-algebras have finite weak inner length as well as finite weak approximately central exponential length; recall Definitions \ref{def:weak-inner-length} and \ref{def:acel}.
This will allow us to apply Theorem \ref{thm:main-result} in order to conclude Theorem \ref{Theorem-C}.

Before we do all this, we highlight the abundance of Rokhlin flows on simple $KK$-contractible \cstar-algebras.

\begin{example} \label{ex:R-flow-on-W}
There is a Rokhlin flow on the Razak--Jacelon algebra $\CW$.
\end{example}
\begin{proof}
We may consider some Rokhlin flow on some irrational rotation \cstar-algebra  $\gamma: \IR\curvearrowright A_\theta$; see \cite[Proposition 2.5]{Kishimoto96_R}. By classification, we have $\CW\cong A_\theta\otimes\CW$, and thus $\gamma\otimes\id_\CW$ gives rise to a Rokhlin flow on $\CW$.
\end{proof}

\begin{cor}
Let $A$ be a separable, simple, stably projectionless \cstar-algebra with finite nuclear dimension. Suppose that $KK(A,A)=0$. 
Let $\beta: \IR\curvearrowright A$ be any flow.
Then there exists a Rokhlin flow $\alpha: \IR\curvearrowright A$ such that $\alpha$ and $\beta$ induce the same actions on traces.
\end{cor}
\begin{proof}
Note that $A\cong A\otimes\CW$ by classification.
If $\gamma: \IR\curvearrowright\CW$ is as in Example \ref{ex:R-flow-on-W}, then fix some isomorphism $\phi: A\to A\otimes\CW$ with $(\tau\otimes\tau^\CW) \circ \phi = \tau$ for all $\tau\in T(A)$ (this exists by \cite[Theorem 7.5]{ElliottGongLinNiu17}), and set $\alpha=\phi^{-1}\circ(\beta\otimes\gamma)\circ\phi$.
\end{proof}

In particular, any action on the traces that may come from any flow, can in fact be realized by a Rokhlin flow on a classifiable $KK$-contractible \cstar-algebra.
It remains open what kind of actions of $\IR$ on the cone of traces can be lifted to flows, but already the trace-scaling flows constructed by Kishimoto--Kumjian \cite{KishimotoKumjian96, KishimotoKumjian97} yield interesting examples.

In what follows for the rest of this section, we need to appeal to the following fundamental uniqueness theorem between classifiable $KK$-contractible \cstar-algebras, which is a known consequence of their classification theory.

\begin{theorem} \label{thm:kk-contractible-uniqueness}
Let $A$ and $B$ be two separable, simple, stably projectionless \cstar-algebra with finite nuclear dimension. 
Suppose that $KK(A,A)=KK(B,B)=0$.
Let $\phi, \psi: A\to B$ be two $*$-homomorphisms.
Then $\phi$ and $\psi$ are approximately unitarily equivalent if and only if they agree on traces, i.e., $\tau\circ\phi=\tau\circ\psi$ for every extended trace $\tau\in T(B)$.
\end{theorem}
\begin{proof}
The main classification theorem of Elliott--Gong--Lin--Niu \cite{ElliottGongLinNiu17} implies that $A$ and $B$ are simple inductive limits of so-called Razak building blocks.
This is a special case of the fact that $A$ and $B$ belong to Robert's class \cite{Robert12} of \cstar-algebras that can be expressed as inductive limits of 1-NCCW complexes with vanishing $K_1$-groups at the level of each building block.
Since such building blocks have stable rank 1, the same follows for $A$ and $B$.
Since $A$ and $B$ have vanishing $K_0$-groups, it follows directly from \cite[Theorem 1.0.1 and Proposition 6.2.3]{Robert12} that $*$-homomorphisms $A\to B$ are approximately unitarily equivalent precisely when they induce the same maps $T(B)\to T(A)$.
\end{proof}

Let us now follow the strategy outlined at the beginning of this section towards proving Theorem \ref{Theorem-C}.
We first need some technical preparations.
We remark that in many instances below, it is relevant to assume that some \cstar-algebra $A$ in question has the property called \emph{almost stable rank one}, which requires $A\subseteq\overline{\GL(\widetilde{A})}$.
It is an observation of Robert \cite{Robert15} that every projectionless Jiang--Su stable \cstar-algebras has this property, and hence it will automatically hold for the \cstar-algebras of interest in this section.
In this context recall from \cite{Winter12, Tikuisis14, CETWW, CastillejosEvington19} that any separable non-elementary nuclear simple \cstar-algebra is Jiang--Su stable if and only if it has finite nuclear dimension.

\begin{lemma} \label{lem:asr1}
Let $A$ be a \cstar-algebra. Let $h\in A$ be a positive contraction and consider $A^{(0)}=\overline{hAh}$. Suppose that $A^{(0)}\subseteq\overline{\GL(\widetilde{A^{(0)}})}$.
Then for any $\eps>0$, there exists $\delta>0$ with the following property:

For every contraction $x\in\CM(A)$ with 
\[
\|(\eins-x^*x)h\|,\ \|(\eins-xx^*)h\|\leq\delta
\]
and
\[
\dist(xh, A^{(0)}),\ \dist(hx, A^{(0)})\leq\delta,
\]
there exists a unitary $v\in\CU(\widetilde{A^{(0)}})$ such that $\|(x-v)h\|,\ \|h(x-v)\|\leq\eps$.
\end{lemma}
\begin{proof}
Let $\eps>0$ be given.
Using functional calculus, we may find $\delta>0$ with the following property, which we claim will satisfy the desired property as in the statement:

Whenever $a,b$ are positive contractions in a \cstar-algebra, then $\|(\eins-a)b\|\leq 2\delta^{1/2}$ implies $\|(\eins-a^{1/2})b\|\leq\eps/2$. We may assume $\delta\leq \eps/3$.

Now suppose $x\in\CM(A)$ satisfies 
\[
\|(\eins-x^*x)h\|,\ \|(\eins-xx^*)h\|\leq\delta
\]
and
\[
\dist(xh, A^{(0)}),\ \dist(hx, A^{(0)})\leq\delta,
\]
We may find $\eta>0$ such that
\[
\|xh-h^{\eta}xh^{\eta}h\| = \|(x-h^{\eta}x h^{\eta})h\| \leq 5\delta/4
\]
and also
\[
\|h(x-h^{\eta}xh^{\eta})\|\leq 5\delta/4.
\]
Now we have $h^{\eta}xh^{\eta}\in A^{(0)}\subseteq\overline{\GL(\widetilde{A^{(0)}})}$, so we find $z\in\GL(\widetilde{A^{(0)}})$ with $\|z\|\leq 1$ and $\|(x-z)h\|, \|h(x-z)\|\leq 3\delta/2$. 
Set $v=z|z|^{-1}=|z^*|^{-1}z\in\CU(\widetilde{A^{(0)}})$, which is a unitary.

We have
\[
\|(\eins-z^*z)^{1/2}h\|^2 = \|h(\eins-z^*z)h\| \leq \|(\eins-x^*x)h\|+\|(x-z)h\|+\|h(x-z)\|\leq 4\delta
\]
In particular
\[
\|(\eins-z^*z)h\| \leq \|(\eins-z^*z)^{1/2}h\|\leq 2\delta^{1/2}.
\]
By our choice of $\delta$, this implies
\[
\|(\eins-|z|)h\|\leq \eps/2.
\]
Analogously we have $\|(\eins-|z^*|)h\|\leq\eps/2$.
Thus
\[
\|(v-x)h\|\leq \|(x-z)h\|+\|v(\eins-|z|)h\|\leq \eps/2+3\delta/2\leq\eps
\]
and analogously $\|h(v-x)\|\leq\eps$. 
\end{proof}


\subsection{Basic homotopy lemma revisited}

The technical results in this subsection arise upon a close inspection of \cite[Lemma 14.8]{GongLin17} and its proof, as well as the classification theory \cite{ElliottGongLinNiu17} of $KK$-contractible \cstar-algebras.
For what follows, recall the notion of scale for \cstar-algebras:\ For a separable simple \cstar-algebra $A$ with $T(A)\neq\emptyset$ and a strictly positive element $e\in A$, one defines the scale function as the affine map $\Sigma_A: T(A)\to [0,\infty]$ by $\Sigma_A(\tau)=d_\tau(e):=\lim_{n\to\infty} \tau( e(e+\frac1n)^{-1} )$.
If additionally $A$ is exact, algebraically simple, and has strict comparison of positive elements, we say that $A$ has continuous scale, if $\Sigma_A$ is a continuous map; see \cite[Definition 2.1]{Lin04_2} and \cite[Theorem 5.3]{ElliottGongLinNiu17_2} for a more general treatment.
Note that any \cstar-algebra covered by the aforementioned classification theorems satisfies the assumptions of the next lemma if it has continuous scale:

\begin{lemma} \label{lem:CU-strictly-dense}
Let $A$ be an exact, non-unital, algebraically simple, and Jiang--Su stable \cstar-algebra with stable rank one. 
Then for every $\eps>0$, $\CF\fin A$ and every unitary $u\in\CU_0(\eins+A)$, one can find a unitary path $v: [0,1]\to\CU(\eins+A)$ such that
\[
v(1)=u,\quad \max_{a\in\CF}~ \max_{0\leq t\leq 1}~ \|(v(t)-u)a\|+\|a(v(t)-u)\|\leq\eps,
\]
and such that $v(0)$ is in the commutator subgroup of $\CU(\eins+A)$.
\end{lemma}
\begin{proof}
This fact appears as a reduction argument in the proof of \cite[Corollary 14.9]{GongLin17} with stronger assumptions, but we shall briefly recall it here.

Given $\eps,\CF,u$, one perturbs $u$ slightly and chooses positive contractions $0 \neq e_1, e_2\in A$ with $e_1\perp e_2$, such that $e_1$ acts $\eps/2$-approximately like a unit on $\CF$ and $u-\eins\in\overline{e_1 A e_1}$.

One writes $u=\exp(ih_1)\cdots\exp(ih_n)$ for some self-adjoint elements $h_j\in A$ and picks $h_0=h_0^*\in\overline{e_2Ae_2}$ satisfying $\tau(h_0)=\sum_{j=1}^n \tau(h_j)$ for all $\tau\in T(A)$; see \cite[Corollary 3.10]{BrownPereraToms08}. 
The path $v: [0,1]\to\CU(\eins+A)$ is then constructed via
\[
v(1-t)=u\exp(-ith_0) \quad\text{for}\quad t\in [0,1].
\]
By $e_1\perp e_2$ one sees that the exponentials $\exp(ith_0)$ approximately act like a unit on elements of $\CF$, and they all commute with $u$. In particular, this shows
\[
\max_{a\in\CF}~ \max_{0\leq t\leq 1}~ \|(v(t)-u)a\|+\|a(v(t)-u)\|\leq\eps.
\]
Moreover one has that $v(0)=u\exp(-ih_0)$ is in the commutator subgroup in $\CU(\eins+A)$ as all of its Skandalis--de la Harpe determinants vanish by construction; see \cite{delaHarpeSkandalis84, NgRobert17}.
All the assumptions on $A$ not explicitly used so far ensure that the main result of \cite{NgRobert17} is applicable.
Namely, the exactness of $A$ implies that quasi-traces on $A$ are traces \cite{Haagerup14}, and Jiang--Su stability implies that $A$ is a pure \cstar-algebra \cite{Rordam04}.
\end{proof}

\begin{rem} \label{rem:local-continuous-scale}
Due to Gong--Lin's classification, it is easy to see that all \cstar-algebras covered by \cite[Theorem 15.6]{GongLin17} admit some increasing sequence of hereditary subalgebras with continuous scale.
This is a fact we will sometimes use below.
\end{rem}

\begin{lemma} \label{lem:GL-celw}
Let $A$ be a separable, simple, stably projectionless \cstar-algebra with finite nuclear dimension. 
Suppose that the pairing map between $K_0(A)$ and $T(A)$ is trivial.
Then the weak inner length of $A$ is at most $12\pi$.
\end{lemma}
\begin{proof}
First we note that by \cite[Section 15]{GongLin17}, $A$ has generalized tracial rank at most one in the sense of \cite[Definition 3.9]{GongLin17}.

Let $\eps>0$, $\CF\fin A_{\leq 1}$ and $u\in\CU_0(\eins+A)$ be given.
Without loss of generality, assume $\CF=\CF^*$.
By perturbing $\CF$ and $u$ slightly, if necessary, we may assume by Remark \ref{rem:local-continuous-scale} without loss of generality that $\CF$ and $u$ are contained in a hereditary subalgebra with continuous scale.
So, let us simply assume that $A$ has continuous scale. 

Apply Lemma \ref{lem:CU-strictly-dense} and find a unitary path $v_2: [0,1]\to\CU(\eins+A)$ such that $v_2(1)=u$, $u_0:=v_2(0)$ is in the commutator subgroup of $\CU(\eins+A)$, and
\begin{equation} \label{eq:GL-celw-1}
\max_{a\in\CF}~ \max_{0\leq t\leq 1}~ \|a(v_2(t)-u)\|\leq\eps/8.
\end{equation}
By appealing to \cite[Theorem 4.4]{GongLin17}, we can find a unitary path $v_1: [0,1]\to\CU(\eins+A)$ satisfying $v_1(0)=\eins$, $v_1(1)=u_0$, and $\ell(v_1)\leq 6\pi+\eps/4$. 
By considering the arc-length parameterization of $v_1$, if necessary, we may assume that it is $(6\pi+\eps/4)$-Lipschitz.

Let us define $v: [0,1]\to\CU(\eins+A)$ via
\[
v(t)=v_2(t)u_0^*v_1(t).
\]
By the endpoints of $v_1$ and $v_2$ we see that $v(0)=\eins$ and $v(1)=u$.

For any $a\in\CF$ and $t_1, t_2\in [0,1]$, we then have
\[
\renewcommand\arraystretch{1.25}
\begin{array}{cl}
\multicolumn{2}{l}{ \|v(t_1)^*av(t_1)-v(t_2)^*av(t_2)\| } \\
\leq& \|(v(t_1)-v(t_2))^*a\|+\|a(v(t_1)-v(t_2))\| \\
=& \|a^*(v_2(t_1)u_0^*v_1(t_1)-v_2(t_2)u_0^*v_1(t_2))\| \\
& + \|a(v_2(t_1)u_0^*v_1(t_1)-v_2(t_2)u_0^*v_1(t_2))\| \\
\stackrel{\eqref{eq:GL-celw-1}}{\leq}& \eps/2+\|a^*uu_0^* (v_1(t_1)-v_1(t_2))\| + \|auu_0^* (v_1(t_1)-v_1(t_2))\| \\
\leq& \eps/2+(12\pi+\eps/2)|t_1-t_2| \ \leq \ \eps+12\pi|t_1-t_2|.
\end{array}
\]
Since $\eps>0$, $\CF\fin A_{\leq 1}$ and $u$ were arbitrary, this verifies the condition in Definition \ref{def:weak-inner-length} for the constant $12\pi$.
\end{proof}

\begin{lemma} \label{lem:GL-bhl}
Let $A$ be a separable, simple, stably projectionless \cstar-algebra with finite nuclear dimension. Suppose that $KK(A,A)=0$ and that $A$ has continuous scale.
Then for every $\eps>0$ and $\CF\fin A$, there exists $\delta>0$ and $\CG\fin A$ with the following property:

For every unitary $u\in\CU(\eins+A)$ in the closed commutator subgroup with
\[
\max_{a\in\CG}~ \|[a,u]\|\leq\delta,
\]
there exists a unitary path $v: [0,1]\to\CU(\eins+A)$ with
\[
v(0)=\eins,\quad v(1)=u,\quad \ell(v)\leq 2\pi+\eps,\quad \max_{a\in\CF}~ \max_{0\leq t\leq 1}~ \|[a,v(t)]\|\leq\eps.
\] 
\end{lemma}
\begin{proof}
This is precisely the statement proved in the second half of the proof of \cite[Lemma 14.8]{GongLin17} applied to the special case $A=B$ and $\phi=\id_A$; see also \cite[Corollary 14.9]{GongLin17}. 
Let us assume $\CF\fin A_{\leq 1}$ and choose a pair $(\delta,\CG)$ for $(\eps,\CF)$ according to that Lemma.\footnote{Note here that the general version of \cite[Lemma 14.8]{GongLin17} also involves a constant $\sigma>0$, a finite set $\CP\fin\underline{K}(A)$ and a subgroup $G_u\subset K_0(A)$, all depending on the pair $(\eps,\CF)$.
The constant $\sigma$ is redundant due to our assumption that $u$ is in the closed commutator subgroup, and both $\CP$ and $G_u$ are redundant because we assume $A$ to be $KK$-trivial.}

The only non-obvious part of the statement is that one may arrange $\ell(v)\leq 2\pi+\eps$, but this is true upon a close inspection of how the path $v$ is constructed in the proof. 
Namely, a close look at the construction\footnote{The actual construction of the path begins on the line after (e14.131) on page 110 in \cite{GongLin17}.} shows that there is a self-adjoint element $h\in A$ with $\|h\|\leq\pi$ such that the path $v: [0,1]\to \CU(\eins+A)$ arises (in the reverse direction from $1$ to $0$) as the path composition of
\begin{itemize}
\item a path $V_1: [0,1]\to\CU(\eins+A)$ given by $V_1(t)=u\exp(ith)$;
\item a path $V_2: [0,1]\to\CU(\eins+A)$ with $\ell(V_2)\leq\eps$;
\item a path $V_3: [0,1]\to\CU(\eins+A)$ given by 
\[
V_3(t)=V_2(1)W^*\exp(i(1-t)h)W
\]
for some unitary $W\in\CU(\eins+A)$.
\end{itemize}
In particular, the path $v$ indeed has length at most $2\pi+\eps$.
\end{proof}

\begin{lemma} \label{lem:GL-bhl-2}
Let $A$ be a separable, simple, stably projectionless \cstar-algebra with finite nuclear dimension. Suppose that $KK(A,A)=0$.
Then for every $\eps>0$ and $\CF\fin A_{\leq 1}$, there exists $\delta>0$ and $\CG\fin A_{\leq 1}$ with the following property:

For every unitary $u\in\CU(\eins+A)$ with
\[
\max_{a\in\CG}~ \|[a,u]\|\leq\delta,
\]
there exists a unitary path $v: [0,1]\to\CU(\eins+A)$ with
\[
v(0)=\eins,\quad v(1)=u,\quad \max_{a\in\CF}~\max_{0\leq t\leq 1}~ \|[a,v(t)]\|\leq\eps
\]
and
\[
\|a(v(t_1)-v(t_2))\|\leq \eps+2\pi|t_1-t_2| \quad\text{for all } t_1,t_2\in [0,1] \text{ and } a\in\CF.
\]
\end{lemma}
\begin{proof}
By Remark \ref{rem:local-continuous-scale}, we may assume that $\CF\fin A^{(0)}$ for a hereditary subalgebra $A^{(0)}\subseteq A$ with continuous scale. 

Choose $\delta_1>0$ and $\CG\fin A^{(0)}$ as in Lemma \ref{lem:GL-bhl} for $\eps/5$ and $\CF$ and $A^{(0)}$ in place of $A$. We may assume $\delta_1\leq\eps/20$. Let us also assume without loss of generality that $\CG\supseteq\CF$ consists of contractions and contains a positive contraction $e$ with $A^{(0)}=\overline{eAe}$ and $ex=xe=x$ for all $x\in\CG\setminus\set{e}$. 
 
By \cite{Robert15}, we have $A^{(0)} \subseteq\overline{\GL(\widetilde{A^{(0)}})}$, so we may choose $\delta>0$ according to Lemma \ref{lem:asr1} for $e\in A^{(0)}\subseteq A$ and $\delta_1/4$ in place of $\eps$. 
We may assume $\delta\leq\delta_1/4$.

Now let $u\in\CU(\eins+A)$ with
\[
\max_{a\in\CG}~ \|[a,u]\|\leq\delta.
\]
By our choice of $\delta$, this implies that there exists a unitary $u_0\in\CU(\widetilde{A^{(0)}})$ with 
\begin{equation} \label{eq:GL-1}
\|(u-u_0)e\|\leq\delta_1/4 \quad\text{and}\quad \|e(u-u_0)\|\leq\delta_1/4 .
\end{equation}
By our choice of $e$ this implies
\[
\max_{a\in\CG}~ \|(u-u_0)a\|\leq \delta_1/4, \quad \max_{a\in\CG}~ \|a(u-u_0)\|\leq \delta_1/4
\]
and with \eqref{eq:GL-1} therefore
\begin{equation} \label{eq:GL-2}
\max_{a\in\CG}~ \|[a,u_0]\|\leq 3\delta_1/4.
\end{equation}
Applying Lemma \ref{lem:CU-strictly-dense} to $A^{(0)}$ in place of $A$, we choose a unitary path $v_2: [0,1]\to\CU(\eins+A^{(0)})$ such that $v_2(1)=u_0$, $v_2(0)\in\CU(\eins+A^{(0)})$ is in the closed commutator subgroup, and
\begin{equation} \label{eq:GL-3}
\max_{a\in\CG}~ \max_{0\leq t\leq 1}~ \|(v_2(t)-u)a\|+\|a(v_2(t)-u)\|\leq\delta_1/4.
\end{equation}
Applying \eqref{eq:GL-2} and \eqref{eq:GL-3} at $t=0$ yields that $u_{00}=v_2(0)\in\CU(\eins+A^{(0)})$ is a unitary in the closed commutator subgroup satisfying
\begin{equation} \label{eq:GL-4}
\max_{a\in\CG}~ \|[a,u_{00}]\| \ \leq \ \max_{a\in\CG}~ \max_{0\leq t\leq 1}~ \|[a,v_2(t)]\| \leq \delta_1.
\end{equation} 
By our choice of $\delta_1$, we may now find a unitary path
$v_1: [0,1]\to\CU(\eins+A^{(0)})$ with
\begin{equation} \label{eq:GL-5}
v_1(0)=\eins,\quad v_1(1)=u_{00},\quad \ell(v_1)\leq 2\pi+\eps/5,
\end{equation}
and
\begin{equation} \label{eq:GL-6}
\max_{a\in\CF}~ \max_{0\leq t\leq 1}~ \|[a,v_1(t)]\|\leq\eps/5.
\end{equation}
By considering the arc-length parameterization of $v_1$, we may assume that it is $(2\pi+\eps/5)$-Lipschitz.

Combining \eqref{eq:GL-1} and \eqref{eq:GL-2} we see that
\[
\|(\eins-e)(\eins-uu_0^*)(\eins-e)-(\eins-u_0^*u)\|\leq 2\delta_1 \leq \eps/10.
\]
In particular, the unitary $uu_0^*\in\CU(\eins+A)$ is $\eps/10$-close to a unitary $u_1\in\CU(\eins+A^{(1)})$ for $A^{(1)}=\overline{(\eins-e)A(\eins-e)}$. 
Let us find a unitary path $v_3: [0,1]\to\CU(\eins+A^{(1)})$ of length at most $\eps/5$ with $v_3(0)=\eins$ and $v_3(1)=u_1^*uu_0^*$. 
Lastly, we know that the unitary group $\CU(\eins+A^{(1)})$ is connected, so we may find a unitary path $v_4: [0,1]\to\CU(\eins+A^{(1)})$ with $v_4(0)=\eins$ and $v_4(1)=u_1$.

We claim that the unitary path $v: [0,1]\to\CU(\eins+A)$ given by
\[
v(t)=v_4(t)v_3(t)v_2(t)u_{00}^*v_1(t)
\]
satisfies the desired property. 
Recalling the endpoints of the paths $v_1, v_2, v_3, v_4$, we see that indeed $v(0)=\eins$ and $v(1)=u$.

Now we compute for all $t\in [0,1]$ and $a\in\CF$ that
\[
\renewcommand\arraystretch{1.25}
\begin{array}{ccl}
v(t)a &=& v_4(t)v_3(t)v_2(t)u_{00}^*v_1(t)a \\
&\stackrel{\eqref{eq:GL-6},\eqref{eq:GL-5}}{=}_{\makebox[0pt]{\footnotesize \hspace{-6mm}$2\eps/5$}}& v_4(t)v_3(t)v_2(t)au_{00}^* v_1(t) \\
&\stackrel{\eqref{eq:GL-4}}{=}_{\makebox[0pt]{\footnotesize\hspace{1mm}$\eps/5$}}& v_4(t)v_3(t)a v_2(t)u_{00}^*v_1(t) \\
&=_{\makebox[0pt]{\footnotesize \hspace{7mm}$2\eps/5$}}& v_4(t) a v_3(t)v_2(t)u_{00}^*v_1(t) \\
&=& a v_4(t)v_3(t)v_2(t)u_{00}^*v_1(t) \ = \ a v(t).
\end{array}
\]
This proves
\[
\max_{a\in\CF}~ \max_{0\leq t\leq 1}~ \|[a,v(t)]\|\leq\eps.
\]
Moreover, for all $t_1, t_2\in [0,1]$ and $a\in\CF$ we have
\[
\renewcommand\arraystretch{1.25}
\begin{array}{cl}
\multicolumn{2}{l}{ \|a\big( v(t_1)-v(t_2) \big)\| }\\
=& \| a \Big( v_4(t_1)v_3(t_1)v_2(t_1)u_{00}^*v_1(t_1) - v_4(t_2)v_3(t_2)v_2(t_2)u_{00}^*v_1(t_2) \Big)\| \\
=& \| a \Big( v_3(t_1)v_2(t_1)u_{00}^*v_1(t_1) - v_3(t_2)v_2(t_2)u_{00}^*v_1(t_2) \Big)\| \\
\leq& 2\eps/5+ \| a \Big( v_2(t_1)u_{00}^*v_1(t_1) - v_2(t_2)u_{00}^*v_1(t_2) \Big)\| \\
\stackrel{\eqref{eq:GL-3}}{\leq}& 3\eps/5+\| auu_{00}^* \big( v_1(t_1)-v_1(t_2) \big) \| \\
\stackrel{\eqref{eq:GL-5}}{\leq}& 3\eps/5+(2\pi+\eps/5)|t_1-t_2| \ \leq \ \eps+2\pi|t_1-t_2|.
\end{array}
\]
This finishes the proof.
\end{proof}

\begin{theorem} \label{thm:finite-acelw}
Let $A$ be a separable, simple, stably projectionless \cstar-algebra with finite nuclear dimension. Suppose that $KK(A,A)=0$. Then the weak approximately central exponential length of $A$ is at most $2\pi$.
\end{theorem}
\begin{proof}
This follow directly from Lemma \ref{lem:GL-bhl-2}.
\end{proof}

\begin{rem}
Following the analogous line of thought as so far, it is likely that Theorem \ref{thm:finite-acelw} should generalize for all \cstar-algebras classified in \cite{GongLin17}.
This would likely involve the full force of \cite[Lemma 14.8]{GongLin17} or even a generalization of it; at present the homotopy lemma additionally requires the absorption of some infinite-dimensional UHF algebra.
Since this larger class of \cstar-algebras is far more challenging when trying to classify $*$-homomorphisms of the form $A\to\CC([0,1], A)$, this is not further pursued here.
\end{rem}


\subsection{Interlude --- Rokhlin automorphisms}

In this subsection we observe an analog of Theorem \ref{Theorem-C} for single automorphisms.
This builds on our previous observations about approximately central unitaries inside a \cstar-algebra in the $KK$-contractible classifiable class.
This can be viewed as an extension of an approach pursued recently by Nawata \cite{Nawata19}.

For the next result, recall the notion of the central sequence algebra from Definition \ref{def:central-sequence}. 

\begin{theorem}[cf.\ {\cite[Theorem 5.7]{Nawata19}}] \label{thm:nawata-property}
Let $A$ be a separable, simple, stably projectionless \cstar-algebra with finite nuclear dimension. Suppose that $KK(A,A)=0$. 
Then the central sequence algebra $F_\infty(A)$ has a connected unitary group and has exponential length at most $2\pi$.
\end{theorem}
\begin{proof}
Let $\CF_n\fin A_{\leq 1}$ be an increasing sequence of finite sets in the unit ball with dense union.
Let $\eps_n>0$ be a decreasing sequence with $\eps_n\to 0$.
For each $n$, apply Lemma \ref{lem:GL-bhl-2} to the pair $(\eps_n,\CF_n)$ and choose a pair $(\delta_n,\CG_n)$.

Let $U$ be a unitary in $F_\infty(A)$.
We identify 
\[
F_\infty(A)=\big( (\tilde{A})_\infty\cap A' \big)/\ann\big(A, (\tilde{A})_\infty \big).
\]
Since $A$ has almost stable rank 1, it follows by \cite[Proposition 4.9]{Nawata19} (one can also apply Lemma \ref{lem:asr1}) that $U$ has a representing sequence $(u_n')_n$ consisting of unitaries in $\CU(\tilde{A})$.

First, we may find $\pi$-Lipschitz paths $\lambda_n: [0,1]\to\IT\subset\CU(\tilde{A})$ such that $\lambda_n(0)=1$ and $u_n:=\lambda_n(1)^*u_n'\in\CU(\eins+A)$. 
Let $\Lambda: [0,1]\to F_\infty(A)$ be the resulting path given by $\Lambda(t)=[(\lambda_n(t))_n]+\ann(A, (\tilde{A})_\infty)$, which is well-defined.

Now as $(u_n)_n$ defines an approximately central sequence of unitaries, it is possible to find increasing numbers $1\leq k_1< k_2<\dots$ with the property that
\[
\sup_{m\geq k_n}~ \max_{a\in\CG_n}~ \|[u_m, a]\|\leq \delta_n.
\]
For every $n$ and $l\in\set{k_n,\dots,k_{n+1}-1}$, our choice of the pair $(\delta_n,\CG_n)$ allows us to find a unitary path $v_l: [0,1]\to\CU(\eins+A)$ 
with $v_l(0)=\eins$, $v_l(1)=u_l$, 
\begin{equation} \label{eq:nawata-1}
\max_{a\in\CF_n}~ \max_{0\leq t\leq 1}~ \|[a,v_l(t)]\|\leq\eps_n,
\end{equation}
and
\begin{equation} \label{eq:nawata-2}
\max_{a\in\CF_n}~ \|a(v_l(t_1)-v_l(t_2))\|\leq\eps_n+2\pi|t_1-t_2|
\end{equation}
for all $t_1,t_2\in [0,1]$.

Now consider the map $V: [0,1]\to\CU(F_\infty(A))$ given by
\[
V(t) = [(v_m(t))_m]+\ann(A, (\tilde{A})_\infty ).
\]
First of all, we clearly have $V(0)=\eins$ and $V(1)=\Lambda(1)^*U$. It is well-defined as a map since for any $a\in\bigcup_{n\in\IN}\CF_n$, we have
\[
\limsup_{m\to\infty}~ \|[v_m(t),a]\| = \inf_{n\geq 1}~ \sup_{l\geq k_n}~ \|[v_l(t),a]\| \stackrel{\eqref{eq:nawata-1}}{\leq} \inf_{n\geq 1}~ \eps_n = 0.
\]
We compute for all $t_1,t_2\in [0,1]$ that
\[
\renewcommand\arraystretch{1.25}
\begin{array}{ccl}
\|V(t_1)-V(t_2)\| &=& \dst \sup_{a\in A_{\leq 1}}~ \|a(V(t_1)-V(t_2))\| \\
&=& \dst \sup_{n_0\in\IN}~ \sup_{a\in\CF_{n_0}}~ \|a(V(t_1)-V(t_2))\| \\
&=& \dst \sup_{n_0\in\IN}~ \sup_{a\in\CF_{n_0}}~ \limsup_{m\to\infty}~ \|a(v_m(t_1)-v_m(t_2))\| \\
&=& \dst \sup_{n_0\in\IN}~ \sup_{a\in\CF_{n_0}}~ \inf_{n\geq 1}~ \sup_{l\geq k_n}~ \|a(v_l(t_1)-v_l(t_2))\| \\
&\stackrel{\eqref{eq:nawata-2}}{\leq}& \dst \sup_{n_0\in\IN}~ \sup_{a\in\CF_{n_0}}~ \inf_{n\geq 1}~ \eps_n+2\pi|t_1-t_2| \\
&=& 2\pi|t_1-t_2|.
\end{array}
\]
Thus $V$ is indeed a continuous path with length at most $2\pi$. The product $[t\mapsto \Lambda(t)V(t)]$ then yields a unitary path connecting $U$ with the unit.
In particular, this means that $U$ is homotopic to $\eins$.
Knowing this, we know by functional calculus that $U$ can actually be represented by a sequence of unitaries in $\CU(\eins+A)$ in the first place. 
Thus the unitary path $\Lambda$ becomes (a posteriori) obsolete and the above argument shows that there is a continuous path of length at most $2\pi$ connecting $U$ with the unit. This finishes the proof. 
\end{proof}

Now using Theorem \ref{thm:nawata-property} in place of \cite[Theorem 5.7]{Nawata19}, the arguments detailed in \cite[Section 7]{Nawata19} show the following result, which generalizes the unique trace case considered by Nawata.
Note that the assumption involves the Rokhlin property for non-unital \cstar-algebras in the sense of \cite[Definition 6.1]{Nawata19}.

\begin{theorem}[cf.\ {\cite[Theorems 7.3, 7.4]{Nawata19}}] \label{thm:Rokhlin-automorphisms}
Let $A$ be a separable, simple, stably projectionless \cstar-algebra with finite nuclear dimension. 
Suppose that $KK(A,A)=0$.
Let $\phi$ and $\psi$ be automorphisms on $A$ with the Rokhlin property.
Then they are cocycle conjugate if and only if their induced maps on the extended traces of $A$ are affinely conjugate. 
Moreover, any conjugacy on the level of traces lifts to an automorphism on $A$ inducing cocycle conjugacy between $\phi$ and $\psi$.
\end{theorem}
\begin{proof}
Since the ``only if'' part is trivial, we show the ``if'' part.
Suppose that $\phi$ and $\psi$ have affinely conjugate maps on traces.
This means that there exists an affine homeomorphism $\gamma: (T(A), \Sigma_A)\to (T(A), \Sigma_A)$ such that $\gamma(\tau\circ\phi)=\gamma(\tau)\circ\psi$ for all extended traces $\tau\in T(A)$.

By applying \cite[Theorem 7.5]{ElliottGongLinNiu17}, we may find an automorphism $\sigma$ on $A$ that represents $\gamma$, meaning $\gamma(\tau)=\tau\circ\sigma$ for all $\tau\in T(A)$.
Thus, if we replace $\psi$ by $\sigma\circ\psi\circ\sigma^{-1}$, our claim reduces to the case where $\phi$ and $\psi$ induce the same maps on traces.
By Theorem \ref{thm:kk-contractible-uniqueness}, it follows that $\phi$ and $\psi$ are approximately unitarily equivalent.
 
By appealing to Theorem \ref{thm:nawata-property} in place of \cite[Theorem 5.7]{Nawata19}, we can directly generalize \cite[Theorem 7.1]{Nawata19} with the well-known method of Herman--Ocneanu \cite[Theorem 1]{HermanOcneanu84}.
Namely, every unitary $u\in F_\infty(A)$ is a $\tilde{\alpha}_\infty$-coboundary, i.e., there is a unitary $v\in F_\infty(A)$ with $u=v\tilde{\alpha}_\infty(v)^*$. 
This allows us to repeat the Evans--Kishimoto intertwining argument in the proof of \cite[Theorem 7.3]{Nawata19} (the proof given there generalizes verbatim with no other modifications and is very similar to our proof of Theorem \ref{thm:main-result}) and conclude that $\phi$ and $\psi$ are cocycle conjugate via an approximately inner automorphism.
This finishes the proof.
\end{proof}


\subsection{Uniqueness for restricted coactions}

In this subsection, we shall show that for any flow $\alpha: \IR\curvearrowright A$ on a \cstar-algebra in the $KK$-contractible classifiable class, the approximate unitary equivalence class of the restricted coaction $\alpha_\co$ is determined by the induced action of $\alpha$ on traces.


\begin{lemma} \label{lem:interval-alg}
Let $A$ be a separable, simple, stably projectionless \cstar-algebra with finite nuclear dimension. Suppose that $KK(A,A)=0$. Let $\phi: A\to\CC\big( [0,1], A \big)$ be a non-degenerate $*$-homomorphism. Then $\phi\ue\eins\otimes\id_A$ if and only if $\tau\circ\phi_t=\tau$ for all $t\in [0,1]$ and all $\tau\in T(A)$.
\end{lemma}
\begin{proof}
Since the ``only if'' part is trivial, let us show the ``if'' part.

Suppose that $\tau\circ\phi_t=\tau$ for all $t\in [0,1]$ and all $\tau\in T(A)$.
By Theorem \ref{thm:kk-contractible-uniqueness}, we have that $\phi_t$ is an approximately inner endomorphism for every $t\in [0,1]$.

Now let $\CF\fin A_{\leq 1}$ and $\eps>0$. 
We need to show the approximate unitary equivalence with respect to $\CF$ and $\eps$. 
Apply Lemma \ref{lem:GL-bhl-2} and choose $\delta>0$ and $\CG\fin A_{\leq 1}$ with respect to $\CF$ and $\eps/3$.
We may assume $\delta\leq\eps$ and $\CF\subseteq\CG$.

Choose $N\in\IN$ so large that
\begin{equation} \label{eq:interval-2}
\|\phi_{t_1}(a)-\phi_{t_2}(a)\|\leq\delta/3
\end{equation}
for all $a\in\CG$ and $t_1,t_2\in [0,1]$ with $|t_1-t_2|\leq 1/N$.

As $\phi_t$ is approximately inner for each $t$, we find unitaries
\[
w_0,w_1,\dots,w_N\in\CU(\eins+A)
\]
with
\begin{equation} \label{eq:interval-3}
\| \phi_{j/N}(a)-w_jaw_j^*\|\leq\delta/3 \quad\text{for all}\quad a\in\CG \text{ and } j=0,\dots,N.
\end{equation}
So by \eqref{eq:interval-2} and \eqref{eq:interval-3}
\[
\|[w_{j}^*w_{j+1}, a]\| \leq \delta \quad\text{for all}\quad a\in\CG \text{ and } j=0,\dots,N-1.
\]
From our choice of $\delta$ according to Lemma \ref{lem:GL-bhl-2} it now follows that there exist unitary paths
\[
v^{(j)}: [0,1]\to\CU(\eins+A),\quad j=0,\dots,N-1
\]
satisfying
\begin{equation} \label{eq:interval-6}
v^{(j)}(0)=\eins,\quad v^{(j)}(1)=w_{j}^*w_{j+1},\quad \max_{a\in\CF}~\max_{0\leq t\leq 1}~ \|[a,v^{(j)}(t)]\|\leq\eps/3.
\end{equation}
We then define a unitary path $V: [0,1]\to\CU(\eins+A)$ by the recursive formula
\[
V\big( (j+t)/N \big) = w_jv^{(j)}(t) \quad\text{for all}\quad j=0,\dots,N-1 \text{ and } t\in [0,1].
\]
Then $V$ is well-defined and satisfies $V(j/N)=w_j$ for all $j=1,\dots,N$.
It also represents a unitary in $\CU\big( \eins+\CC([0,1],A) \big)$.
By combining \eqref{eq:interval-2} and \eqref{eq:interval-6}, we finally see that
\[
\|\phi(a)-V(\eins\otimes a)V^*\| = \max_{0\leq s\leq 1}~ \|\phi_s(a)-V(s)aV(s)^*\|\leq\eps
\]
for all $a\in\CF$. 
This finishes the proof.
\end{proof}

\begin{cor} \label{cor:KKc-restricted-coaction}
Let $A$ be a separable, simple, stably projectionless \cstar-algebra with finite nuclear dimension. Suppose that $KK(A,A)=0$. Let $\alpha,\beta: \IR\curvearrowright A$ be two flows. Then $\alpha_\co\ue\beta_\co$ if and only if $\alpha$ and $\beta$ induce the same actions on traces.
\end{cor}
\begin{proof}
Since the ``only if'' part is trivial, we show the ``if'' part.

Assume that $\alpha$ and $\beta$ induce the same actions on traces. Then in particular $\tau\circ\beta_t\circ\alpha_{-t}=\tau$ for every trace $\tau$ on $A$.
We consider the automorphism $\kappa$ on $\CC([0,1],A)$ given by $\kappa(f)(t)=(\beta_t\circ\alpha_{-t})(f(t))$.

By Lemma \ref{lem:interval-alg}, it follows that the $*$-homomorphism $\kappa\circ (\eins\otimes\id_A)$ is approximately unitarily equivalent to $\eins\otimes\id_A$. Since $\kappa$ acts trivially on $\CC[0,1]\subset\CM\big( \CC([0,1],A) \big)$ and one has
\[
\CC\big( [0,1], A\big) \ \subset \ \cstar\Big( \CC[0,1]\cup\eins\otimes A \Big) \ \subset \ \CM\big(\CC( [0,1], A)\big),
\]
it follows that $\kappa$ is an approximately inner automorphism. 
As $\beta_\co = \kappa\circ\alpha_\co$ by definition, this shows our claim.
\end{proof}


\subsection{Proof Theorem \ref{Theorem-C}}

We are now ready to prove our last main result:

\begin{proof}[Proof of Theorem {\em\ref{Theorem-C}}]
Let $A$ be a separable, simple, stably projectionless, $KK$-contractible \cstar-algebra with finite nuclear dimension.
Let $\alpha, \beta: \IR\curvearrowright A$ be two Rokhlin flows, and suppose that there exists an affine homeomorphism $\gamma: (T(A), \Sigma_A) \to (T(A), \Sigma_A)$ such that $\gamma(\tau\circ\alpha_t)=\gamma(\tau)\circ\beta_t$ for all extended traces $\tau\in T(A)$ and $t\in\IR$.
We have to show that $\alpha$ and $\beta$ are cocycle conjugate via an automorphism representing $\gamma$.

By applying \cite[Theorem 7.5]{ElliottGongLinNiu17}, we may find an automorphism $\sigma$ on $A$ that represents $\gamma$, meaning $\gamma(\tau)=\tau\circ\sigma$ for all $\tau\in T(A)$.
Thus, if we replace $\beta$ by $\sigma\circ\beta\circ\sigma^{-1}$, our claim reduces to the case where $\alpha$ and $\beta$ induce the same actions on traces.
With this assumption, it is now enough to show that $\alpha$ and $\beta$ are cocycle conjugate via an approximately inner automorphism.

By Corollary \ref{cor:KKc-restricted-coaction}, the restricted coactions $\alpha_\co$ and $\beta_\co$ are approximately unitarily equivalent.
By Lemma \ref{lem:GL-celw}, $A$ has finite weak inner length.
By Theorem \ref{thm:finite-acelw}, $A$ has finite weak approximately central exponential length.
So by Theorem \ref{thm:main-result}, it follows that $\alpha$ and $\beta$ are indeed (strongly) cocycle conjugate via an approximately inner automorphism. This finishes the proof.
\end{proof}


\subsection{Concluding remarks}

As it was alluded to in the introduction, there are certain classes of flows where one might expect the Rokhlin property to hold automatically.
Let us discuss this problem in the context of stably projectionless \cstar-algebras.

\begin{question}[cf.\ \cite{KishimotoKumjian96, KishimotoKumjian97}]
\label{q:rokhlin-automatic}
Let $A$ be a separable, simple, stably projectionless \cstar-algebra with finite nuclear dimension and satisfying the UCT. Let $\alpha: \IR\curvearrowright A$ be a flow, and suppose that it is trace-scaling, i.e., there exists $\lambda\neq 0$ such that $\tau\circ\alpha_{t}=e^{\lambda t}\cdot\tau$ for all extended traces $\tau$ on $A$. Does $\alpha$ satisfy the Rokhlin property?
\end{question}

Recall that the class of \cstar-algebras covered by this question has recently been classified by Gong--Lin \cite{GongLin17}.\footnote{Note that any flow acts trivially on the $K$-theory of $A$, so the existence of a trace-scaling flow on $A$ implies that the pairing between $K_0(A)$ and $T(A)$ must be trivial.} 
More specifically, let us ask:

\begin{question} \label{q:trace-scaling-WK}
Does every trace-scaling flow on $\CW\otimes\CK$ have the Rokhlin property?
In other words, by Theorem \ref{Theorem-C}, does the scaling factor uniquely determine such a flow up to cocycle conjugacy?
\end{question}

In \cite{Nawata19}, the analogous question for single automorphisms has been settled, but the question is much more difficult for flows.

\begin{rem}
It might not be immediately clear why Question \ref{q:trace-scaling-WK} has a special appeal.
For comparison, let us translate this into an analogous question about $\mathrm{W}^*$-algebras:

If we close up $\CW\otimes\CK$ in the image of its GNS representation associated to the unique extended trace, then we obtain the hyperfinite II${}_\infty$-factor.
In particular, the analogous $\mathrm{W}^*$-question is whether all trace-scaling flows on the hyperfinite II${}_\infty$-factor have the Rokhlin property.

Now this is known to be true. For the reader's convenience, let us briefly discuss why. 
It is straightforward that the $\mathrm{W}^*$-crossed product of such a flow is an injective III${}_1$-factor, and that the dual flow will have a unique (unbounded) KMS weight for the inverse temperature given by the scaling factor.
By Haagerup's uniqueness theorem \cite{Haagerup87}, there is only one injective III${}_1$-factor.
Hence by a result of Connes \cite{Connes73}, the dual flow of any such trace-scaling flow is uniquely determined by the scaling factor up to cocycle conjugacy.
By passing to the double dual flow, one can thus show that the scaling factor uniquely determines the flow up to stable conjugacy, and by a known trick \cite[Theorem 6.1(2)]{MasudaTomatsu16}, in fact up to cocycle conjugacy.
Since there exists some Rokhlin flow on the hyperfinite II${}_1$-factor, it is then clear that such flows must have the Rokhlin property.

Upon analyzing the above line of argument, one might point out that the crucial result used above is Haagerup's uniqueness theorem for the injective III${}_1$-factor.
Let us, for the sake of toying around, pretend that we do not know about Haagerup's theorem.

Suppose then that $M^{(1)}$ and $M^{(2)}$ are injective factors of type III${}_1$.
For $i=1,2$ there exists, by definition, a flow $\alpha^{(i)}: \IR\curvearrowright M^{(i)}$ with a unique normal KMS state for the temperature $1$ such that the crossed product is a factor.
In fact, then $M^{(i)}\rtimes_{\alpha^{(i)}}\IR$ must be the hyperfinite II${}_\infty$-factor by \cite{Connes76}, and the dual flow $\hat{\alpha}^{(i)}$ scales the trace by $\tau\circ\hat{\alpha}^{(i)}_{t}=e^t\cdot\tau$.
If we know that these flows satisfy the Rokhlin property, then Masuda--Tomatsu's classification theorem \cite[Theorem 1]{MasudaTomatsu16} gives us a cocycle conjugacy between $\hat{\alpha}^{(1)}$ and $\hat{\alpha}^{(2)}$. By Takai duality, this would imply that
\[
M^{(1)} \cong M^{(1)}\bar{\otimes}\CB(\ell^2(\IN)) \cong (M^{(i)}\rtimes_{\alpha^{(1)}}\IR)\rtimes_{\hat{\alpha}^{(1)}}\IR
\]
is isomorphic to
\[
M^{(2)} \cong M^{(2)}\bar{\otimes}\CB(\ell^2(\IN)) \cong (M^{(2)}\rtimes_{\alpha^{(2)}}\IR)\rtimes_{\hat{\alpha}^{(2)}}\IR,
\]
and would hence recover Haagerup's theorem.

In other words, the positive answer to the $\mathrm{W}^*$-variant of Question \ref{q:trace-scaling-WK} is \emph{formally equivalent} to Haagerup's theorem.
Via this line of thought, we may interpret a potential positive answer to Question \ref{q:trace-scaling-WK} as a \cstar-algebraic variant of Haagerup's theorem.
Moreover, by applying the dual flow construction similarly as above, a positive answer to either Question \ref{q:trace-scaling-WK} or \ref{q:rokhlin-automatic} would also open the door towards the classification of naturally occurring classes of flows that do not necessarily have the Rokhlin property.
\end{rem}


\bibliographystyle{gabor}
\bibliography{master}

\end{document}